\def\todaysdate{20\textsuperscript{th} April 2021}
\documentclass[reqno,a4paper]{article}
\usepackage{etex}
\usepackage[T1]{fontenc}
\usepackage[utf8]{inputenc}
\usepackage{lmodern}

\usepackage[style=alphabetic,firstinits=true,backref,backend=biber,isbn=false,url=false,maxbibnames=99]{biblatex}

\renewbibmacro{in:}{}
\newbibmacro{string+doi}[1]{\iffieldundef{doi}{\iffieldundef{url}{#1}{\href{\thefield{url}}{#1}}}{\href{http://dx.doi.org/\thefield{doi}}{#1}}}
\DeclareFieldFormat*{title}{\usebibmacro{string+doi}{\emph{#1}}}
\DefineBibliographyStrings{english}{%
  backrefpage = {$\uparrow$},
  backrefpages = {$\uparrow$},
}

\usepackage{amssymb,amsmath,amsthm,amscd}
\usepackage{thmtools}
\usepackage[dvipsnames]{xcolor}
\usepackage{units}
\usepackage{graphicx}
\usepackage[all]{xy}
\usepackage{tikz}
\usetikzlibrary{arrows,calc,positioning,decorations.pathreplacing,shapes.misc,patterns}
\usepackage{relsize}
\usepackage{mhsetup}
\usepackage{mathtools}
\usepackage{stmaryrd}
\usepackage{tocloft}
\usepackage{paralist} 
\usepackage[left=3cm,right=3cm,top=3cm,bottom=3cm]{geometry} 
\usepackage[bottom]{footmisc}
\usepackage[colorlinks,citecolor=blue,linkcolor=blue,urlcolor=blue,filecolor=blue,breaklinks]{hyperref}
\usepackage{footnotebackref}
\usepackage[format=plain,indention=1em,labelsep=quad,labelfont={up},textfont={footnotesize},margin=2em]{caption}
\usepackage[mathscr]{euscript}
\usepackage{accents}
\usepackage{marginnote}
\usepackage{pinlabel}

\definecolor{lightblue}{rgb}{0.8,0.8,1}

\cftpagenumbersoff{part}

\numberwithin{equation}{section}
\numberwithin{figure}{section}

\definecolor{vdarkred}{rgb}{0.7,0,0}
\declaretheoremstyle[
  spaceabove=\topsep,
  spacebelow=\topsep,
  headpunct=,
  numbered=no,
  postheadspace=1ex,
  headfont=\color{vdarkred}\normalfont\bfseries,
  bodyfont=\normalfont\itshape,
]{colored}
\declaretheoremstyle[
  spaceabove=\topsep,
  spacebelow=\topsep,
  headpunct=,
  numbered=no,
  postheadspace=1ex,
  headfont=\normalfont\bfseries,
  bodyfont=\normalfont\itshape,
]{italic}
\declaretheoremstyle[
  spaceabove=\topsep,
  spacebelow=\topsep,
  headpunct=,
  numbered=no,
  postheadspace=1ex,
  headfont=\normalfont\bfseries,
  bodyfont=\normalfont\upshape,
]{upright}

\declaretheorem[style=italic,name=Theorem,numbered=yes,numberwithin=section]{thm}
\declaretheorem[style=italic,name=Lemma,numbered=yes,numberlike=thm]{lem}
\declaretheorem[style=italic,name=Proposition,numbered=yes,numberlike=thm]{prop}
\declaretheorem[style=italic,name=Corollary,numbered=yes,numberlike=thm]{coro}

\declaretheorem[style=italic,name=Theorem,numbered=yes,numberwithin=section]{athm}

\declaretheorem[style=upright,name=Definition,numbered=yes,numberlike=thm]{defn}
\declaretheorem[style=upright,name=Summary,numbered=yes,numberlike=thm]{summary}
\declaretheorem[style=upright,name=Remark,numbered=yes,numberlike=thm]{rmk}
\declaretheorem[style=upright,name=Example,numbered=yes,numberlike=thm]{eg}
\declaretheorem[style=upright,name=Notation,numbered=yes,numberlike=thm]{notation}
\declaretheorem[style=upright,name=Construction,numbered=yes,numberlike=thm]{construction}

\declaretheorem[style=upright,name=Remark and Notation,numbered=yes,numberlike=thm]{rmknotation}

\makeatletter
\renewcommand*{\@seccntformat}[1]{\upshape\csname the#1\endcsname.\hspace{1ex}}
\renewcommand*{\part}{\@startsection{part}{0}{\z@}%
	{2.5ex \@plus 1ex \@minus 0.2ex}%
	{1.5ex \@plus 0.2ex}%
	{\normalfont\large\bfseries\centering}}
\renewcommand*{\section}{\@startsection{section}{1}{\z@}%
	{2.5ex \@plus 1ex \@minus 0.2ex}%
	{1.5ex \@plus 0.2ex}%
	{\normalfont\large\bfseries}}
\renewcommand*{\subsection}{\@startsection{subsection}{2}{\z@}%
	{2.5ex \@plus 1ex \@minus 0.2ex}%
	{-1.5ex \@plus -0.2ex}%
	{\normalfont\normalsize\bfseries}}
\renewcommand*{\subsubsection}{\@startsection{subsubsection}{3}{\z@}%
	{2.5ex \@plus 1ex \@minus 0.2ex}%
	{-1.5ex \@plus -0.2ex}%
	{\normalfont\normalsize\bfseries}}
\renewcommand*{\paragraph}{\@startsection{paragraph}{4}{\z@}%
	{2.5ex \@plus 1ex \@minus 0.2ex}%
	{-1.5ex \@plus -0.2ex}%
	{\normalfont\normalsize\bfseries}}
\renewcommand*{\subparagraph}{\@startsection{subparagraph}{5}{\z@}%
	{2.5ex \@plus 1ex \@minus 0.2ex}%
	{-1.5ex \@plus -0.2ex}%
	{\normalfont\normalsize\slshape}}
\makeatother

\newcommand{\cf}{\textit{cf}.\ }
\newcommand{\Cf}{\textit{Cf}.\ }
\newcommand{\longtwoheadleftarrow}{\leftarrow\mathrel{\mkern-14mu}\longleftarrow}
\newcommand{\longtwoheadrightarrow}{\longrightarrow\mathrel{\mkern-14mu}\rightarrow}
\newcommand{\longhookrightarrow}{\ensuremath{\lhook\joinrel\longrightarrow}}
\newcommand{\hotop}{\ensuremath{\mathrm{Ho}(\mathrm{Top})}}
\newcommand{\too}{\ensuremath{\longrightarrow}}

\newenvironment{itemizeb}%
{\begin{compactitem}

}%
{\end{compactitem}}

\newcommand{\cA}{\mathcal{A}}
\newcommand{\cB}{\mathcal{B}}
\newcommand{\cC}{\mathcal{C}}
\newcommand{\cD}{\mathcal{D}}

\newcommand{\cM}{\mathcal{M}}
\newcommand{\cN}{\mathcal{N}}
\newcommand{\cO}{\mathcal{O}}

\newcommand{\cS}{\mathcal{S}}
\newcommand{\cT}{\mathcal{T}}

\newcommand{\bF}{\mathbb{F}}

\newcommand{\bK}{\mathbb{K}}

\newcommand{\bN}{\mathbb{N}}

\newcommand{\bQ}{\mathbb{Q}}
\newcommand{\bR}{\mathbb{R}}

\newcommand{\bZ}{\mathbb{Z}}

\newcommand{\eswc}{\ensuremath{\mathsf{ESC}}}
\newcommand{\vswc}{\ensuremath{\mathsf{VSC}}}
\newcommand{\swc}{\ensuremath{\mathsf{SC}}}
\newcommand{\cswc}{\ensuremath{\mathsf{CSC}}}
\newcommand{\lswc}{\ensuremath{\mathsf{LSC}}}
\newcommand{\xswc}{\ensuremath{\mathsf{XSC}}}
\newcommand{\disc}{\ensuremath{\mathsf{D}}}
\newcommand{\sa}{\ensuremath{\mathsf{a}}}
\newcommand{\sm}{\ensuremath{\mathsf{m}}}

\newcommand{\cmap}[4]{\ensuremath{\mathrm{CMap}_{#1}^{#2}(#3;#4)}}
\newcommand{\cmapd}[4]{\ensuremath{\dot{\mathrm{C}}\mathrm{Map}_{#1}^{#2}(#3;#4)}}
\newcommand{\cgamma}[4]{\ensuremath{\mathrm{C}\Gamma_{#1}^{#2}(#3;#4)}}
\newcommand{\cgammad}[4]{\ensuremath{\dot{\mathrm{C}}\Gamma_{#1}^{#2}(#3;#4)}}

\renewcommand{\geq}{\geqslant}
\renewcommand{\leq}{\leqslant}
\renewcommand{\footnoterule}{%
  \kern -3pt
  \hrule width \textwidth height 0.4pt
  \kern 2.6pt
}

\begin{document}
\title{\Large\bfseries Configuration-mapping spaces and homology stability}
\author{\normalsize Martin Palmer and Ulrike Tillmann}
\date{\small\todaysdate}
\maketitle
{
\makeatletter
\renewcommand*{\BHFN@OldMakefntext}{}
\makeatother
\footnotetext{This work was initiated at the INI programme \emph{Homotopy Harnessing Higher Structures} (EPSRC grant number EP/R014604/1). The first author was partially supported by a grant of the Romanian Ministry of Education and Research, CNCS - UEFISCDI, project number PN-III-P4-ID-PCE-2020-2798, within PNCDI III. The second author would also like to thank MSRI (NSF Grant 1440140).}
}
\begin{abstract}
For a given bundle $\xi \colon E \to M$ over a manifold, \emph{configuration-section spaces} parametrise finite subsets $z \subseteq M$ equipped with a section of $\xi$ defined on $M \smallsetminus z$, with prescribed ``charge'' in a neighbourhood of the points $z$. These spaces may be interpreted physically as spaces of fields that are permitted to be singular at finitely many points, with constrained behaviour near the singularities. As a special case, they include the Hurwitz spaces, which parametrise branched covering spaces of the $2$-disc with specified deck transformation group.

We prove that configuration-section spaces are homologically stable (with $\bZ$ coefficients) whenever the underlying manifold $M$ is connected and has non-empty boundary and the charge is ``small'' in a certain sense. This has a partial intersection with the work on Hurwitz spaces of Ellenberg, Venkatesh and Westerland.
\end{abstract}

\section{Introduction}

Configuration spaces of points in manifolds have been intensively studied in topology and geometry, and may be interpreted physically as a model for particles moving in a background space. In \emph{labelled} configuration spaces, each particle is equipped with an additional parameter, taking values in a fixed space $X$ or, more generally, in a bundle over the underlying manifold. An equally physically relevant setting corresponds to equipping not the particles, but instead their \emph{complement}, with a map to $X$ or a section of a bundle over the underlying manifold. For maps to a fixed space $X$, these are the \emph{configuration-mapping spaces}, discussed in \cite{EllenbergVenkateshWesterland2016HomologicalstabilityHurwitz, EllenbergVenkateshWesterland2012HomologicalstabilityHurwitz}. Since these spaces are intended to model particles moving in physical fields, which typically take values in a bundle over the underlying manifold, one is naturally led, as suggested by Segal in \cite{Segal2014}, to consider more generally \emph{configuration-section spaces}, which we introduce in \S\ref{s:cgamma}.

Roughly, configuration-mapping spaces are defined as follows. Given a $d$-dimensional manifold $M$,\footnote{We assume in this paragraph that $M$ is oriented, to avoid a small subtlety arising for non-orientable manifolds; see \S\ref{s:cmap} for the general case.} a space $X$ and a set $c \subseteq [S^{d-1},X]$ of unbased homotopy classes of maps from $S^{d-1}$ to $X$, a point of the $k$-th \emph{configuration-mapping space}
\[
\cmap{k}{c}{M}{X}
\]
consists of a subset $z \subset \mathring{M}$ of the interior of $M$ of cardinality $k$ and a continuous map $f \colon M \smallsetminus z \to X$. Moreover, we require that the restriction of $f$ to a small punctured neighbourhood of each point of $z$ lies in one of the homotopy classes in $c$. The homotopy class of the germ of $f$ near $p \in z$ may be thought of as the ``\emph{charge}'' (or ``\emph{monodromy}'' or ``\emph{singularity type}'') of the particle $p$, and $c$ is therefore the set of allowed charges of the particles in the system being modelled. For a subset $D \subseteq \partial M$ (usually a point or a disc) and a basepoint $* \in X$, one may also impose the boundary condition that $f(D) = \{*\}$. See \S\ref{s:cmap} for precise definitions, including how to topologise this set.

Configuration-section spaces are defined in an analogous way, for a bundle $\xi \colon E \to M$ and a basepoint $* \in \partial M$. In this case the notion of ``charge'' is a little more subtle to define, and takes values in a certain quotient (depending on the bundle $\xi$) of the set $[S^{d-1},X]$, where $X$ denotes the fibre of $\xi$ over the basepoint. (Thus, in this setting, we fix a subset $c$ of this quotient.) A point of the $k$-th \emph{configuration-section space}
\[
\cgamma{k}{c}{M}{\xi}
\]
consists of a subset $z \subset \mathring{M}$ of cardinality $k$ and a continuous section of $\xi$ defined over $M \smallsetminus z$, with a condition (determined by $c$) on the local behaviour of the section near each point $p \in z$. See \S\ref{s:cgamma} for precise definitions, including an explanation of the ``charge'' of a configuration-section when the bundle $\xi$ is non-trivial.

\paragraph{Examples of configuration-mapping spaces.}
In the special case when $d=2$ and $X=BG$, the ``set of allowed charges'' $c$ is a collection of conjugacy classes of $G$, since $[S^1,BG] = \mathrm{Conj}(G)$. Up to homotopy equivalence, the space $\cmap{k}{c}{M}{BG}$ parametrises branched coverings of the surface $M$ with deck transformation group $G$ and with monodromy around the branch points lying in $c$. In particular, when $M=D^2$, these are the \emph{Hurwitz spaces} associated to the pair $(G,c)$; see Remark \ref{rmk:Hurwitz1} and Example \ref{eg:Hurwitz2} for more details. The main result of \cite{EllenbergVenkateshWesterland2016HomologicalstabilityHurwitz} is a rational homological stability theorem for Hurwitz spaces when $G$ is a finite group and $c$ is a single conjugacy class that generates $G$ and is ``non-splitting'' (\cf Example \ref{rmk:EVW-comparison} below).

\paragraph{Examples of configuration-section spaces.}
Examples of configuration-section spaces that are typically not configuration-mapping spaces (in other words, where the bundle $\xi$ is typically non-trivial) include moduli spaces of configurations of points whose complement is equipped with:
\begin{itemizeb}
\item a tangential structure (orientation, spin, etc.);
\item a flat connection on a given principal $G$-bundle over $M$;
\item a non-vanishing vector field.
\end{itemizeb}
These examples are discussed in more detail in Examples \ref{eg:tangential}, \ref{eg:flat-connection} and \ref{eg:vector-fields}.

\paragraph{Homological stability.}
Our main result is that configuration-section spaces are \emph{homologically stable}, subject to a condition on the ``set of allowed charges'' $c$. (This condition will in fact force $c$ to have size $1$.) As mentioned above, $c$ is a subset of a certain quotient (depending on the bundle $\xi \colon E \to M$) of the set $[S^{d-1},X]$, where $X$ is the fibre of $\xi$. By taking pre-images, we will view it simply as a subset of $[S^{d-1},X]$, which we also denote by $c$ by abuse of notation.

\begin{athm}[{Theorem \ref{thm-hom-stab}}]
\label{tmain}
Let $M$ be a connected manifold of dimension $d \geq 2$ with basepoint $* \in \partial M$ and let $\xi \colon E \to M$ be a fibre bundle whose fibre over $*$ we denote by $X$. Assume that
\begin{equation}
\label{eq:charge-condition}
\text{the preimage of } c \text{ under } [S^{d-1},X] \longtwoheadleftarrow \pi_{d-1}(X) \text{ is a single element, }
\end{equation}
so $c$ corresponds to a fixed point of the $\pi_1(X)$-action on $\pi_{d-1}(X)$. Then the stabilisation maps
\begin{equation}
\label{eq:stabilisation-map}
\cgamma{k}{c,*}{M}{\xi} \longrightarrow \cgamma{k+1}{c,*}{M}{\xi}
\end{equation}
induce isomorphisms on $H_i(-;\bZ)$ in the range $k \geq 2i+4$ and surjections in the range $k \geq 2i+2$. With field coefficients, these ranges may be improved to $k \geq 2i+2$ and $k \geq 2i$ respectively.
\end{athm}

The (slope of the) stability range can be improved for certain manifolds when 2 is invertible in the coefficients (Theorem \ref{thm-hom-stab-improved}).

Theorem \ref{tmain} is an analogue of classical homological stability (see \cite{Segal1973Configurationspacesand, McDuff1975Configurationspacesof, Segal1979topologyofspaces}) for ordinary (unordered) configuration spaces on a connected, open manifold. Although homological stability results have also been proven for configuration spaces defined by \emph{local} data more generally, such as labelled configuration spaces \cite{Randal-Williams2013Homologicalstabilityunordered} and bounded symmetric powers \cite{KupersMillerTran2016}, here we are dealing with \emph{global} data, where there are fewer such results in the literature (one example being oriented configuration spaces \cite{Palmer2013}).

We remark that, in the ``local-to-global'' principle for homological stability of \cite{KupersMiller2018}, the terms ``local'' and ``global'' are used a little differently: they show that, for a framed, $\bN$-graded $E_d$-algebra $A$, homological stability holds for the topological chiral homology $\int_{\bR^d} A$ (``locally'') if and only if it holds for $\int_M A$ for all connected, open $d$-manifolds $M$ (``globally''). In particular, the framed $E_d$-algebra $A$ under consideration may be built out of configurations in $\bR^d$ equipped with either local or global data (in our sense).

\begin{eg}\label{eg:spin-etc}
Note that, if $\pi_{d-1}(X) = \{*\}$, then our result applies directly to the full configuration-section space $\cgamma{k}{*}{M}{\xi}$ without any restriction on the charges. This is the case in several of the examples mentioned above. In particular, Theorem \ref{tmain} applies directly to configuration-section spaces where the complement is equipped with a spin structure (if $\mathrm{dim}(M) \geq 3$), string structure (if $\mathrm{dim}(M) \geq 5$), etc. See Example \ref{eg:tangential} for more details. It also applies directly to the example where the complement of a configuration is equipped with a flat connection for a given principal $G$-bundle over $M$, if $\mathrm{dim}(M) = 3$ and $G$ is a Lie group. See Example \ref{eg:flat-connection} for more details.
\end{eg}

\begin{eg}
In the case of non-vanishing vector fields on the complement, we have $\pi_{d-1}(X) = \pi_{d-1}(\bR^d \smallsetminus \{0\}) \cong \bZ$, so Theorem \ref{tmain} does not apply unless we restrict the allowed charge of the field at the configuration points. Geometrically, the charge in this case is simply the winding number of the vector field, so Theorem \ref{tmain} applies when we restrict to non-vanishing vector fields with fixed winding number around each configuration point.
\end{eg}

\begin{eg}[\emph{Relation with the result of \textup{\cite{EllenbergVenkateshWesterland2016HomologicalstabilityHurwitz}}.}]
\label{rmk:EVW-comparison}
In the case of Hurwitz spaces, our assumption is that $c \subseteq \mathrm{Conj}(G)$ is a single conjugacy class of size $1$ (corresponding to an element of the centre of $G$). The result of \cite{EllenbergVenkateshWesterland2016HomologicalstabilityHurwitz}, by contrast, allows larger conjugacy classes, although it is specific to the setting of Hurwitz spaces $\mathrm{Hur}_{G,k}^c \simeq \cmap{k}{c,*}{D^2}{BG}$. These two results are therefore somewhat orthogonal in terms of generality -- and indeed our methods are entirely different to those of Ellenberg, Venkatesh and Westerland.

In more detail, the main result of \cite{EllenbergVenkateshWesterland2016HomologicalstabilityHurwitz} says that, in a stable range, the rational homology groups $H_i(\mathrm{Hur}_{G,k}^c;\bQ)$ are periodic with respect to $k$ (with a period depending on the pair $(G,c)$), as long as $G$ is finite, the conjugacy class $c$ generates $G$ and $c$ is ``non-splitting'' in the sense that $H \cap c$ is either empty or a single conjugacy class for every subgroup $H \leq G$. On the other hand, specialising Theorem \ref{tmain} to the case $(M=D^2,X=BG)$, we show that the integral homology groups $H_i(\mathrm{Hur}_{G,k}^c;\bZ)$ are constant ($1$-periodic) in a stable range, as long as $\lvert c \rvert = 1$. The intersection of our results is the case where $G$ is finite cyclic and $c = \{g\}$ for a generator $g \in G$.
\end{eg}

\paragraph{Split-injectivity.}

For classical configuration spaces, and configuration spaces where the points (rather than the complement) are labelled, the analogous stabilisation maps induce \emph{split-injections} on homology in all degrees. This follows from an easy argument (see \cite[page 103]{McDuff1975Configurationspacesof}; also \cite[\S 4]{ManthorpeTillmann2014Tubularconfigurationsequivariant} for a diffeomorphism equivariant stable splitting) that depends on the existence of multi-valued maps $C_k(M) \dashrightarrow C_{k-r}(M)$ that ``forget $r$ points from the configuration in all $\binom{k}{r}$ possible ways''.

However, such maps do not exist for configuration-section spaces, since in this setting one cannot simply forget a point; one must also extend the section on the complement of the configuration to the forgotten point, which is in general impossible. Indeed, we expect that split-injectivity on homology does \emph{not} hold for configuration-section spaces in general. Nevertheless, we do have a partial positive result in this direction for configuration-mapping spaces; under additional hypotheses on the underlying manifold $M$, the stabilisation maps \eqref{eq:stabilisation-map} induce split-injections on the $E^2$ pages of the map of associated Serre spectral sequences. See Theorem \ref{thm:split-injectivity} for the precise statement. This relies on a detailed study of the monodromy action associated to the fibration \eqref{eq:forgetful-map}, which is carried out in the companion paper \cite{PalmerTillmann2020Pointpushingactions}.

\paragraph{Outline.}

The paper is organised as follows. Sections \ref{s:cmap} and \ref{s:cgamma} contain the precise definitions of configuration-mapping spaces (recalled from \cite{EllenbergVenkateshWesterland2012HomologicalstabilityHurwitz}) and configuration-section spaces (a natural generalisation that we introduce); several different classes of examples are also discussed in section \ref{s:cgamma}. The structure on configuration-section spaces that we need, including the stabilisation maps, is encapsulated in the statement that they form an \emph{$E_0$-module over an $E_{d-1}$-algebra}.
In section \ref{s:em-modules-over-en-algebras}, we first explain precisely what this means and then define this structure on configuration-section spaces in appropriate models (Proposition \ref{p:e0-modules}). Along the way, we take care to recall and relate several different flavours of Swiss cheese operads. In section \ref{s:monodromy} we then define the (up-to-homotopy) monodromy action of the fundamental group $\pi_1(C_k(\mathring{M}))$ coming from the forgetful map
\begin{equation}
\label{eq:forgetful-map}
\cgamma{k}{c,*}{M}{\xi} \too C_k(\mathring{M}).
\end{equation}
In section \ref{s:extension} we show that these actions fit together to define a \emph{monodromy functor}
\begin{equation}
\label{eq:monodromy-functor}
\cC(M) \too \hotop,
\end{equation}
where $\cC(M)$ is a certain \emph{braid category} on $M$ recalled in section \ref{s:braid-categories}. This allows us to use recent twisted homological stability results by Krannich~\cite{Krannich2019Homologicalstabilitytopological} in their full strength to prove Theorem \ref{tmain} in section \ref{s:polynomiality}: By a spectral sequence argument, it suffices to prove \emph{twisted} homological stability for ordinary configuration spaces with coefficients in the composition of the monodromy functor \eqref{eq:monodromy-functor} with homology in any fixed degree $q$; we prove in Proposition \ref{prop:polynomiality} that this composite functor is \emph{polynomial} of degree $q$.
In section \ref{s:extension}, relying on \cite{PalmerTillmann2020Pointpushingactions}, we discuss when a factorisation of the monodromy functor through an enhanced braid category $\cB_\sharp(M)$ is possible. This allows us to improve the homology stability range for coefficients in which $2$ is invertible, see Theorem \ref{thm-hom-stab-improved}.  It also allows us to prove the partial split-injectivity result mentioned above (Theorem \ref{thm:split-injectivity}).

\paragraph{Acknowledgements.}
The authors would like to thank Andrea Bianchi and the referee for their helpful comments on an earlier version of this paper.

\tableofcontents

\section{Configuration-mapping spaces}\label{s:cmap}

We begin by recalling the definition of \emph{configuration-mapping spaces} from \cite{EllenbergVenkateshWesterland2012HomologicalstabilityHurwitz}, which generalises the classical notion of \emph{Hurwitz spaces} \cite{Clebsch1872ZurTheorieRiemannschen,Hurwitz1891UeberRiemannscheFlaechen}. In fact, we slightly extend the definition of \cite{EllenbergVenkateshWesterland2012HomologicalstabilityHurwitz} by considering also \emph{non-orientable} manifolds (we will generalise this further to \emph{configuration-section spaces} in the next section).

Let $M$ be a smooth, compact, connected manifold with non-empty boundary of dimension $d\geq 2$ and let $X$ be a space. Also, let
\[
c \subseteq [S^{d-1},X]
\]
be a non-empty subset of the set of (unbased) homotopy classes of maps $S^{d-1} \to X$. There is a $\bZ/2$-action on the set $[S^{d-1},X]$ given by precomposition by a reflection of the sphere, and, in the case when $M$ is \emph{non-orientable}, we assume that $c$ is a union of orbits of this action (see Remark \ref{rmk:fixed-points} for why). If $M$ is \emph{orientable}, there is no condition on $c$.

\begin{rmk}
If $X$ is path-connected, the set $[S^{d-1},X]$ may be identified with the set of orbits $\pi_{d-1}(X)/\pi_1(X)$ for any choice of basepoint of $X$. When $d=2$ this is the set of conjugacy classes of $\pi_1(X)$.
\end{rmk}

\begin{defn}[\emph{Configuration-mapping spaces, I.}]
\label{d:cmspace}
For a positive integer $k$, the underlying set of the \emph{configuration-mapping space}
\[
\cmap{k}{c}{M}{X}
\]
is the set of pairs $(z,f)$, where $z$ is a subset of the interior $\mathring{M}$ of $M$ of cardinality $k$ and $f$ is a continuous map $M \smallsetminus z \to X$ with the following property: for any embedding $e \colon D^d \hookrightarrow M$ where $e(D^d) \cap z$ consists of a single point in the interior of $e(D^d)$, we have
\begin{equation}\label{eq:singularity-condition}
[f \circ e \circ i] \in c ,
\end{equation}
where $i$ denotes the inclusion $S^{d-1} \hookrightarrow D^d$. If $M$ is orientable, we fix a choice of orientation and only consider \emph{orientation-preserving} embeddings $e$ in the above condition.
\end{defn}

\begin{rmk}
\label{rmk:fixed-points}
The reason for assuming that $c$ is a union of orbits of the involution on $[S^{d-1},X]$, in the case when $M$ is non-orientable, is the following. In the definition of configuration-mapping spaces above, the set $c$ of homotopy classes of maps $S^{d-1} \to X$ is to be thought of as the set of permitted ``monodromies'' of the continuous map to $X$ that is defined on the complement of a configuration in $M$. These monodromies are detected by embeddings of discs surrounding individual points of the configuration. When $M$ is orientable, we fix an orientation and demand that these embeddings are orientation-preserving. However, if $M$ is non-orientable, we are unable to do this, and so we may only detect the monodromy \emph{up to} the involution given by an orientation reversing diffeomorphism of the disc. Thus we may only prescribe subsets of $[S^{d-1},X]_{\bZ/2}$, the set of orbits under this involution. Equivalently, we may prescribe subsets of $[S^{d-1},X]$ that are unions of orbits, which is how we have described it above. See also Example \ref{eg:trivial-bundle}, where this is explained from the perspective of \emph{configuration-section spaces}.
\end{rmk}

To topologise the set of Definition \ref{d:cmspace}, we will give a second definition of $\cmap{k}{c}{M}{X}$ that has a natural topology, and then prove that there is a natural bijection between the two definitions. To do this, we first recall some auxiliary definitions and results.

\begin{defn}[{\cite{Palais1960Localtrivialityof,Cerf1961Topologiedecertains}}]
If $G$ is a topological group with a continuous left-action on a space $X$, the action \emph{admits local sections} if each $x \in X$ has an open neighbourhood $U$ and a continuous map $\gamma \colon U \to G$ such that $\gamma(x').x = x'$ for each $x' \in U$.
\end{defn}

The utility of this definition is given by the following result.

\begin{prop}[{\cite[Theorem A]{Palais1960Localtrivialityof}}]
\label{p:fibre-bundle}
If $X$ and $Y$ are left $G$-spaces such that the action of $G$ on $Y$ admits local sections, and $f \colon X \to Y$ is $G$-equivariant, then $f$ is a fibre bundle.
\end{prop}

\begin{defn}
For a smooth manifold-with-boundary $M$, we write $\mathrm{Diff}_\partial(M)$ for the topological group of self-diffeomorphisms of $M$ that restrict to the identity on $\partial M$, equipped with the subspace topology induced by the smooth Whitney topology on $C^\infty(M,M)$, the space of all smooth self-maps of $M$.

For a smooth manifold (without boundary) $M$, we write $\mathrm{Diff}_c(M)$ for the topological group of self-diffeomorphisms $\varphi$ of $M$ such that are \emph{compactly-supported}, meaning that $\{ p \in M \mid \varphi(p) \neq p \}$ is relatively compact in $M$. This is topologised as follows. For a compact subset $K$ of $M$, write $\mathrm{Diff}_K(M)$ for the topological group of self-diffeomorphisms $\varphi$ of $M$ such that $\{ p \in M \mid \varphi(p) \neq p \} \subseteq K$, equipped with the subspace topology induced by the Whitney topology on $C^\infty(M,M)$. Note that, as a set, $\mathrm{Diff}_c(M)$ is the union of $\mathrm{Diff}_K(M)$ over all choices of $K$. We define the topology of $\mathrm{Diff}_c(M)$ to be the colimit of the topologies of $\mathrm{Diff}_K(M)$ over all choices of $K$.
\end{defn}

\begin{lem}
\label{l:local-sections1}
If $M$ is a smooth manifold without boundary, the continuous left-action of $\mathrm{Diff}_c(M)$ on $C_k(\mathring{M})$ admits local sections.
\end{lem}
\begin{proof}
Theorem B of \cite{Palais1960Localtrivialityof} implies, as a special case, that the action of $\mathrm{Diff}_\partial(M)$ on the \emph{ordered} configuration space $F_k(\mathring{M})$ admits local sections. Then one may use this, and the fact that the covering map $F_k(\mathring{M}) \to C_k(\mathring{M})$ is $\mathrm{Diff}_\partial(M)$-equivariant, to construct local sections for the action of $\mathrm{Diff}_\partial(M)$ on $C_k(\mathring{M})$. Alternatively, the statement follows directly from Proposition 4.15 of \cite{Palmer2018HomologicalstabilitymoduliI}.
\end{proof}

\begin{defn}[\emph{Configuration-mapping spaces, II.}]
\label{d:cmspace2}
Fix a subset $\hat{z} \subseteq \mathring{M}$ of cardinality $k$ and let $\mathrm{Map}_*^c(M \smallsetminus \hat{z},X)$ denote the space of continuous maps $M \smallsetminus \hat{z} \to X$ satisfying the condition \eqref{eq:singularity-condition}, equipped with the compact-open topology. Write $\mathrm{Diff}_\partial(M)$ for the group of diffeomorphisms of $M$ that fix a neighbourhood of $\partial M$, equipped with the smooth Whitney topology, and let $\mathrm{Diff}_\partial(M,\hat{z})$ denote the subgroup of diffeomorphisms that fix $\hat{z}$ as a subset. This acts (on the right) on $\mathrm{Map}^c(M \smallsetminus \hat{z},X)$ by precomposition, and on $\mathrm{Diff}_\partial(M)$ by right-multiplication, and we define:
\begin{equation}\label{eq:def-cmap}
\cmap{k}{c}{M}{X} \;\coloneqq\; \frac{\mathrm{Map}^c(M \smallsetminus \hat{z},X) \times \mathrm{Diff}_\partial(M)}{\mathrm{Diff}_\partial(M,\hat{z})}.
\end{equation}
\end{defn}

\begin{lem}
\label{l:cmap-equivalence-of-definitions}
There is a bijection between the set defined in Definition \ref{d:cmspace} and the space \eqref{eq:def-cmap}.
\end{lem}
\begin{proof}
Consider the map
\begin{equation}\label{eq:equivariant-map}
p \colon \cmap{k}{c}{M}{X} \longrightarrow C_k(\mathring{M})
\end{equation}
given by $[f,\varphi] \mapsto \varphi(\hat{z})$. There is a continuous left-action of $\mathrm{Diff}_\partial(M)$ on $\cmap{k}{c}{M}{X}$ induced by its action on itself by left-multiplication and \eqref{eq:def-cmap}, and on $C_k(\mathring{M})$ given by sending a configuration to its image under a diffeomorphism. The map \eqref{eq:equivariant-map} is equivariant with respect to these actions. The action of $\mathrm{Diff}_\partial(M)$ on $C_k(\mathring{M})$ admits local sections by Lemma \ref{l:local-sections1} and hence \eqref{eq:equivariant-map} is a fibre bundle by Proposition \ref{p:fibre-bundle}.

The fibre of \eqref{eq:equivariant-map} over a configuration $z \in C_k(\mathring{M})$ is
\begin{equation}
\label{eq:fibre-of-pi}
p^{-1}(z) \;=\; \frac{\mathrm{Map}^c(M \smallsetminus \hat{z},X) \times \mathrm{Diff}_\partial(M,\hat{z}) \cdot \varphi}{\mathrm{Diff}_\partial(M,\hat{z})},
\end{equation}
where $\mathrm{Diff}_\partial(M,\hat{z}) \cdot \varphi$ denotes the coset of $\varphi \in \mathrm{Diff}_\partial(M)$ under the action of $\mathrm{Diff}_\partial(M,\hat{z})$, where $\varphi$ is any diffeomorphism taking $\hat{z}$ to $z$. There is a canonical identification of \eqref{eq:fibre-of-pi} with $\mathrm{Map}^c(M \smallsetminus z,X)$ via $[f,\varphi] \mapsto f \circ \varphi^{-1}$. Hence a point of $\cmap{k}{c}{M}{X}$ may be specified by its image under \eqref{eq:equivariant-map}, namely an unordered configuration $z \in C_k(\mathring{M})$, together with an element of the fibre $p^{-1}(z)$, which is a continuous map $M \smallsetminus z \to X$ satisfying condition \eqref{eq:singularity-condition}. This gives a natural bijection between the set $\cmap{k}{c}{M}{X}$ defined in Definition \ref{d:cmspace}, and the formal definition \eqref{eq:def-cmap}.
\end{proof}

\begin{defn}[\emph{Configuration-mapping spaces with a boundary condition.}]
\label{d:cmspaceboundary}
If we fix a subset $D \subseteq \partial M$ and a basepoint $* \in X$, we may define
\[
\cmap{k}{c,D}{M}{X}
\]
to be the subspace of $\cmap{k}{c}{M}{X}$ consisting of pairs $(z,f)$ such that $f(p)=*$ for all $p \in D$. Equivalently (via the proof of Lemma \ref{l:cmap-equivalence-of-definitions}), we replace $\mathrm{Map}^c(M \smallsetminus \hat{z},X)$ in \eqref{eq:def-cmap} with its subspace $\mathrm{Map}^c((M \smallsetminus \hat{z},D),(X,*))$ of maps taking $D$ to $\{*\}$. Typically, we will take $D = D^{d-1} \subseteq \partial M$ to be an embedded disc.
\end{defn}

\begin{defn}[\emph{The associated fibre bundle.}]
\label{def:canonical-fibration}
From Definition \ref{d:cmspace2} and the proof of Lemma \ref{l:cmap-equivalence-of-definitions}, we have a fibre bundle
\begin{equation}\label{eq:fibre-bundle3}
p \colon \cmap{k}{c}{M}{X} \longrightarrow C_k(\mathring{M})
\end{equation}
whose fibre over a configuration $z \in C_k(\mathring{M})$ is the space of maps $M \smallsetminus z \to X$ satisfying condition \eqref{eq:singularity-condition}, and whose total space $\cmap{k}{c}{M}{X}$ is the \emph{configuration-mapping space}. For $D \subseteq \partial M$ and \emph{based} spaces $X$, there are also restricted versions of the configuration-mapping space, from Definition \ref{d:cmspaceboundary},
\[
\cmap{k}{c,D}{M}{X} \subset \cmap{k}{c}{M}{X},
\]
corresponding to restricting each fibre of \eqref{eq:fibre-bundle3} to the space of maps of pairs $(M \smallsetminus z,D) \to (X,*)$ satisfying condition \eqref{eq:singularity-condition}:
\begin{equation}\label{eq:fibre-bundle4}
p \colon \cmap{k}{c,D}{M}{X} \longrightarrow C_k(\mathring{M}),
\end{equation}
with $p^{-1}(z) = \mathrm{Map}^c((M \smallsetminus z,D),(X,*))$.
\end{defn}

\begin{rmk}
\label{rmk:path-components1}
Condition \eqref{eq:singularity-condition} depends only on the homotopy class of the map $f$, so the subspace $\mathrm{Map}^c(M \smallsetminus z,X)$ of $\mathrm{Map}(M \smallsetminus z,X)$ is a union of path-components (a similar statement also holds for the version with boundary condition on $D$).
\end{rmk}

\begin{rmk}
When $X$ is a point (so necessarily $c = [S^{d-1},X] = \{ * \}$), we have, by definition \eqref{eq:def-cmap}, $\cmap{k}{c}{M}{*} = \mathrm{Diff}_\partial(M) / \mathrm{Diff}_\partial(M,\hat{z})$. In this case, by \eqref{eq:fibre-of-pi}, each fibre of the fibre bundle \eqref{eq:equivariant-map} is a single point (here we are essentially using the fact that $\mathrm{Diff}_\partial(M)$ acts transitively on $C_k(\mathring{M})$), so \eqref{eq:equivariant-map} is a homeomorphism (since fibre bundles are open maps):
\[
\cmap{k}{c}{M}{*} = \mathrm{Diff}_\partial(M) / \mathrm{Diff}_\partial(M,\hat{z}) \cong C_k(\mathring{M}),
\]
identifying the configuration-mapping space with the usual (unordered) configuration space in this case.
\end{rmk}

\begin{rmk}
\label{rmk:Hurwitz1}
When $M = D^2$ is the $2$-disc with $D = I \subseteq \partial D^2$ an embedded interval in its boundary and $X=BG$ is the classifying space of a discrete group $G$ (so $c$ is a set of conjugacy classes of $G$), we have a homotopy equivalence:
\[
\cmap{k}{c,I}{D^2}{BG} \simeq \mathrm{Hur}_{G,k}^c,
\]
(see \cite[\S 5.8]{EllenbergVenkateshWesterland2012HomologicalstabilityHurwitz}), where $\mathrm{Hur}_{G,k}^c$ is the corresponding \emph{Hurwitz space}, the moduli space (topologised appropriately) of the data (up to an appropriate notion of equivalence) consisting of $(S,s_0,\nu,i)$, where:
\begin{itemizeb}
\item $S$ is a Riemann surface with basepoint $s_0 \in \partial S$
\item $\nu \colon S \to D^2$ is a based covering map, branched at $k$ point in the interior of $D^2$,
\item $i \colon G \hookrightarrow \mathrm{Deck}(\nu)$ is an embedding of groups,
\end{itemizeb}
such that $G$ acts transitively on the generic fibres of $\nu$ and the monodromy of $\nu$ around each of its branch points lies in one of the conjugacy classes in $c$. See also Example \ref{eg:Hurwitz2} and in particular the weak equivalences \eqref{eq:cmap-and-hur}.
\end{rmk}

\section{Configuration-section spaces}\label{s:cgamma}

We now explain how to generalise this notion to \emph{configuration-section spaces}, where we consider sections of a bundle over the complement of a configuration, instead of just a map to a fixed space. This includes, for example, moduli spaces of configurations whose complement is equipped with a \emph{tangential structure} or with a \emph{tuple of linearly independent vector fields} (which need not extend over the whole manifold).

\begin{defn}[\emph{Configuration-section spaces, without prescribed monodromy, I.}]
\label{d:csection1}
Fix a smooth, connected $d$-manifold $M$ (possibly with boundary) and a fibre bundle $\xi \colon E \to M$. For a non-negative integer $k$, the (unrestricted) \emph{configuration-section space}, as a set, is given by
\begin{equation}
\label{eq:def-csect-set}
\cgamma{k}{}{M}{\xi} = \{ (z,s) \mid z \subseteq \mathring{M} \text{ subset of cardinality } k, \; s \colon M \smallsetminus z \to E \text{ section of } \xi|_{M \smallsetminus z} \}.
\end{equation}
\end{defn}

We will topologise this, and construct the associated fibre bundle over $C_k(\mathring{M})$, in a similar way as for configuration-mapping spaces in \S\ref{s:cmap}.

\begin{defn}
For a smooth manifold-with-boundary $M$ and fibre bundle $\xi \colon E \to M$, the automorphism group $\mathrm{Aut}_\partial(\xi)$ is the subgroup
\[
\mathrm{Aut}_\partial(\xi) \;\leq\; \mathrm{Diff}_\partial(M) \times \mathrm{Homeo}(E)
\]
of those pairs $(\varphi,g)$ such that $\xi \circ g = \varphi \circ \xi$. For a subset $z \subseteq \mathring{M}$, we write $\mathrm{Aut}_\partial(\xi,z)$ for the subgroup of $(\varphi,g)$ such that $\varphi(z)=z$. Similarly, for subsets $z,z' \subseteq \mathring{M}$ of the same cardinality, we write $\mathrm{Aut}_\partial(\xi,z \mapsto z')$ for the subgroup of $(\varphi,g)$ such that $\varphi(z)=z'$. If $\partial M = \varnothing$, we write $\mathrm{Aut}_c(\xi)$ (resp.\ $\mathrm{Aut}_c(\xi,z)$ and $\mathrm{Aut}_c(\xi,z \mapsto z')$) if we replace $\mathrm{Diff}_\partial(M)$ with $\mathrm{Diff}_c(M)$ in the definition.
\end{defn}

\begin{rmk}
An element $(\varphi,g) \in \mathrm{Aut}_\partial(\xi)$ is determined by its second component $g \in \mathrm{Homeo}(E)$, since a self-homeomorphism of the total space $E$ can descend to a self-diffeomorphism of the base manifold $M$ in at most one way. Hence there is a continuous injection $\mathrm{Aut}_\partial(\xi) \hookrightarrow \mathrm{Homeo}(E)$, although the topology on $\mathrm{Aut}_\partial(\xi)$ is generally \emph{finer} than the subspace topology induced by this injection.
\end{rmk}

\begin{lem}
\label{l:local-sections2}
If $M$ is a smooth manifold-with-boundary and $\xi \colon E \to M$ is a fibre bundle, the continuous left-action of $\mathrm{Aut}_\partial(\xi)$ on $C_k(\mathring{M})$ given by $(\varphi,g) \colon z \mapsto \varphi(z)$ admits local sections.
\end{lem}
\begin{proof}
Let $z \in C_k(\mathring{M})$ and choose an embedded codimension-zero ball $B \subseteq \mathring{M}$ such that $z \subseteq \mathring{B}$. By Lemma \ref{l:local-sections1}, there is an open neighbourhood $U$ of $z$ in $C_k(\mathring{B})$ and a continuous map $\gamma \colon U \to \mathrm{Diff}_c(\mathring{B})$ such that $\gamma(z').z = z'$ for all $z' \in U$. Choose a trivialisation of $\xi|_{\mathring{B}}$. This induces a continuous group homomorphism $\mathrm{Diff}_c(\mathring{B}) \to \mathrm{Aut}_c(\xi|_{\mathring{B}})$. Now extending both the diffeomorphism of the underlying manifold and the homeomorphism of the total space by the identity on $M \smallsetminus \mathring{B}$ and on $\xi^{-1}(M \smallsetminus \mathring{B})$ defines a continuous group homomorphism $\mathrm{Aut}_c(\xi|_{\mathring{B}}) \to \mathrm{Aut}_\partial(\xi)$. Composing both of these with $\gamma$ completes the construction of local sections for the action of $\mathrm{Aut}_\partial(\xi)$ on $C_k(\mathring{M})$.
\end{proof}

\begin{defn}[\emph{Configuration-section spaces, without prescribed monodromy, II.}]
\label{d:csection-topologise}
Fix a subset $\hat{z} \subseteq \mathring{M}$ of cardinality $k$, and define
\begin{equation}\label{eq:def-csect}
\cgamma{k}{}{M}{\xi} \coloneqq \frac{\Gamma(M \smallsetminus \hat{z},\xi) \times \mathrm{Aut}_\partial(\xi)}{\mathrm{Aut}_\partial(\xi,\hat{z})},
\end{equation}
where, for a subspace $S \subseteq M$, we write $\Gamma(S,\xi) = \{ s \in \mathrm{Map}(S,E) \mid \xi\circ s = \mathrm{incl} \}$. The right-action of $\mathrm{Aut}_\partial(\xi,\hat{z})$ on $\Gamma(M \smallsetminus \hat{z},\xi)$ is given by $(\varphi,g) \colon s \mapsto g^{-1} \circ s \circ \varphi$.
\end{defn}

\begin{lem}
\label{l:cgamma-equivalence-of-definitions}
There is a bijection between the set defined in Definition \ref{d:csection1} and the space \eqref{eq:def-csect}.
\end{lem}
\begin{proof}
There is a continuous map
\begin{equation}\label{eq:fibre-bundle-section}
p \colon \cgamma{k}{}{M}{\xi} \longrightarrow C_k(\mathring{M})
\end{equation}
given by $[s,(\varphi,g)] \mapsto \varphi(z)$. This map is equivariant with respect to the continuous left-actions of $\mathrm{Aut}_\partial(\xi)$, and the action of $\mathrm{Aut}_\partial(\xi)$ on $C_k(\mathring{M})$ admits local sections by Lemma \ref{l:local-sections2}. Thus, by Proposition \ref{p:fibre-bundle}, the map \eqref{eq:fibre-bundle-section} is a fibre bundle. The fibre of \eqref{eq:fibre-bundle-section} over $z \in C_k(\mathring{M})$ is
\begin{equation}
\label{eq:fibre-of-pi2}
p^{-1}(z) \;=\; \frac{\Gamma(M \smallsetminus \hat{z},\xi) \times \mathrm{Aut}_\partial(\xi,\hat{z} \mapsto z)}{\mathrm{Aut}_\partial(\xi,\hat{z})}.
\end{equation}
Note that $\mathrm{Aut}_\partial(\xi)$ acts \emph{transitively} on $C_k(\mathring{M})$. (This can be seen as follows. Let $z,z' \in C_k(\mathring{M})$ and choose an embedded codimension-zero ball $B \subseteq \mathring{M}$ such that $z \cup z' \subseteq \mathring{B}$. Since $\mathrm{Diff}_c(\mathring{B})$ acts transitively on $C_k(\mathring{B})$, we may find a diffeomorphism $\varphi$ of $\mathring{B}$ such that $\varphi(z)=z'$, lift it to an automorphism of $\xi|_{\mathring{B}}$ by a choice of trivialisation of $\xi|_{\mathring{B}}$ and then extend it by the identity to obtain an element of $\mathrm{Aut}_\partial(\xi)$ sending $z$ to $z'$.) Thus the subspace $\mathrm{Aut}_\partial(\xi,\hat{z} \mapsto z)$ is a coset of $\mathrm{Aut}_\partial(\xi,\hat{z})$ in $\mathrm{Aut}_\partial(\xi)$, and may be identified with $\Gamma(M \smallsetminus z,\xi)$ via $[s,(\varphi,g)] \mapsto g \circ s \circ \varphi^{-1}$. Thus there is a natural bijection between the set defined in Definition \ref{d:csection1} and the formal definition \eqref{eq:def-csect}.
\end{proof}

\begin{defn}[\emph{The associated fibre bundle.}]
\label{def:canonical-fibration2}
From Definition \ref{d:csection-topologise} and the proof of Lemma \ref{l:cgamma-equivalence-of-definitions}, we have a fibre bundle
\begin{equation}
\label{eq:fibre-bundle6}
p \colon \cgamma{k}{}{M}{\xi} \longrightarrow C_k(\mathring{M})
\end{equation}
whose fibre over a configuration $z \in C_k(\mathring{M})$ is the space of sections $\Gamma(M \smallsetminus z,\xi)$.
\end{defn}

\begin{defn}[\emph{Boundary conditions.}]
\label{d:cgamma-boundary}
If we fix a subset $D \subseteq \partial M$ and a section $s_D \in \Gamma(D,\xi)$, we have a restricted version of the configuration-section space:
\begin{equation}
\label{eq:def-csect-restricted}
\cgamma{k}{D}{M}{\xi} = \{ (z,s) \in \cgamma{k}{}{M}{\xi} \mid s|_D = s_D \}.
\end{equation}
Moreover, since $D$ is contained in the boundary of $M$, the subspace $\cgamma{k}{D}{M}{\xi} \subseteq \cgamma{k}{}{M}{\xi}$ is invariant under the action of $\mathrm{Aut}_\partial(\xi)$, so Proposition \ref{p:fibre-bundle} and Lemma \ref{l:local-sections2} imply that the restriction
\begin{equation}
\label{eq:fibre-bundle7}
p \colon \cgamma{k}{D}{M}{\xi} \longrightarrow C_k(\mathring{M})
\end{equation}
of the map \eqref{eq:fibre-bundle-section} is a fibre bundle. The fibre $p^{-1}(z)$ over a configuration $z \in C_k(\mathring{M})$ is the space of sections $\{ s \in \Gamma(M \smallsetminus z,\xi) \mid s|_D = s_D \}$ (generalising Definition \ref{def:canonical-fibration2}, which corresponds to $D = \varnothing$).
\end{defn}

We now wish to define \emph{configuration-section spaces with prescribed monodromy}
\[
\cgamma{k}{c}{M}{\xi} \subseteq \cgamma{k}{}{M}{\xi}.
\]
In other words, we wish to specify a certain ``local behaviour'' $c$ of the section $s \in \Gamma(M \smallsetminus z,\xi)$ near the ``singularities'' $z \subseteq \mathring{M}$. To do this, we first construct a covering space over the interior $\mathring{M}$ of $M$ depending on the fibre bundle $\xi \colon E \to M$.

\begin{defn}[\emph{The covering space of local sections of $\xi$, informal.}]
Informally, the covering space
\[
\eta(\xi) \colon \Sigma(\xi) \too \mathring{M}
\]
is defined by prescribing that its fibre $\eta(\xi)^{-1}(p)$ over a point $p$ in the interior of $M$ consists of the set of all \emph{germs} of sections of $\xi$ defined on a small punctured-disc neighbourhood of $p$ in $M$.
\end{defn}

\begin{defn}[\emph{The covering space of local sections of $\xi$, formal.}]
\label{d:local-sections}
Fix a point $p \in \mathring{M}$. A \emph{local section near $p$ of $\xi$} is a pair $(B,\sigma)$ consisting of a subset $B \subseteq M$ homeomorphic to a $d$-dimensional ball, containing $p$ in its interior, together with a section $\sigma$ of $\xi|_{\partial B}$. Let $\sim$ be the equivalence relation on such pairs generated by $(B,\sigma) \sim (B',\sigma')$ if $B$ is contained in the interior of $B'$ and the section $\sigma \sqcup \sigma'$ defined on $\partial B \sqcup \partial B'$ extends to a section over $B' \smallsetminus \mathrm{int}(B)$. Write $[B,\sigma]_p$ for the equivalence class containing $(B,\sigma)$.

Let $\Sigma(\xi)$ be the set of all pairs $(p,[B,\sigma]_p)$ where $p \in \mathring{M}$ and $[B,\sigma]_p$ is an equivalence class of local sections near $p$ of $\xi$. We define a topology on $\Sigma(\xi)$ as follows. Let $(p,[B,\sigma]_p) \in \Sigma(\xi)$ and choose a representative $(B,\sigma)$ for the equivalence class $[B,\sigma]_p$. Define
\[
\cN_{p,B,\sigma} = \{ (q,[B,\sigma]_q) \mid q \in \mathrm{int}(B) \} \subseteq \Sigma(\xi).
\]
Then one may check that the collection $\cN$ of the sets $\cN_{p,B,\sigma}$ for all $(p,[B,\sigma]_p) \in \Sigma(\xi)$ is a basis for a topology $\cT$ on $\Sigma(\xi)$ such that the map
\[
\eta(\xi) \colon \Sigma(\xi) \longrightarrow \mathring{M}
\]
given by $(p,[B,\sigma]) \mapsto p$ is a covering map.
\end{defn}

\begin{rmk}
Although the basis $\cN$ depends on a choice, for each point $(p,[B,\sigma]_p) \in \Sigma(\xi)$, of a representative $(B,\sigma)$ of the equivalence class $[B,\sigma]_p$, the topology $\cT$ that it generates does not depend on these choices.
\end{rmk}

\begin{rmk}
Over a point $p \in \mathring{M}$, the fibre of $\eta(\xi) \colon \Sigma(\xi) \to \mathring{M}$ may be identified with the set of (unbased) homotopy classes of maps $[S^{d-1},F]$, where $F = \xi^{-1}(p)$. This identification is canonical once we have chosen a local orientation of $\mathring{M}$ at $p$.
\end{rmk}

\begin{defn}
\label{s:singularity}
A \emph{singularity condition} for $\xi \colon E \to M$ is a set of components of the space $\Sigma(\xi)$.
\end{defn}

In order to define configuration-section spaces with prescribed singularity conditions, we will use the following construction.

\begin{construction}
\label{construction-local-section}
There is a continuous (hence locally-constant) map
\begin{equation}
\label{eq:local-section}
\mathrm{loc}_\xi \colon \cgamma{k}{}{M}{\xi} \longrightarrow SP_k(\pi_0(\Sigma(\xi)))
\end{equation}
to the set (with the discrete topology) of unordered $k$-tuples of components of $\Sigma(\xi)$, which records the local behaviour of a given configuration-section $(z,s)$ in a punctured neighbourhood of each $p \in z$. Formally, this is defined as follows: for each $p \in z$, choose a codimension-$0$ ball $B \subseteq M$ containing $p$ in its interior and disjoint from $z \smallsetminus \{p\}$, and let $U_{(s,p)}$ be the component of $\Sigma(\xi)$ containing the point $(p,[B,s|_{\partial B}]_p) \in \Sigma(\xi)$. Then we set
\[
\mathrm{loc}_\xi(z,s) = \{ U_{(s,p)} \mid p \in z \},
\]
where the right-hand side is viewed as a multiset.

The component $U_{(s,p)} \in \pi_0(\Sigma(\xi))$ is referred to as the \emph{charge} of the particle $p$ in the field $s$.
\end{construction}

\begin{rmk}
It may be tempting to think that a configuration-section $(z,s)$ determines a \emph{section} of the covering $\Sigma(\xi) \to \mathring{M}$, and hence that each $U_{(s,p)}$ in the construction above is the same trivial component of the covering. Indeed, the first statement is true, \emph{at the level of sets and functions}, since the section $s$ is well-defined away from a discrete subset of $\mathring{M}$, so it is certainly well-defined on a punctured-disc neighbourhood of each point $p \in \mathring{M}$. However, this induced section of the covering $\Sigma(\xi) \to \mathring{M}$ will in general have discontinuities at the configuration points $z$.
\end{rmk}

\begin{defn}[\emph{Configuration-section spaces with prescribed monodromy.}]
\label{d:csection2}
As in Definition \ref{d:csection1}, let $M$ be a smooth, connected $d$-manifold and $\xi \colon E \to M$ be a fibre bundle, and now also choose a singularity condition $c \subseteq \pi_0(\Sigma(\xi))$. Then
\[
\cgamma{k}{c}{M}{\xi} \coloneqq \mathrm{loc}_{\xi}^{-1}(SP_k(c)) \subseteq \cgamma{k}{}{M}{\xi}.
\]
If we fix a subset $D \subseteq \partial M$ and a section $s_D \in \Gamma(D,\xi)$, we define (just as in \eqref{eq:def-csect-restricted}):
\begin{equation}
\label{eq:def-csect-restricted2}
\cgamma{k}{c,D}{M}{\xi} = \{ (z,s) \in \cgamma{k}{c}{M}{\xi} \mid s|_D = s_D \} .
\end{equation}
\end{defn}

\begin{rmk}
\label{rmk:path-components2}
Since \eqref{eq:local-section} is locally-constant, the subspace $\cgamma{k}{c}{M}{\xi}$ of $\cgamma{k}{}{M}{\xi}$ is a union of path-components, and similarly for the restricted version $\eqref{eq:def-csect-restricted2} \subseteq \eqref{eq:def-csect-restricted}$. (Compare Remark \ref{rmk:path-components1}.)
\end{rmk}

\begin{defn}[\emph{The associated fibration.}]
\label{d:associated-fibration}
From Definition \ref{d:cgamma-boundary} we have a fibre bundle \eqref{eq:fibre-bundle7} with total space $\cgamma{k}{D}{M}{\xi}$. By Remark \ref{rmk:path-components2}, the subspace $\cgamma{k}{c,D}{M}{\xi} \subseteq \cgamma{k}{D}{M}{\xi}$ is a union of path-components. Thus, the restriction
\begin{equation}
\label{eq:fibre-bundle8}
p \colon \cgamma{k}{c,D}{M}{\xi} \longrightarrow C_k(\mathring{M})
\end{equation}
of the fibre bundle \eqref{eq:fibre-bundle7} to the subspace $\cgamma{k}{c,D}{M}{\xi}$ is a Hurewicz fibration.\footnote{The base space of the fibre bundle \eqref{eq:fibre-bundle7} is paracompact, so it is a Hurewicz fibration. In general, if $f \colon E \to B$ is a Hurewicz fibration and $E_0 \subseteq E$ is a union of path-components, the composition $f \circ \mathrm{incl} \colon E_0 \hookrightarrow E \to B$ is also a Hurewicz fibration.} This is the \emph{associated fibration} of the configuration-section space $\cgamma{k}{c,D}{M}{\xi}$. Its fibre over $z \in C_k(\mathring{M})$ is denoted by
\[
\Gamma^{c,D}(M \smallsetminus z;\xi) = \{ s \in \Gamma(M \smallsetminus z;\xi) \mid s|_D = s_D \text{ and } \mathrm{loc}_\xi(z,s) \in c \} .
\]
\end{defn}

As our first family of examples, we explain how to recover the notion of \emph{configuration-mapping space} of \S\ref{s:cmap} in the case of a trivial bundle over $M$.

\begin{eg}
\label{eg:trivial-bundle}
If $\xi$ is the trivial bundle $M \times X \to M$ for a space $X$, then we clearly have an identification
\begin{equation}
\label{eq:trivial-bundle-identification}
\cgamma{k}{}{M}{\xi} = \cmap{k}{}{M}{X}.
\end{equation}

If $M$ is orientable, then the covering $\eta(\xi) \colon \Sigma(\xi) \to \mathring{M}$ is simply the disjoint union of copies of the trivial (identity) covering $\mathring{M} \to \mathring{M}$, one for each element of $[S^{d-1},X]$. In other words, $\eta(\xi)$ is isomorphic as a covering to the projection $\mathring{M} \times [S^{d-1},X] \to \mathring{M}$. This isomorphism of coverings depends on a choice of orientation of $M$.

If $M$ is non-orientable, then the covering $\eta(\xi) \colon \Sigma(\xi) \to \mathring{M}$ is a disjoint union of a number of copies of the identity covering $\mathring{M} \to \mathring{M}$ and a number of copies of the orientation double covering $\mathring{M}^{\mathrm{or}} \to \mathring{M}$. More precisely, consider the involution on the set $[S^{d-1},X]$ given by precomposition by a reflection of $S^{d-1}$. There is one copy of the identity covering in $\eta(\xi)$ for each fixed point of this action and one copy of the orientation double covering in $\eta(\xi)$ for each orbit of size two. In other words, writing $\cO_1$ for the set of orbits of size $1$ and $\cO_2$ for the set of orbits of size $2$ in $[S^{d-1},X]$, we have that $\eta(\xi)$ is isomorphic as a covering to
\begin{equation}
\label{eq:covering-space-nonorientable}
\mathrm{pr} \circ (\mathrm{id} \sqcup (\mathrm{or} \times \mathrm{id})) \colon
(\mathring{M} \times \cO_1) \sqcup (\mathring{M}^{\mathrm{or}} \times \cO_2) \too \mathring{M} \times (\cO_1 \sqcup \cO_2) \too \mathring{M},
\end{equation}
where $\mathrm{or}$ is the orientation double covering of $\mathring{M}$. This isomorphism is canonical up to the action of the $2^{\cO_2}$ deck transformations of \eqref{eq:covering-space-nonorientable} corresponding to the deck transformations of each of the copies of $\mathring{M}^{\mathrm{or}}$ (acting independently). In particular, the components $\pi_0(\Sigma(\xi))$ of the covering space are canonically identified with the orbits of the involution on $[S^{d-1},X]$.

A singularity condition $c \subseteq \pi_0(\Sigma(\xi))$ therefore corresponds to (compare Remark \ref{rmk:fixed-points}):
\begin{itemizeb}
\item (if $M$ is orientable) a subset of $[S^{d-1},X]$,
\item (if $M$ is non-orientable) a subset of $[S^{d-1},X]_{\bZ/2}$, the orbits of $[S^{d-1},X]$ under the involution given by pre-composition with a reflection of $S^{d-1}$.
\end{itemizeb}
In the orientable case, this correspondence depends on a choice of orientation of $M$; in the non-orientable case, it does not depend on any choice.
Interpreting the singularity condition $c$ in this way on the right-hand side, the identification \eqref{eq:trivial-bundle-identification} restricts to identifications:
\begin{equation}
\label{eq:trivial-bundle-identification2}
\cgamma{k}{c}{M}{\xi} = \cmap{k}{c}{M}{X}
\end{equation}
for each $c \subseteq \pi_0(\Sigma(\xi))$.
\end{eg}

\begin{rmk}
Theorem \ref{tmain} holds under the hypothesis that the pre-image of the singularity condition in $\pi_{d-1}(X)$ has size $1$. (The precise meaning of this is spelled out in the discussion just before Theorem \ref{thm-hom-stab}.) In light of the above discussion, this means that, for configuration-\emph{mapping} spaces with target $X$ and underlying manifold $M$, it holds precisely for those singularity conditions that correspond to an element of $\pi_{d-1}(X)$ that is fixed under the $\pi_1(X)$-action and additionally (\emph{if} $M$ is non-orientable) is fixed under the operation of taking inverses.
\end{rmk}

There is a natural variant of configuration-section spaces where configurations are equipped just with \emph{homotopy-classes} of sections on their complements, which are covering spaces of $C_k(\mathring{M})$:

\begin{defn}[\emph{Configuration-section spaces up to homotopy.}]
\label{d:cgamma-up-to-homotopy}
Define $h\cgamma{k}{}{M}{\xi}$ to be the quotient of $\cgamma{k}{}{M}{\xi}$ by the equivalence relation given by $(z,s) \sim (z',s')$ if $z=z'$ and the sections $s,s' \colon M \smallsetminus z \to E$ are homotopic through sections of $\xi|_{M \smallsetminus z}$, in other words, lie in the same path-component of the space $\Gamma(M \smallsetminus z,\xi)$.
The locally-constant map $\eqref{eq:local-section} = \mathrm{loc}_\xi$ factors into the quotient map $\cgamma{k}{}{M}{\xi} \to h\cgamma{k}{}{M}{\xi}$ and a locally-constant map
\begin{equation}\label{eq:local-section2}
h\mathrm{loc}_\xi \colon h\cgamma{k}{}{M}{\xi} \longrightarrow \pi_0(\Sigma(\xi)),
\end{equation}
since homotopic sections of $\xi|_{M \smallsetminus z}$ have the same local germs in punctured neighbourhoods of every point in $z$. We then define
\[
h\cgamma{k}{c}{M}{\xi} \coloneqq h\mathrm{loc}_{\xi}^{-1}(SP_k(c))
\]
for any singularity condition $c \subseteq \pi_0(\Sigma(\xi))$. We may equivalently define $h\cgamma{k}{c}{M}{\xi}$ to be the quotient of $\cgamma{k}{c}{M}{\xi}$ by the restriction of the equivalence relation above.
More generally, if we fix a subset $D \subseteq \partial M$ and a section $s_D \in \Gamma(D,\xi)$, we may define $h\cgamma{k}{c,D}{M}{\xi}$ to be the quotient of $\cgamma{k}{c,D}{M}{\xi}$ by the equivalence relation given by $(z,s) \sim (z',s')$ if $z=z'$ and the sections $s,s'$ lie in the same path-component of the space
\[
\Gamma(M \smallsetminus z , \xi ; s_D) = \{ s \in \Gamma(M \smallsetminus z,\xi) \mid s|_D = s_D \} .
\]
\end{defn}

\begin{eg}[\emph{Configurations equipped with a cohomology class of the complement.}]
\label{eg:cohomology-class}
As a special case, of course, we have \emph{configuration-mapping spaces up to homotopy}
\[
h\cmap{k}{c}{M}{X} = h\cgamma{k}{c}{M}{M \times X}.
\]
As a set, $h\cmap{k}{c}{M}{X}$ consists of pairs $(z,f)$, where $z \subseteq \mathring{M}$ has cardinality $k$ and $f$ is a homotopy class of maps $M \smallsetminus z \to X$ whose monodromy around each point of $z$ lies in $c \subseteq [S^{d-1},X]$. If we take $X$ to be the Eilenberg-MacLane space $K(G,d-1)$ for an abelian group $G$, then $c$ is a subset of $H^{d-1}(S^{d-1};G) \cong G$ and a point in
\[
h\cmap{k}{c}{M}{K(G,d-1)}
\]
is a configuration $z \subseteq \mathring{M}$ equipped with a cohomology class $\alpha \in H^{d-1}(M \smallsetminus z ; G)$ whose restriction to each embedded sphere $S^{d-1} \subseteq M \smallsetminus z$ that encloses exactly one point of $z$, lies in $c$.
\end{eg}

Before describing the next example of configuration-mapping spaces up to homotopy, we note that, under certain conditions, configuration-mapping spaces up to homotopy have the same weak homotopy type as the corresponding configuration-mapping spaces (not up to homotopy).

\begin{lem}
\label{lem:contractible-components}
Let $M$ be a compact, connected $d$-manifold with basepoint $* \in \partial M$ and $X$ a based, path-connected space with a choice of subset $c \subseteq [S^{d-1},X]$. Assume that $M$ is homotopy equivalent to a wedge of $(d-1)$-spheres and that $X$ is $d$-coconnected, meaning that $\pi_i(X)=0$ for all $i\geq d$. Then, for any $k\geq 0$, the quotient map
\begin{equation}
\label{eq:cmap-and-cmapt2}
\cmap{k}{c,*}{M}{X} \longrightarrow h\cmap{k}{c,*}{M}{X}
\end{equation}
is a weak homotopy equivalence. In particular, this holds if $M=S$ is a compact, connected surface, $X=BG$ is the classifying space of a discrete group $G$ and $c$ is a set of conjugacy classes of $G$.
\end{lem}
\begin{proof}
The quotient map fits into a map of fibration sequences:
\begin{equation}
\label{eq:cmap-and-cmapt-diagram}
\centering
\begin{split}
\begin{tikzpicture}
[x=1mm,y=1mm]
\node (tl) at (0,30) {$\mathrm{Map}^{c,*}(M \smallsetminus z,X)$};
\node (tr) at (50,30) {$\pi_0(\mathrm{Map}^{c,*}(M \smallsetminus z,X))$};
\node (ml) at (0,15) {$\cmap{k}{c,*}{M}{X}$};
\node (mr) at (50,15) {$h\cmap{k}{c,*}{M}{X}$};
\node (b) at (25,0) {$C_k(\mathring{M})$};
\draw[->] (tl) to node[above,font=\small]{$(\star)$} (tr);
\draw[->] (ml) to node[above,font=\small]{\eqref{eq:cmap-and-cmapt2}} (mr);
\draw[->] (tl) to (ml);
\draw[->] (tr) to (mr);
\draw[->] (ml) to (b);
\draw[->] (mr) to (b);
\end{tikzpicture}
\end{split}
\end{equation}
Since $M$ is homotopy equivalent to a wedge of some number $\ell$ of $(d-1)$-spheres and has non-empty boundary, we have $M \smallsetminus z \simeq \bigvee^{\lvert z \rvert + \ell} S^{d-1}$ and thus
\begin{equation}
\label{eq:product-of-loopspaces}
\mathrm{Map}^{c,*}(M \smallsetminus z,X) \simeq (\Omega^{d-1}X)^\ell \times (\Omega_c^{d-1}X)^{\lvert z \rvert},
\end{equation}
where $\Omega_c^{d-1}X \subseteq \Omega^{d-1}X$ is the union of path-components corresponding to $c \subseteq [S^{d-1},X]$ under the identification of $[S^{d-1},X]$ with $\pi_{d-1}(X)/\pi_1(X)$. Since $X$ is $d$-coconnected, the space \eqref{eq:product-of-loopspaces} is $1$-coconnected, in other words its path-components are weakly contractible, and so the map $(\star)$ of fibres in \eqref{eq:cmap-and-cmapt-diagram} is a weak homotopy equivalence, and therefore so is the map \eqref{eq:cmap-and-cmapt2}.
\end{proof}

\begin{rmk}
Lemma \ref{lem:contractible-components} generalises in an analogous way to configuration-section spaces, for a bundle $\xi \colon E \to M$ equipped with a section over $\{*\} \subseteq \partial M$, i.e., a point $e_0 \in E$ with $\xi(e_0) = *$.
\end{rmk}

\begin{eg}[\emph{Hurwitz spaces.}]
\label{eg:Hurwitz2}
Following on from Example \ref{eg:cohomology-class}, let $S$ be a compact, connected surface with boundary and now take $X = K(G,1) = BG$ for a (not necessarily abelian) discrete group $G$. Note that $S \smallsetminus z$, for any finite subset $z \subseteq \mathring{S}$, is aspherical, so we have a natural bijection
\[
[S \smallsetminus z , BG] \cong \mathrm{Hom}(\pi_1(S \smallsetminus z),G) / G,
\]
where the quotient on the right-hand side is by the action of $G$ given by post-composition by inner automorphisms. A point in
\[
h\cmap{k}{c}{S}{BG}
\]
therefore consists of a configuration $z \subseteq \mathring{S}$ equipped with a homomorphism $\pi_1(S \smallsetminus z) \to G$ modulo inner automorphisms of $G$. If we now take $D = \{*\} \subseteq \partial S$ to be a point, we have a natural bijection
\[
\langle S \smallsetminus z , BG \rangle \cong \mathrm{Hom}(\pi_1(S \smallsetminus z),G),
\]
where $\langle -,- \rangle$ denotes based homotopy classes of based maps, and so a point in
\[
h\cmap{k}{c,*}{S}{BG}
\]
consists of a configuration $z \subseteq \mathring{S}$ equipped with a homomorphism $\pi_1(S \smallsetminus z) \to G$. In particular, if $S = D^2$ we have a homeomorphism
\[
h\cmap{k}{c,*}{D^2}{BG} \cong \mathrm{Hur}_{G,k}^c,
\]
where the right-hand side is the corresponding Hurwitz space. Moreover, Lemma \ref{lem:contractible-components} implies that the quotient map
\begin{equation}
\label{eq:cmap-and-cmapt}
\cmap{k}{c,*}{S}{BG} \longrightarrow h\cmap{k}{c,*}{S}{BG}
\end{equation}
is a weak homotopy equivalence for any compact, connected surface-with-boundary $S$. In the case $S=D^2$, we therefore have weak homotopy equivalences (compare Remark \ref{rmk:Hurwitz1}):
\begin{equation}
\label{eq:cmap-and-hur}
\cmap{k}{c,I}{D^2}{BG} \simeq \cmap{k}{c,*}{D^2}{BG} \simeq h\cmap{k}{c,*}{D^2}{BG} \cong \mathrm{Hur}_{G,k}^c .
\end{equation}
\end{eg}

\begin{rmk}
\label{rmk:fibrations-and-homotopy}
Before discussing our next two examples, we remark that it is not really necessary for $\xi \colon E \to M$ to be a \emph{fibre bundle} for our results to hold; it is enough for it to be a \emph{fibration}. In addition, our methods and results go through equally well when configuration-section spaces $\mathrm{C}\Gamma$ are replaced by configuration-section spaces up to homotopy $h\mathrm{C}\Gamma$ (Definition \ref{d:cgamma-up-to-homotopy}). In fact, since configuration-section spaces up to homotopy $h\cgamma{k}{c,*}{M}{\xi}$ are simply covering spaces of the corresponding configuration spaces $C_k(\mathring{M})$,\footnote{Moreover, if $M$ has dimension at least $3$, its handle dimension is at most $\mathrm{dim}(M)-2$ and $c$ corresponds to a single element of $\pi_{d-1}$ of the fibre of $\xi$, then $h\cgamma{k}{c,*}{M}{\xi}$ is a \emph{trivial} covering space of $C_k(\mathring{M})$. Thus the interesting examples of configuration-section spaces up to homotopy are those where $M$ either is a surface or has large handle dimension, or the singularity condition $c$ corresponds to a larger subset of $\pi_{d-1}$ of the fibre of $\xi$.} our proofs simplify drastically when replacing $\mathrm{C}\Gamma$ with $h\mathrm{C}\Gamma$, and homological stability in this case is a relatively straightforward consequence of classical homological stability for the configuration spaces $C_k(\mathring{M})$. However, even if they are a somewhat degenerate instance of our main theorem, we mention these examples, as they nevertheless have interesting interpretations in terms of topological structures on the complement of a configuration.
\end{rmk}

\begin{eg}[\emph{Tangential structures.}]
\label{eg:tangential}
If $\theta \colon E \to BO(d)$ is fibration, and $\tau_M$ denotes the homotopy class of maps $M \to BO(d)$ classifying the tangent bundle of $M$, then a \emph{tangential structure of type $\theta$} on $M$ is a lift of $\tau_M$ up to homotopy to a map $M \to E$. Picking a specific map representing the homotopy class $\tau_M$ and pulling $\theta$ back along it, we may equivalently view a tangential structure of type $\theta$ as a section up to homotopy of $(\tau_M)^*(\theta) \colon (\tau_M)^*(E) \to M$. Hence
\[
h\cgamma{k}{*}{M}{(\tau_M)^*(\theta)}
\]
is the moduli space of $k$-point configurations in $M$ whose complement is equipped with a tangential structure of type $\theta$, and its subspaces corresponding to a singularity condition $c$ may be interpreted as moduli spaces of configurations whose complement is equipped with a tangential structure of type $\theta$ with prescribed monodromy around the configuration points.

In particular, if $\theta = \theta_r \colon BO(d)\langle r \rangle \to BO(d)$ is the $r$-connected covering of $BO(d)$, then its homotopy fibre $X$ has trivial homotopy group in degrees $\geq r$. In particular, if $d = \mathrm{dim}(M) \geq r+1$, we have $\pi_{d-1}(X)=0$, so Theorem \ref{tmain} applies to the full, unrestricted configuration-section spaces (up to homotopy; see Remark \ref{rmk:fibrations-and-homotopy}). For example, if $d \geq 3$ this applies to spin structures ($r=2$), if $d \geq 5$ it applies to string structures ($r=4$), etc., as mentioned in Example \ref{eg:spin-etc} in the introduction.
\end{eg}

\begin{eg}[\emph{Flat connections.}]
\label{eg:flat-connection}
As another example, let $P \to M$ be a principal $G$-bundle, for a topological group $G$, classified by a map $f \colon M \to BG$. Let $\theta \colon BG^\delta \to BG$ be the map induced by the continuous homomorphism $G^\delta \to G$ given by the identity on underlying sets. If we take $\xi = f^*(\theta)$, then a section of $\xi$ (up to homotopy through sections) is a flat connection for $P \to M$. Hence
\[
h\cgamma{k}{*}{M}{f^*(\theta)}
\]
is the space of $k$-point configurations in $M$ whose complement is equipped with a flat connection for $P \to M$. From the long exact sequence of the map $\theta$, we see that $\pi_i(X) \cong \pi_i(G)$ for $i \geq 2$, where $X$ is the homotopy fibre of $\theta$. Thus, if $d = \mathrm{dim}(M) = 3$ and $G$ is a Lie group, we have $\pi_{d-1}(X) = \pi_2(X) = \pi_2(G) = 0$, and so Theorem \ref{tmain} applies to the full, unrestricted configuration-section spaces (up to homotopy; see Remark \ref{rmk:fibrations-and-homotopy}), as mentioned in Example \ref{eg:spin-etc} in the introduction.
\end{eg}

\begin{eg}[\emph{Linearly independent vector fields and distributions.}]
\label{eg:vector-fields}
Let $TM \to M$ denote the tangent bundle of $M$ and write $\xi_r \colon \mathrm{Lin}_r(TM) \to M$ for the associated fibre bundle whose fibre over $p \in M$ is the subspace of $(T_p M)^r$ consisting of linearly independent $r$-tuples of tangent vectors at $p$. Then the configuration-section space
\[
\cgamma{k}{}{M}{\xi_r}
\]
is the space of $k$-point configurations $z \subseteq \mathring{M}$ equipped with an $r$-tuple of linearly independent vector fields on $M \smallsetminus z$. In particular, when $r=1$, this is the space of configurations equipped with a non-vanishing vector field on the complement.

On the other hand, if we replace $\xi_r$ by its quotient bundle whose fibre over $p \in M$ is the Grassmannian of $r$-planes in $T_p M$, then we obtain the space of $k$-point configurations whose complement is equipped with an $r$-plane distribution.
\end{eg}

\begin{eg}
If we take $X = B\Gamma_k$ to be Haefliger's classifying space for codimension-$k$ smooth foliations \cite{Haefliger1958, Haefliger1970, Haefliger1971}, then configuration-mapping spaces with target $X$ are moduli spaces of configurations whose complement is equipped with a rank-$k$ foliated smooth microbundle \cite{Segal1978}.
\end{eg}

\begin{rmk}
Although configuration-section spaces, as we have defined them, include all possible \emph{tangential} structures on the complement, it would also be very interesting to study configuration spaces whose complements are equipped with geometric structures that do \emph{not} arise as a tangential structure, for example foliations or Riemannian metrics with bounded curvature, etc.
\end{rmk}

\begin{rmk}
There is a closely-related notion of \emph{configuration-bundle-map spaces} \cite[\S 2.2]{Nariman2018}, which consist of configurations $z$ in the interior of $M$ together with a bundle map
\[
( T(M \smallsetminus z) \to M \smallsetminus z ) \too ( \eta \colon E \to B ),
\]
where $\eta \colon E \to B$ is a fixed fibre bundle. This has some overlap with the notion of configuration-section spaces, although neither is more general than the other. In particular, if the bundle $\xi \colon E \to M$ is pulled back from a bundle over $BO(d)$ along a map classifying the tangent bundle of $M$, then the corresponding configuration-section space may be expressed as a configuration-bundle-map space. Conversely, if the bundle $E \to B$ is the pullback of the tautological bundle over $BO(d)$ along some map $f \colon B \to BO(d)$, then the corresponding configuration-bundle-map space may be expressed as a configuration-section space.

We expect that the techniques and results of this paper carry over analogously to configuration-bundle-map spaces, although we do not go into the details of this here.
\end{rmk}

\section{On \texorpdfstring{$E_m$}{E-m}-modules over \texorpdfstring{$E_n$}{E-n}-algebras}\label{s:em-modules-over-en-algebras}

All of the structure that we will use in studying configuration-section spaces will arise from their structure as an \emph{$E_0$-module over an $E_{d-1}$-algebra}, which will be defined explicitly, in appropriate models, in this section. Before this, we recall the notion of \emph{$E_m$-module over an $E_n$-algebra} for any $0\leq m\leq n$, and explain why, for $n$ fixed, these notions coincide for all $m \in \{0,1,\ldots,n-1\}$.

\begin{rmk}
\label{rmk:e0modules}
The last sentence above means that, when $d\geq 3$, we may equally well describe the structure that we construct as an ``$E_1$-module'' over an $E_{d-1}$-algebra. However, since we also consider the case of dimension $d=2$, we prefer to use the term ``$E_0$-module'' throughout, for consistency. This diverges from the terminology of \cite{Krannich2019Homologicalstabilitytopological}, where the name ``$E_1$-module'' is used -- this is (as we explain below) equivalent since that paper considers only $E_n$-algebras where $n\geq 2$.
\end{rmk}

We begin by recalling several different flavours of Swiss cheese operads. For an integer $n\geq 0$, let $D^n$ denote the closed unit disc in $\bR^n$. For a space $X$ and an integer $k\geq 0$, let $\bar{F}_k(X)$ denote the ordered configuration space of $k$ points in $X$ labelled by positive real numbers:
\[
\bar{F}_k(X) = \{ ((x_1,r_1),\ldots,(x_n,r_n)) \in (X \times (0,\infty))^n \mid x_i \neq x_j \text{ for } i \neq j \} .
\]
Fix an integer $n\geq 0$.

\begin{defn}
\label{d:little-discs}
The \emph{little $n$-discs operad} $\disc_n$ has one colour $\sa$. For any integer $k\geq 0$ its space of operations $\disc_n(\sa^k;\sa)$ is the subspace of $\bar{F}_k(D^n)$ of configurations satisfying
\begin{itemizeb}
\item[(i)] $\lvert x_i \rvert \leq 1 - r_i$
\item[(ii)] $\lvert x_i - x_j \rvert \geq r_i + r_j$ for $i\neq j$.
\end{itemizeb}
Interpreting such configurations as little $n$-discs in $D^n$ whose interiors are disjoint, with centres $x_i$ and radii $r_i$ --  see Figure \ref{fig:swiss-cheese}(a) -- the operadic composition is defined by embedding $D^n$ into these smaller discs by translations and dilations. The symmetric action is given by the natural action of $\Sigma_k$ on $\bar{F}_k(D^n)$.
\end{defn}

\begin{defn}
A \emph{$\disc_n$-module operad} is any operad $\mathsf{O}$ with two colours $\sa$ and $\sm$, and whose space of operations $\mathsf{O}(\sa^k,\sm^l;\sa)$, for any integers $k,l\geq 0$, is equal to $\disc_n(\sa^k;\sa)$ for $l=0$ and empty for $l\geq 1$. Moreover, its operadic composition, restricted to the colour $\sa$, must be equal to that of $\disc_n$.
\end{defn}

Fix two integers $n\geq m\geq 0$. We now define five different $\disc_n$-module operads, the \emph{Swiss cheese operad} ($\swc_{m,n}$), as well as its \emph{extended} ($\eswc_{m,n}$), \emph{variant} ($\vswc_{m,n}$), \emph{concentric} ($\cswc_{m,n}$) and \emph{linear} ($\lswc_n$) cousins. In each case it will suffice to specify the spaces of operations $\mathsf{O}(\sa^k,\sm^l;\sm)$ for integers $k,l\geq 0$ and describe how to extend the operadic composition of $\disc_n$ to the colour $\sm$.

\begin{rmk}
The original Swiss cheese operad $\swc_{1,2}$ was introduced by Voronov \cite{Voronov1999}, inspired by constructions of Kontsevich \cite{Kontsevich1994Feynman-diagrams, Kontsevich2003deformation-quantization-Poisson}. The extended Swiss cheese operad $\eswc_{m,n}$ was introduced by Willwacher \cite{Willwacher2017}, who cites V.\ Turchin for its invention in codimension $1$ (when $n=m+1$). The variant Swiss cheese operad $\vswc_{m,n}$ was introduced by Idrissi \cite{Idrissi2018}. The linear Swiss cheese operad $\lswc_n$ was introduced (under the name $\swc_n$) by Krannich \cite[\S 2.1]{Krannich2019Homologicalstabilitytopological}.
\end{rmk}

\begin{defn}[\emph{Extended, variant and original Swiss cheese operads.}]
\label{d:swiss-cheese}
Let $\eswc_{m,n}(\sa^k,\sm^l;\sm)$ be the subspace of $\bar{F}_{k+l}(D^n)$ of configurations satisfying (i) and (ii) above, and
\begin{itemizeb}
\item[(iii)] $x_i \in D^m = D^n \cap (\bR^m \times \{0\}^{n-m})$ for $i\geq k+1$.
\end{itemizeb}
The space $\vswc_{m,n}(\sa^k,\sm^l;\sm)$ is the subspace of configurations additionally satisfying the condition
\begin{itemizeb}
\item[(iv)] $\mathrm{dist}(x_i,D^m) > r_i$ for $i\leq k$.
\end{itemizeb}
The space $\swc_{m,n}(\sa^k,\sm^l;\sm)$ is the subspace of configurations satisfying conditions (i)--(iii) and
\begin{itemizeb}
\item[(v)] $x_i \in \bR^m \times (r_i,\infty)^{n-m}$ for $i\leq k$.
\end{itemizeb}
(Note that condition (v) is stronger than condition (iv).) Interpreting such configurations again as little non-overlapping $n$-discs in $D^n$ -- see Figure \ref{fig:swiss-cheese}(b--d) -- we may extend the operadic composition of $\disc_n$ to each of these $2$-coloured operads by embedding $D^n$ into these smaller discs by translations and dilations.
\end{defn}

\begin{defn}[\emph{The concentric Swiss cheese operad.}]
\label{d:concentric-swiss-cheese}
The space of operations $\cswc_{m,n}(\sa^k,\sm^l;\sm)$ is empty unless $l=1$, in which case it is the subspace of $\bar{F}_{k+1}(D^n)$ of configurations satisfying (i), (ii) and (v) above, as well as
\begin{itemizeb}
\item[(vi)] $x_{k+1} = 0$.
\end{itemizeb}
Its operadic composition is defined as before, interpreting these configuration spaces as spaces of little non-overlapping discs -- see Figure \ref{fig:swiss-cheese}(e).
\end{defn}

\begin{defn}[\emph{The linear Swiss cheese operad.}]
\label{d:linear-swiss-cheese}
The space of operations $\lswc_n(\sa^k,\sm^l;\sm)$ is empty unless $l=1$, in which case it is the subspace of $\bar{F}_k([0,\infty) \times [-1,1]^{n-1}) \times [0,\infty)$ of configurations $((x_1,r_1),\ldots,(x_k,r_k))$ and $t\geq 0$ satisfying (ii) above, and
\begin{itemizeb}
\item[(vii)] $x_i \in [r_i,t-r_i] \times [r_i-1,1-r_i]^{n-1}$.
\end{itemizeb}
Its operadic composition is given by interpreting these configurations as configurations of little non-overlapping discs and embedding $D^n$ by translations and dilations, as well as placing copies of the cuboid $[0,t] \times [-1,1]^{n-1}$ end-to-end in the first coordinate direction and adding the corresponding values of $t$. See Figure \ref{fig:swiss-cheese}(f), and also \cite[Definition 2.1]{Krannich2019Homologicalstabilitytopological} for precise formulas for the operadic composition.
\end{defn}

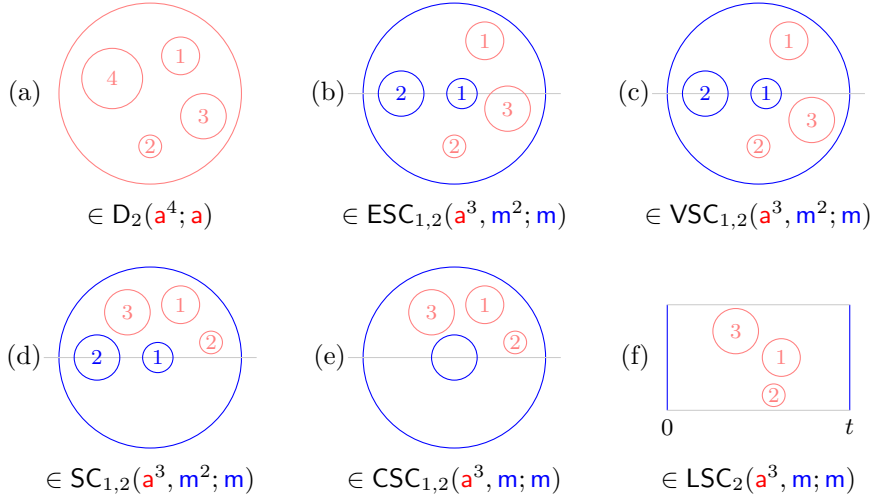
\begin{figure}[t]
\centering
\begin{tikzpicture}
[x=1mm,y=1mm]

\begin{scope}
\draw[red!50] (0,0) circle (12);
\draw[red!50] (4,5) node[font=\footnotesize]{$1$} circle (2.5);
\draw[red!50] (-5,2) node[font=\footnotesize]{$4$} circle (4);
\draw[red!50] (7,-3) node[font=\footnotesize]{$3$} circle (3);
\draw[red!50] (0,-7) node[font=\footnotesize]{$2$} circle (1.5);
\node at (0,-13) [anchor=north] {$\in\disc_2(\textcolor{red}{\sa}^4;\textcolor{red}{\sa})$};
\node at (-13,0) [anchor=east] {(a)};
\end{scope}

\begin{scope}[xshift=40mm]
\draw[blue] (0,0) circle (12);
\draw[black!20] (-14,0) -- (14,0);
\draw[red!50] (4,7) node[font=\footnotesize]{$1$} circle (2.5);
\draw[red!50] (7,-2) node[font=\footnotesize]{$3$} circle (3);
\draw[red!50] (0,-7) node[font=\footnotesize]{$2$} circle (1.5);
\draw[blue] (1,0)  node[font=\footnotesize]{$1$} circle (2);
\draw[blue] (-7,0) node[font=\footnotesize]{$2$} circle (3);
\node at (0,-13) [anchor=north] {$\in\eswc_{1,2}(\textcolor{red}{\sa}^3,\textcolor{blue}{\sm}^2;\textcolor{blue}{\sm})$};
\node at (-13,0) [anchor=east] {(b)};
\end{scope}

\begin{scope}[xshift=80mm]
\draw[blue] (0,0) circle (12);
\draw[black!20] (-14,0) -- (14,0);
\draw[red!50] (4,7) node[font=\footnotesize]{$1$} circle (2.5);
\draw[red!50] (7,-3.5) node[font=\footnotesize]{$3$} circle (3);
\draw[red!50] (0,-7) node[font=\footnotesize]{$2$} circle (1.5);
\draw[blue] (1,0)  node[font=\footnotesize]{$1$} circle (2);
\draw[blue] (-7,0) node[font=\footnotesize]{$2$} circle (3);
\node at (0,-13) [anchor=north] {$\in\vswc_{1,2}(\textcolor{red}{\sa}^3,\textcolor{blue}{\sm}^2;\textcolor{blue}{\sm})$};
\node at (-13,0) [anchor=east] {(c)};
\end{scope}

\begin{scope}[yshift=-35mm]
\draw[blue] (0,0) circle (12);
\draw[black!20] (-14,0) -- (14,0);
\draw[red!50] (4,7) node[font=\footnotesize]{$1$} circle (2.5);
\draw[red!50] (-3,6) node[font=\footnotesize]{$3$} circle (3);
\draw[red!50] (8,2) node[font=\footnotesize]{$2$} circle (1.5);
\draw[blue] (1,0) node[font=\footnotesize]{$1$} circle (2);
\draw[blue] (-7,0) node[font=\footnotesize]{$2$} circle (3);
\node at (0,-13) [anchor=north] {$\in\swc_{1,2}(\textcolor{red}{\sa}^3,\textcolor{blue}{\sm}^2;\textcolor{blue}{\sm})$};
\node at (-13,0) [anchor=east] {(d)};
\end{scope}

\begin{scope}[yshift=-35mm,xshift=40mm]
\draw[blue] (0,0) circle (12);
\draw[black!20] (-14,0) -- (14,0);
\draw[red!50] (4,7) node[font=\footnotesize]{$1$} circle (2.5);
\draw[red!50] (-3,6) node[font=\footnotesize]{$3$} circle (3);
\draw[red!50] (8,2) node[font=\footnotesize]{$2$} circle (1.5);
\draw[blue] (0,0) circle (3);
\node at (0,-13) [anchor=north] {$\in\cswc_{1,2}(\textcolor{red}{\sa}^3,\textcolor{blue}{\sm};\textcolor{blue}{\sm})$};
\node at (-13,0) [anchor=east] {(e)};
\end{scope}

\begin{scope}[yshift=-35mm,xshift=80mm]
\draw[blue] (-12,-7) -- (-12,7);
\draw[blue] (12,-7) -- (12,7);
\draw[black!20] (-12,-7) -- (12,-7);
\draw[black!20] (-12,7) -- (12,7);
\draw[red!50] (3,0) node[font=\footnotesize]{$1$} circle (2.5);
\draw[red!50] (-3,3.5) node[font=\footnotesize]{$3$} circle (3);
\draw[red!50] (2,-5) node[font=\footnotesize]{$2$} circle (1.5);
\node at (-12,-7) [anchor=north,font=\small] {$0$};
\node at (12,-7) [anchor=north,font=\small] {$t$};
\node at (0,-13) [anchor=north] {$\in\lswc_2(\textcolor{red}{\sa}^3,\textcolor{blue}{\sm};\textcolor{blue}{\sm})$};
\node at (-13,0) [anchor=east] {(f)};
\end{scope}

\end{tikzpicture}
\caption{Some operations of the little discs operad and of different flavours of Swiss cheese operads in dimensions $(1,2)$. Little blue (\textcolor{blue}{$\sm$} colour) discs are always centred on the $x$-axis. Little red (\textcolor{red}{$\sa$} colour) discs are unrestricted in $\eswc_{1,2}$ (except that they must not overlap each other or the little blue discs, of course). In $\vswc_{1,2}$, little red discs must be disjoint from the $x$-axis (little red discs are called \emph{aerial} and little blue discs \emph{terrestrial} in \cite{Idrissi2018}). In $\swc_{1,2}$, little red discs must lie in the upper half-disc (and we may choose to think of the little blue discs simply as little half-discs attached to the $x$-axis). In $\cswc_{1,2}$ the same conditions apply to little red discs, and there is now required to be exactly one little blue disc, centred at the origin. In $\lswc_2$, little red discs may lie anywhere in the rectangle. See Definitions \ref{d:little-discs}, \ref{d:swiss-cheese}, \ref{d:concentric-swiss-cheese} and \ref{d:linear-swiss-cheese} for the precise definitions.}
\label{fig:swiss-cheese}
\end{figure}

By definition, there are inclusions of operads
\begin{equation}
\label{eq:inclusions-operads}
\cswc_{m,n} \longhookrightarrow \swc_{m,n} \longhookrightarrow \vswc_{m,n} \longhookrightarrow \eswc_{m,n}
\end{equation}
that restrict to the identity map of $\disc_n$ on the $\sa$ colour. The natural inclusions of discs $D^n \hookrightarrow D^{n+1}$ and cubes $[-1,1]^{n-1} \hookrightarrow [-1,1]^n$ induce dimension-increasing inclusions (see Figure \ref{fig:dimension-increasing})
\begin{equation}
\label{eq:dimension-increasing}
\lswc_n \too \lswc_{n+1} \qquad\text{and}\qquad \xswc_{m,n} \too \xswc_{m,n+1}
\end{equation}
(for $\mathsf{X} \in \{ \mathsf{E},\mathsf{V},\varnothing,\mathsf{C} \}$) and similarly
\begin{equation}
\label{eq:dimension-increasing-2}
\xswc_{m-1,n} \too \xswc_{m,n}
\end{equation}
(for $\mathsf{X} \in \{ \mathsf{E},\mathsf{V},\varnothing,\mathsf{C} \}$), which commute with the ``flavour-changing'' inclusions \eqref{eq:inclusions-operads}. The connection between the four operads \eqref{eq:inclusions-operads} and the linear Swiss cheese operads is given by the following lemma. Recall that a \emph{weak equivalence of operads} is a map of operads $\mathsf{O} \to \mathsf{P}$ such that the maps of spaces $\mathsf{O}(\sa^k,\sm^l;\sa) \to \mathsf{P}(\sa^k,\sm^l;\sa)$ and $\mathsf{O}(\sa^k,\sm^l;\sm) \to \mathsf{P}(\sa^k,\sm^l;\sm)$ are all weak equivalences.

\begin{figure}[t]
\centering
\begin{tikzpicture}
[x=1mm,y=1mm]
\begin{scope}
\draw[blue] (0,0) circle (12);
\draw[black!20] (-14,0) -- (14,0);
\draw[red!50] (4,7) node[font=\footnotesize]{$1$} circle (2.5);
\draw[red!50] (7,-2) node[font=\footnotesize]{$3$} circle (3);
\draw[red!50] (0,-7) node[font=\footnotesize]{$2$} circle (1.5);
\draw[blue] (1,0)  node[font=\footnotesize]{$1$} circle (2);
\draw[blue] (-7,0) node[font=\footnotesize]{$2$} circle (3);
\end{scope}
\node at (25,0) {$\longmapsto$};
\begin{scope}[xshift=50mm]
\draw[blue] (0,0) circle (12);
\draw[blue] (-12,0) arc (180:360:12 and 4.8);
\draw[blue,densely dotted] (12,0) arc (0:180:12 and 4.8);
\draw[black!20] (-14,0) -- (14,0);
\draw[red!50] (4,7) circle (2.5);
\draw[red!50] (1.5,7) arc (180:360:2.5 and 1);
\draw[red!50,densely dotted] (6.5,7) arc (0:180:2.5 and 1);
\draw[red!50] (7,-2) circle (3);
\draw[red!50] (4,-2) arc (180:360:3 and 1.2);
\draw[red!50,densely dotted] (10,-2) arc (0:180:3 and 1.2);
\draw[red!50] (0,-7) circle (1.5);
\draw[red!50] (-1.5,-7) arc (180:360:1.5 and 0.6);
\draw[red!50,densely dotted] (1.5,-7) arc (0:180:1.5 and 0.6);
\draw[blue] (1,0) circle (2);
\draw[blue] (-1,0) arc (180:360:2 and 0.8);
\draw[blue,densely dotted] (3,0) arc (0:180:2 and 0.8);
\draw[blue] (-7,0) circle (3);
\draw[blue] (-10,0) arc (180:360:3 and 1.2);
\draw[blue,densely dotted] (-4,0) arc (0:180:3 and 1.2);
\end{scope}
\end{tikzpicture}
\caption{The dimension-increasing inclusion $\eswc_{1,2}(\textcolor{red}{\sa}^3,\textcolor{blue}{\sm}^2;\textcolor{blue}{\sm}) \to \eswc_{1,3}(\textcolor{red}{\sa}^3,\textcolor{blue}{\sm}^2;\textcolor{blue}{\sm})$. The numbering of the $3$-discs has been omitted on the right-hand side to avoid overloading the diagram.}
\label{fig:dimension-increasing}
\end{figure}

\begin{lem}
\label{l:linear-to-concentric}
For any $0\leq m\leq n$ there is a map of operads\footnote{Strictly speaking, the map $\iota_{m,n}$ that we construct is only a homotopy map of operads: it commutes only up to homotopy with operadic composition. After applying any homotopy-invariant functor of spaces, such as homology, this of course induces a strict map of operads.}
\begin{equation}
\iota_{m,n} \colon \label{eq:linear-to-concentric}
\lswc_n \too \cswc_{m,n}
\end{equation}
commuting with \eqref{eq:dimension-increasing} and \eqref{eq:dimension-increasing-2} in the sense that the following diagrams commute up to homotopy:
\begin{equation}
\label{eq:linear-to-concentric-diagrams}
\centering
\begin{split}
\begin{tikzpicture}
[x=1mm,y=1mm]
\node (tl) at (0,12) {$\lswc_n$};
\node (tr) at (30,12) {$\cswc_{m,n}$};
\node (bl) at (0,0) {$\lswc_{n+1}$};
\node (br) at (30,0) {$\cswc_{m,n+1}$};
\draw[->] (tl) to node[above,font=\small]{$\iota_{m,n}$} (tr);
\draw[->] (bl) to node[below,font=\small]{$\iota_{m,n+1}$} (br);
\draw[->] (tl) to node[left,font=\small]{\eqref{eq:dimension-increasing}} (bl);
\draw[->] (tr) to node[right,font=\small]{\eqref{eq:dimension-increasing}} (br);
\begin{scope}[xshift=60mm]
\node (l) at (0,6) {$\lswc_n$};
\node (tr) at (30,12) {$\cswc_{m-1,n}$};
\node (br) at (30,0) {$\cswc_{m,n}$};
\draw[->] (l) to node[above,font=\small]{$\iota_{m-1,n}$} (tr);
\draw[->] (l) to node[below,font=\small]{$\iota_{m,n}$} (br);
\draw[->] (tr) to node[right,font=\small]{\eqref{eq:dimension-increasing-2}} (br);
\end{scope}
\end{tikzpicture}
\end{split}
\end{equation}
When $0\leq m\leq n-1$, the map $\iota_{m,n}$ is a weak equivalence of operads, so $\eqref{eq:dimension-increasing-2} \colon \cswc_{m-1,n} \to \cswc_{m,n}$ is also a weak equivalence of operads in this range.
\end{lem}
\begin{proof}
The map $\iota_{m,n}$ is defined as pictured (for the cases $(m,n)=(1,2)$ and $(m,n)=(0,2)$) in Figure \ref{fig:iota}: a configuration in the rectangle is mapped via the indicated mapping of the rectangle onto a segment of the annulus, together with an appropriate rescaling of the labels in $(0,\infty)$ of the configuration points (interpreted as radii). If the rescaling is done carefully, this gives a well-defined map of the labelled configuration spaces underlying the operads, which commutes on the nose with the operadic composition of the form $(\sa^k,\sm;\sm) \times (\sa^l,\sm;\sm) \to (\sa^{k+l},\sm;\sm)$, and commutes up to homotopy with the operadic composition of the form $(\sa^k,\sm;\sm) \times (\sa^l;\sa) \to (\sa^{k+l-1},\sm;\sm)$. The fact that $\iota_{m,n}$ is a weak equivalence of operads for $m\leq n-1$ follows from the fact that this mapping is a homeomorphism from the rectangle onto a proper segment of the annulus (when $m=n$ it is a quotient map from the rectangle onto the whole annulus). The left-hand square of \eqref{eq:linear-to-concentric-diagrams} commutes on the nose, and the right-hand triangle of \eqref{eq:linear-to-concentric-diagrams} commutes up to a ``stretching'' homotopy that is also pictured in Figure \ref{fig:iota} (in the case $(m,n) = (1,2)$).
\end{proof}

\begin{figure}[t]
\centering
\begin{tikzpicture}
[x=1mm,y=1mm]

\fill[green!20] (-12,-7) rectangle (12,7);
\draw[blue] (-12,-7) -- (-12,7);
\draw[blue] (12,-7) -- (12,7);
\draw[black!20] (-12,-7) -- (12,-7);
\draw[black!20] (-12,7) -- (12,7);
\draw[green!50!black] (-0.5,6)--(0.5,7)--(-0.5,8);
\draw[green!50!black] (-0.1,-6)--(0.9,-7)--(-0.1,-8);
\draw[green!50!black] (-0.9,-6)--(0.1,-7)--(-0.9,-8);
\node at (-12,-7) [anchor=north,font=\small] {$0$};
\node at (12,-7) [anchor=north,font=\small] {$t$};

\begin{scope}[xshift=100mm]
\fill[green!20] (0,12) arc (90:0:12) -- (3,0) arc (0:90:3) -- cycle;
\draw[blue] (0,0) circle (12);
\draw[black!20] (-14,0) -- (14,0);
\draw[black!20] (0,-14) -- (0,14);
\draw[blue] (0,0) circle (3);
\draw[green!50!black] (-1,8)--(0,7)--(1,8);
\node at (4,0) [anchor=north,font=\footnotesize] {$e^{-t}$};
\draw[green!50!black] (8.4,1)--(7.4,0)--(8.4,-1);
\draw[green!50!black] (7.6,1)--(6.6,0)--(7.6,-1);
\end{scope}

\begin{scope}[xshift=100mm,yshift=-40mm]
\fill[green!20] (0,12) arc (90:0:12) -- (3,0) arc (0:90:3) -- cycle;
\draw[blue] (0,0) circle (12);
\draw[black!20] (-14,0) -- (14,0);
\draw[black!20] (0,3) -- (0,12);
\draw[blue] (0,0) circle (3);
\draw[green!50!black] (-1,8)--(0,7)--(1,8);
\node at (4,0) [anchor=north,font=\footnotesize] {$e^{-t}$};
\draw[green!50!black] (8.4,1)--(7.4,0)--(8.4,-1);
\draw[green!50!black] (7.6,1)--(6.6,0)--(7.6,-1);
\end{scope}

\begin{scope}[xshift=30mm,yshift=-40mm]
\fill[green!20] (-12,0) arc (180:0:12) -- (3,0) arc (0:180:3) -- cycle;
\draw[blue] (0,0) circle (12);
\draw[black!20] (-14,0) -- (14,0);
\draw[blue] (0,0) circle (3);
\draw[green!50!black] (-8,1)--(-7,0)--(-8,-1);
\node at (4,0) [anchor=north,font=\footnotesize] {$e^{-t}$};
\draw[green!50!black] (8.4,1)--(7.4,0)--(8.4,-1);
\draw[green!50!black] (7.6,1)--(6.6,0)--(7.6,-1);
\end{scope}

\begin{scope}[xshift=65mm,yshift=-40mm]
\fill[green!20] (3,0) arc (0:135:3) -- (135:12) arc (135:0:12) -- cycle;
\draw[blue] (0,0) circle (12);
\draw[black!20] (-14,0) -- (14,0);
\draw[blue] (0,0) circle (3);
\draw[black!50,<->] (175:5.5) arc (175:95:5.5);
\draw[black!50,<->] (175:7.5) arc (175:95:7.5);
\draw[black!50,<->] (175:9.5) arc (175:95:9.5);
\end{scope}

\draw[|->] (17,0) to node[above,font=\small]{$\iota_{0,2}$} (83,0);
\draw[|->] (5,-12) to node[below left=-1mm,font=\small]{$\iota_{1,2}$} (18,-28);
\draw[|->] (100,-17) to node[right,font=\small]{\eqref{eq:dimension-increasing-2}} (100,-23);
\node at (47.5,-40) {$\sim$};
\node at (82.5,-40) {$\sim$};

\end{tikzpicture}
\caption{The map $\iota_{m,n}$ and the homotopy-commutativity of the right-hand triangle of \eqref{eq:linear-to-concentric-diagrams}.}
\label{fig:iota}
\end{figure}
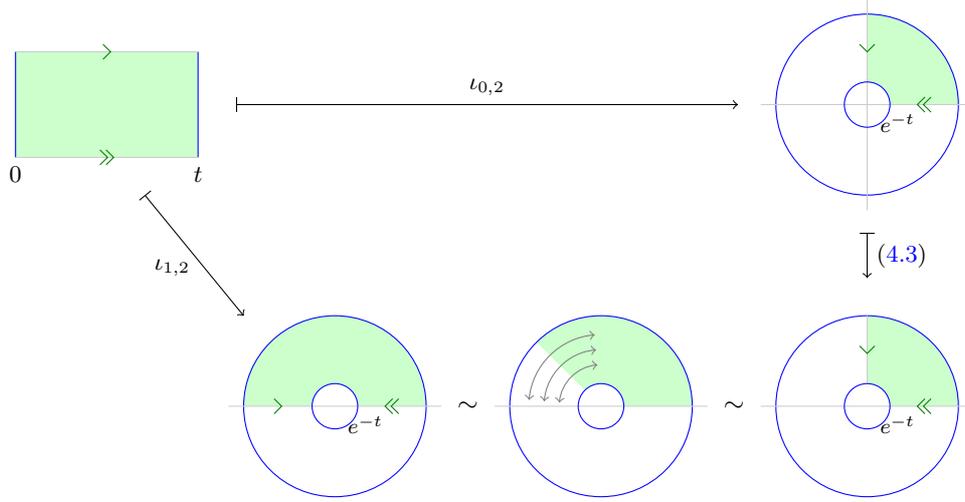

\paragraph{Algebras.}
Recall that, in general, an \emph{algebra} over a two-coloured operad $\mathsf{O}$ (with colours $\sa$ and $\sm$) consists of a pair of spaces $(X_\sa , X_\sm)$ together with maps
\[
\mathsf{O}(\sa^k , \sm^l ; \sa) \times (X_\sa)^k \times (X_\sm)^l \too X_\sa \qquad\text{and}\qquad \mathsf{O}(\sa^k , \sm^l ; \sm) \times (X_\sa)^k \times (X_\sm)^l \too X_\sm
\]
satisfying appropriate associativity axioms. In the following, we will write $Y = X_\sa$ (corresponding to the red colour in Figure \ref{fig:swiss-cheese}) and $X = X_\sm$ (corresponding to the blue colour in Figure \ref{fig:swiss-cheese}).

\paragraph{Algebras over Swiss-cheese operads.}
Algebras over $\disc_n$ are \emph{$E_n$-algebras}, by definition. Now, the restriction of $\swc_{m,n}$ to the $\sa$ colour is $\disc_n$ and its restriction to the $\sm$ colour is isomorphic to $\disc_m$, so an algebra over $\swc_{m,n}$ is a pair $(X,Y)$ consisting of an $E_n$-algebra $Y$, an $E_m$-algebra $X$, together with an additional structure intertwining them. The same remarks apply also to $\vswc_{m,n}$ and $\eswc_{m,n}$, so algebras over each of these three operads consist of an $E_n$-algebra acting on an $E_m$-algebra, where the precise meaning of ``acting'' depends on the flavour.

The restriction of the \emph{concentric} Swiss cheese operad $\cswc_{m,n}$ to the $\sa$ colour is again $\disc_n$, but now its restriction to the $\sm$ colour is trivial (it has no $l$-ary operations except when $l=1$, and its space of $1$-ary operations is contractible). Algebras over $\cswc_{m,n}$ are thus $E_m$-\emph{modules} over $E_n$-algebras, without any $E_m$-\emph{algebra} structure on the module. By Lemma \ref{l:linear-to-concentric}, we have $\cswc_{m,n} \simeq \cswc_{m',n}$ for any $m,m' \leq n-1$, at least after applying a homotopy-invariant functor of spaces, so the notions of ``\emph{$E_m$-module over an $E_n$-algebra}'', for fixed $n$, are equivalent for all $m \in \{0,1,\ldots,n-1\}$, and are equivalently encoded by the \emph{linear} Swiss cheese operad $\lswc_n$. On the other hand, the notion of ``\emph{$E_n$-module over an $E_n$-algebra}'' is stronger, and not encoded by the linear Swiss cheese operad.

\paragraph{Linear Swiss cheese structures on configuration-section spaces.}
Below, we define certain homotopy equivalent models $\dot{C}(M) \simeq C(\mathring{M})$ and $\cgammad{}{c,D}{M}{\xi} \simeq \cgamma{}{c,D}{M}{\xi}$ for configuration spaces and configuration-section spaces (see Definitions \ref{d:cdot} and \ref{d:cgammadot}, and Lemma \ref{l:equivalent-models}). Definition \ref{d:cgammadot} also explains how a singularity condition $c \subseteq \pi_0(\Sigma(\xi))$ determines a subset $c_D \subseteq [S^{d-1},X]$. We use a slight abuse of notation by writing $\cmap{}{c,\partial}{D^d}{X}$ instead of $\cmap{}{c_D,\partial}{D^d}{X}$. We also abbreviate $\partial D^d$ to $\partial$ and write $\frac12\partial$ for one hemisphere of $\partial D^d$.

\begin{prop}
\label{p:e0-modules}
We have the following linear Swiss cheese structures on configuration and configuration-section spaces:
\begin{compactenum}
\renewcommand{\theenumi}{\textup{(\roman{enumi})}}
\item\label{lsc1} $(\dot{C}(M),C(\mathring{D}^d))$ is an algebra over $\lswc_d$,
\item\label{lsc2} $(\cgammad{}{c,D}{M}{\xi},\cmap{}{c,\partial}{D^d}{X})$ is an algebra over $\lswc_d$,
\item\label{lsc3} $(\cgammad{}{c,D}{M}{\xi},\cmap{}{c,\frac12\partial}{D^d}{X})$ is an algebra over $\lswc_{d-1}$.
\end{compactenum}
Moreover, the maps of pairs
\begin{equation}
\label{lsc-maps}
(\cgammad{}{c,D}{M}{\xi},\cmap{}{c,\partial}{D^d}{X}) \hookrightarrow (\cgammad{}{c,D}{M}{\xi},\cmap{}{c,\frac12\partial}{D^d}{X}) \to (\dot{C}(M),C(\mathring{D}^d))
\end{equation}
are maps of $\lswc_{d-1}$-algebras, and their composition is a map of $\lswc_d$-algebras. Here, the first map is the inclusion and the second map sends a configuration-section to its underlying configuration, forgetting the section.
\end{prop}

Point \ref{lsc1} is essentially \cite[Lemma 5.1]{Krannich2019Homologicalstabilitytopological}, and part of points \ref{lsc2} and \ref{lsc3} -- the $\disc_d$-algebra structure, but not the $\lswc_d$- and $\lswc_{d-1}$-algebra structures -- is \cite[Propositions 2.6.1 and 2.6.2]{EllenbergVenkateshWesterland2012HomologicalstabilityHurwitz}. Before proving Proposition \ref{p:e0-modules}, we first define the appropriate models for configuration and configuration-section spaces referred to above.

\begin{defn}
\label{d:Mhat}
Let $M$ be a manifold equipped with an embedded codimension-zero disc $D \subseteq \partial M$ in its boundary and a collar neighbourhood of $\partial M$, namely an embedding $b \colon (-\infty,0] \times \partial M \to M$ such that $b(0,x)=x$ for all $x \in \partial M$. Define:
\[
\hat{M} = M \cup_b (\bR \times D) \qquad\text{and}\qquad \hat{M}_r = M \cup_b ((-\infty,r] \times D)
\]
for $r \in [0,\infty)$. Diagrammatically, $\hat{M}$ may be seen as follows, where $M$ is green and $\bR \times D$ is blue (and hence $(-\infty,0] \times D$, which is identified with $b((-\infty,0] \times D) \subseteq M$, is turquoise).
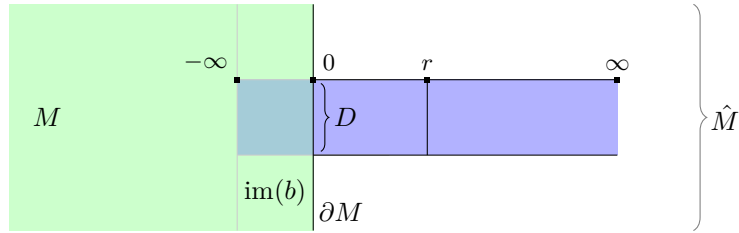
\begin{center}
\begin{tikzpicture}
[x=1mm,y=1mm]
\path[use as bounding box] (0,-15) rectangle (80,15);

\fill[green!20] (0,-15) rectangle (40,15);
\fill[green!30] (30,-5) rectangle (40,5);
\fill[blue!40,opacity=0.5] (30,-5) rectangle (40,5);
\fill[blue!30] (40,-5) rectangle (80,5);
\draw[black!20] (30,-15) -- (30,15);
\draw[black!20] (30,-5) -- (50,-5);
\draw[black!20] (30,5) -- (50,5);
\draw (40,-15) -- (40,15);
\draw (40,5) -- (80,5);
\draw (40,-5) -- (80,-5);
\draw (55,-5) -- (55,5);
\node at (30,5) [fill,inner sep=1pt] {};
\node at (40,5) [fill,inner sep=1pt] {};
\node at (55,5) [fill,inner sep=1pt] {};
\node at (80,5) [fill,inner sep=1pt] {};

\node at (5,0) {$M$};
\node at (35,-10) {$\mathrm{im}(b)$};
\draw[decorate,decoration={brace,amplitude=3pt,mirror}] (41,-4.5) -- (41,4.5);
\node at (41.5,0) [anchor=west] {$D$};
\node at (39.5,-15) [anchor=south west] {$\partial M$};
\node at (30,5) [font=\small,anchor=south east] {$-\infty$};
\node at (40,5) [font=\small,anchor=south west] {$0$};
\node at (55,5) [font=\small,anchor=south] {$r$};
\node at (80,5) [font=\small,anchor=south] {$\infty$};
\draw[black!50,decorate,decoration={brace,amplitude=5pt,mirror}] (90,-15) -- (90,15);
\node at (91,0) [anchor=west] {$\hat{M}$};

\end{tikzpicture}
\end{center}
\end{defn}

\begin{defn}
\label{d:cdot}
Let $\dot{C}_k(M)$ be the subspace of $C_k(\hat{M}) \times (0,\infty)$ of pairs $(z,t)$ with $z \subseteq \mathrm{int}(\hat{M}_t)$, and define
\[
\dot{C}(M) = \bigsqcup_{k \in \bN} \dot{C}_k(M).
\]
\end{defn}

\begin{defn}
\label{d:cgammadot}
Let $M$, $D$ and $b$ be as in Definition \ref{d:Mhat}. Also choose a bundle $\xi \colon E \to M$, a subset $c \subseteq \pi_0(\Sigma(\xi))$ (\cf Definition \ref{s:singularity}) and a section $s_D$ of $\xi|_D$. Denote by $*$ the centre of the disc $D \subseteq \partial M$, let $X = \xi^{-1}(*)$ and choose a point $x_0 \in X$. Choose a trivialisation $\varphi \colon \xi|_D \cong D \times X$ such that $s_D$ corresponds to the constant section of $D \times X$ at $x_0$. Using $\varphi$, glue $\xi$ along $D = D \times \{0\}$ to the trivial $X$-bundle over $[0,\infty) \times D$ to obtain a bundle on $\hat{M}$, which we denote by $\hat{\xi}$. We then have inclusions of bundles $\xi|_D \hookrightarrow \xi = \hat{\xi}|_M \hookrightarrow \hat{\xi}$, which induce inclusions of covering spaces $\eta(\xi|_D) \hookrightarrow \eta(\xi) \hookrightarrow \eta(\hat{\xi})$. Looking at the components of their total spaces, these induce maps
\[
[S^{d-1},X] \cong \pi_0(\Sigma(\xi|_D)) \longtwoheadrightarrow \pi_0(\Sigma(\xi)) \xrightarrow{\;\cong\;} \pi_0(\Sigma(\hat{\xi})),
\]
where the left-hand bijection is induced by the trivialisation $\varphi$ of $\xi|_D$ and the fact that $\eta$ of a trivial bundle over an orientable $d$-manifold with fibre $X$ is the trivial covering space with fibre $[S^{d-1},X]$ (see Example \ref{eg:trivial-bundle}). The right-hand map is clearly a bijection, since we extended $\xi$ to $\hat{\xi}$ by gluing it to a trivial bundle. The middle map is always surjective, and it is injective if and only if $\eta(\xi) \colon \Sigma(\xi) \to \mathring{M}$ is a trivial covering. By taking images and pre-images, the subset $c \subseteq \pi_0(\Sigma(\xi))$ determines subsets
\[
c_D \subseteq [S^{d-1},X] \qquad\text{and}\qquad \hat{c} \subseteq \pi_0(\Sigma(\hat{\xi})).
\]
Define $\cgammad{k}{c,D}{M}{\xi}$ to be the subspace of $\cgamma{k}{\hat{c}}{\hat{M}}{\hat{\xi}} \times (0,\infty)$ of elements $(z,s,t)$ consisting of a real number $t > 0$ and a configuration-section $(z,s)$ such that
\begin{itemizeb}
\item $z \subseteq \mathrm{int}(\hat{M}_t)$,
\item $s$ is the constant section at $x_0$ on the subspace $[t,\infty) \times D \subseteq \hat{M}$,
\end{itemizeb}
and let
\[
\cgammad{}{c,D}{M}{\xi} = \bigsqcup_{k \in \bN} \cgammad{k}{c,D}{M}{\xi}.
\]
We will also use the following slight abuse of notation:
\[
\cmap{}{c,\partial}{D^d}{X} = \bigsqcup_{k \in \bN} \cmap{k}{c_D,\partial D^d}{D^d}{X} \quad\text{,}\quad \cmap{}{c,\frac12\partial}{D^d}{X} = \bigsqcup_{k \in \bN} \cmap{k}{c_D,\partial_0 D^d}{D^d}{X},
\]
where $\partial_0 D^d = \partial D^d \cap \{ x_d \leq 0 \}$ is the southern hemisphere of $\partial D^d$.
\end{defn}

\begin{lem}
\label{l:equivalent-models}
For each $k$, there are natural embeddings
\[
C_k(M) \longhookrightarrow \dot{C}_k(M) \qquad\text{and}\qquad \cgamma{k}{c,D}{M}{\xi} \longhookrightarrow \cgammad{k}{c,D}{M}{\xi}
\]
admitting deformation retractions. In particular, $C(M) \simeq \dot{C}(M)$ and $\cgamma{}{c,D}{M}{\xi} \simeq \cgammad{}{c,D}{M}{\xi}$.
\end{lem}
\begin{proof}
We will prove this just for the second embedding (for configuration-section spaces), since the proof for the first embedding (for configuration spaces) is identical, forgetting the sections and considering just configurations of points. The embedding is defined by sending a configuration-section $(z,s)$ of $\xi$ to the element $(z,\hat{s},1)$ of $\cgammad{k}{c,D}{M}{\xi}$, where $\hat{s}$ is the section of $\hat{\xi}|_{\hat{M} \smallsetminus z}$ given by extending $s$ by the constant section at $x_0$ on $D \times [0,\infty)$.

We now construct a deformation retraction of $\cgammad{k}{c,D}{M}{\xi}$ onto the image of this embedding. First, we choose a family of self-embeddings $\varphi \colon [0,1] \times [0,\infty) \to \mathrm{Emb}(\hat{M},\hat{M})$ with the properties that (i) $\varphi(0,t) = \mathrm{id}$ and (ii) $\varphi(1,t)(\hat{M}_t) \subseteq M$. Namely, we define $\varphi(u,t)$ to be the identity outside of $N = (\partial M \times (-\infty,0]) \cup (D \times [0,\infty))$, and for a point $(x,v)$ of $N$, define $\varphi(u,t)(x,v) = (x,v-ut)$. This may be lifted to a family of self-embeddings $\widetilde{\varphi} \colon [0,1] \times [0,\infty) \to \mathrm{Emb}(\hat{\xi},\hat{\xi})$ such that $\widetilde{\varphi}(u,t)$ covers $\varphi(u,t)$ and $\widetilde{\varphi}(0,t) = \mathrm{id}$, using the observation that the obvious projection $p \colon N \to \partial M$ is a homotopy equivalence, and so we may identify $\hat{\xi}|_N$ with $p^*(\hat{\xi}|_{\partial M})$. The deformation retraction is defined by sending $(z,s,t)$, at time $u \in [0,1]$, to:
\[
\bigl( \varphi(u,t)(z) \, , \, \widetilde{\varphi}(u,t) \circ s \circ \varphi(u,t)^{-1} \, , \, (1-u)t+u \bigr) .
\]
The section $\widetilde{\varphi}(u,t) \circ s \circ \varphi(u,t)^{-1}$ is a priori defined only outside of $(\partial M \smallsetminus D) \times (-ut,0]$; we extend it to this region by defining it to be constant in the collar direction. This uses the identification $\hat{\xi}|_N = p^*(\hat{\xi}|_{\partial M})$ and the fact that the configuration $\varphi(u,t)(z)$ is disjoint from $(\partial M \smallsetminus D) \times (-ut,0]$.
\end{proof}

\begin{proof}[Proof of Proposition \ref{p:e0-modules}]
The $\lswc_d$ action on $(\cgammad{}{c,D}{M}{\xi},\cmap{}{c,\partial}{D^d}{X})$ is determined by maps
\begin{align}
\label{eq:lscd-action}
\begin{split}
\lswc_d(\sa^k;\sa) \times (\cmap{}{c,\partial}{D^d}{X})^k &\too \cmap{}{c,\partial}{D^d}{X} \\
\lswc_d(\sa^k,\sm;\sm) \times (\cmap{}{c,\partial}{D^d}{X})^k \times \cgammad{}{c,D}{M}{\xi} &\too \cgammad{}{c,D}{M}{\xi},
\end{split}
\end{align}
which are defined by picture in Figure \ref{fig:lscd-action}. The $\lswc_{d-1}$ action on $(\cgammad{}{c,D}{M}{\xi},\cmap{}{c,\frac12\partial}{D^d}{X})$ is determined by maps
\begin{align}
\label{eq:lscd-1-action}
\begin{split}
\lswc_{d-1}(\sa^k;\sa) \times (\cmap{}{c,\frac12\partial}{D^d}{X})^k &\too \cmap{}{c,\frac12\partial}{D^d}{X} \\
\lswc_{d-1}(\sa^k,\sm;\sm) \times (\cmap{}{c,\frac12\partial}{D^d}{X})^k \times \cgammad{}{c,D}{M}{\xi} &\too \cgammad{}{c,D}{M}{\xi},
\end{split}
\end{align}
which are defined by picture in Figure \ref{fig:lscd-1-action}. The fact that these are well-defined actions of linear Swiss cheese operads may be verified easily from the construction. This gives the action of $\lswc_d$ for case \ref{lsc2} and of $\lswc_{d-1}$ for case \ref{lsc3}. The action of $\lswc_d$ for case \ref{lsc1} is identical to that for case \ref{lsc2}, forgetting all sections and remembering just the configurations (compare the bottom line of Figure \ref{fig:lscd-action} with \cite[Figure 4]{Krannich2019Homologicalstabilitytopological}). The statements that the maps \eqref{lsc-maps} respect the $\lswc_{d-1}$ structures (and that their composition respects the $\lswc_d$ structures) is also clear from the construction.
\end{proof}

\begin{figure}[t]
\centering
\begin{tikzpicture}
[x=1mm,y=1mm]

\begin{scope}[xshift=6mm]
\draw[red!50] (0,0) circle (12);
\draw[red!50] (4,5) node[font=\footnotesize]{$1$} circle (2.5);
\draw[red!50] (-5,2) node[font=\footnotesize]{$3$} circle (4);
\draw[red!50] (7,-3) node[font=\footnotesize]{$2$} circle (3);
\end{scope}

\begin{scope}[xshift=30mm]
\fill[green!15] (0,0) circle (8);
\draw[green!50!black] (0,0) circle (8);
\node[green!50!black,fill,circle,inner sep=1pt] at (2,0) {};
\node[green!50!black,fill,circle,inner sep=1pt] at (4,-5) {};
\node[green!50!black,fill,circle,inner sep=1pt] at (-6,-1) {};
\node[green!50!black,fill,circle,inner sep=1pt] at (-1,4) {};
\end{scope}

\begin{scope}[xshift=50mm]
\fill[yellow!10] (0,0) circle (8);
\draw[green!50!black] (0,0) circle (8);
\node[green!50!black,fill,circle,inner sep=1pt] at (2,1) {};
\node[green!50!black,fill,circle,inner sep=1pt] at (-4,-4) {};
\end{scope}

\begin{scope}[xshift=70mm]
\fill[blue!20] (0,0) circle (8);
\draw[green!50!black] (0,0) circle (8);
\node[green!50!black,fill,circle,inner sep=1pt] at (-4,5) {};
\node[green!50!black,fill,circle,inner sep=1pt] at (1,1) {};
\node[green!50!black,fill,circle,inner sep=1pt] at (3,-5) {};
\end{scope}

\node at (93,0) {$\longmapsto$};

\begin{scope}[xshift=120mm]
\fill[black!3] (0,0) circle (12);
\draw[green!50!black] (0,0) circle (12);
  \begin{scope}[xshift=4mm,yshift=5mm,scale=0.3125]
  \fill[green!15] (0,0) circle (8);
  \draw[black!20] (0,0) circle (8);
  \node[green!50!black,fill,circle,inner sep=1pt] at (2,0) {};
  \node[green!50!black,fill,circle,inner sep=1pt] at (4,-5) {};
  \node[green!50!black,fill,circle,inner sep=1pt] at (-6,-1) {};
  \node[green!50!black,fill,circle,inner sep=1pt] at (-1,4) {};
  \end{scope}
  \begin{scope}[xshift=7mm,yshift=-3mm,scale=0.375]
  \fill[yellow!10] (0,0) circle (8);
  \draw[black!20] (0,0) circle (8);
  \node[green!50!black,fill,circle,inner sep=1pt] at (2,1) {};
  \node[green!50!black,fill,circle,inner sep=1pt] at (-4,-4) {};
  \end{scope}
  \begin{scope}[xshift=-5mm,yshift=2mm,scale=0.5]
  \fill[blue!20] (0,0) circle (8);
  \draw[black!20] (0,0) circle (8);
  \node[green!50!black,fill,circle,inner sep=1pt] at (-4,5) {};
  \node[green!50!black,fill,circle,inner sep=1pt] at (1,1) {};
  \node[green!50!black,fill,circle,inner sep=1pt] at (3,-5) {};
  \end{scope}
\end{scope}

\begin{scope}[yshift=-30mm,xshift=6mm,scale=0.75]
\draw[blue] (-12,-10) -- (-12,10);
\draw[blue] (12,-10) -- (12,10);
\draw[black!20] (-12,-10) -- (12,-10);
\draw[black!20] (-12,10) -- (12,10);
\draw[red!50] (8,2.5) node[font=\footnotesize]{$1$} circle (2.5);
\draw[red!50] (-5,4.5) node[font=\footnotesize]{$3$} circle (3);
\draw[red!50] (1,-4) node[font=\footnotesize]{$2$} circle (4);
\node at (-12,-10) [anchor=north,font=\small] {$0$};
\node at (12,-10) [anchor=north,font=\small] {$t$};
\end{scope}

\begin{scope}[yshift=-30mm,xshift=23mm,scale=0.5]
\fill[green!15] (0,0) circle (8);
\draw[green!50!black] (0,0) circle (8);
\node[green!50!black,fill,circle,inner sep=1pt] at (2,0) {};
\node[green!50!black,fill,circle,inner sep=1pt] at (4,-5) {};
\node[green!50!black,fill,circle,inner sep=1pt] at (-6,-1) {};
\node[green!50!black,fill,circle,inner sep=1pt] at (-1,4) {};
\end{scope}

\begin{scope}[yshift=-30mm,xshift=35mm,scale=0.5]
\fill[yellow!10] (0,0) circle (8);
\draw[green!50!black] (0,0) circle (8);
\node[green!50!black,fill,circle,inner sep=1pt] at (2,1) {};
\node[green!50!black,fill,circle,inner sep=1pt] at (-4,-4) {};
\end{scope}

\begin{scope}[yshift=-30mm,xshift=47mm,scale=0.5]
\fill[blue!20] (0,0) circle (8);
\draw[green!50!black] (0,0) circle (8);
\node[green!50!black,fill,circle,inner sep=1pt] at (-4,5) {};
\node[green!50!black,fill,circle,inner sep=1pt] at (1,1) {};
\node[green!50!black,fill,circle,inner sep=1pt] at (3,-5) {};
\end{scope}

\begin{scope}[yshift=-30mm,xshift=55mm,scale=0.75]
\fill[YellowOrange!10] (0,-15) rectangle (10,15);
\fill[YellowOrange!10] (10,-10) rectangle (20,10);
\fill[black!3] (20,-10) rectangle (30,10);
\draw[green!50!black] (10,-15) -- (10,-10);
\draw[black!20] (10,-10) -- (10,10);
\draw[green!50!black] (10,10) -- (10,15);
\draw[green!50!black] (10,10) -- (30,10);
\draw[green!50!black] (10,-10) -- (30,-10);
\draw[black!20] (20,-10) -- (20,10);
\node[green!50!black,fill,circle,inner sep=1pt] at (16,-7) {};
\node[green!50!black,fill,circle,inner sep=1pt] at (12,5) {};
\node[green!50!black,fill,circle,inner sep=1pt] at (2.5,-11) {};
\node[green!50!black,fill,circle,inner sep=1pt] at (7,-2) {};
\node[green!50!black,fill,circle,inner sep=1pt] at (5,8) {};
\node at (20,-10) [anchor=north,font=\small] {$t'$};
\end{scope}

\node at (84.25,-30) {$\longmapsto$};

\begin{scope}[yshift=-30mm,xshift=91mm,scale=0.75]
\fill[YellowOrange!10] (0,-15) rectangle (10,15);
\fill[YellowOrange!10] (10,-10) rectangle (20,10);
\fill[black!3] (20,-10) rectangle (54,10);
\draw[green!50!black] (10,-15) -- (10,-10);
\draw[black!20] (10,-10) -- (10,10);
\draw[green!50!black] (10,10) -- (10,15);
\draw[green!50!black] (10,10) -- (54,10);
\draw[green!50!black] (10,-10) -- (54,-10);
\draw[black!20] (20,-10) -- (20,10);
\draw[black!20] (44,-10) -- (44,10);
\node[green!50!black,fill,circle,inner sep=1pt] at (16,-7) {};
\node[green!50!black,fill,circle,inner sep=1pt] at (12,5) {};
\node[green!50!black,fill,circle,inner sep=1pt] at (2.5,-11) {};
\node[green!50!black,fill,circle,inner sep=1pt] at (7,-2) {};
\node[green!50!black,fill,circle,inner sep=1pt] at (5,8) {};
\node at (20,-10) [anchor=north,font=\small] {$t'$};
\node at (44,-10) [anchor=north,font=\small] {$t'+t$};
  \begin{scope}[xshift=40mm,yshift=2.5mm,scale=0.3125]
  \fill[green!15] (0,0) circle (8);
  \draw[green!50!black] (0,0) circle (8);
  \node[green!50!black,fill,circle,inner sep=1pt] at (2,0) {};
  \node[green!50!black,fill,circle,inner sep=1pt] at (4,-5) {};
  \node[green!50!black,fill,circle,inner sep=1pt] at (-6,-1) {};
  \node[green!50!black,fill,circle,inner sep=1pt] at (-1,4) {};
  \end{scope}
  \begin{scope}[xshift=33mm,yshift=-4mm,scale=0.5]
  \fill[yellow!10] (0,0) circle (8);
  \draw[green!50!black] (0,0) circle (8);
  \node[green!50!black,fill,circle,inner sep=1pt] at (2,1) {};
  \node[green!50!black,fill,circle,inner sep=1pt] at (-4,-4) {};
  \end{scope}
  \begin{scope}[xshift=27mm,yshift=4.5mm,scale=0.375]
  \fill[blue!20] (0,0) circle (8);
  \draw[green!50!black] (0,0) circle (8);
  \node[green!50!black,fill,circle,inner sep=1pt] at (-4,5) {};
  \node[green!50!black,fill,circle,inner sep=1pt] at (1,1) {};
  \node[green!50!black,fill,circle,inner sep=1pt] at (3,-5) {};
  \end{scope}
\end{scope}

\end{tikzpicture}
\caption{The maps \eqref{eq:lscd-action} defining the $\lswc_d$ action on $(\cgammad{}{c,D}{M}{\xi},\cmap{}{c,\partial}{D^d}{X})$, in dimension $d=2$ and for $k=3$. The light green, yellow, blue and orange colours represent sections defined on the complement of each configuration. Light grey indicates regions where the section is constant at the basepoint of $X$ (note that $\xi$ is trivial with fibre $X$ over the grey regions, so this makes sense).}
\label{fig:lscd-action}
\end{figure}
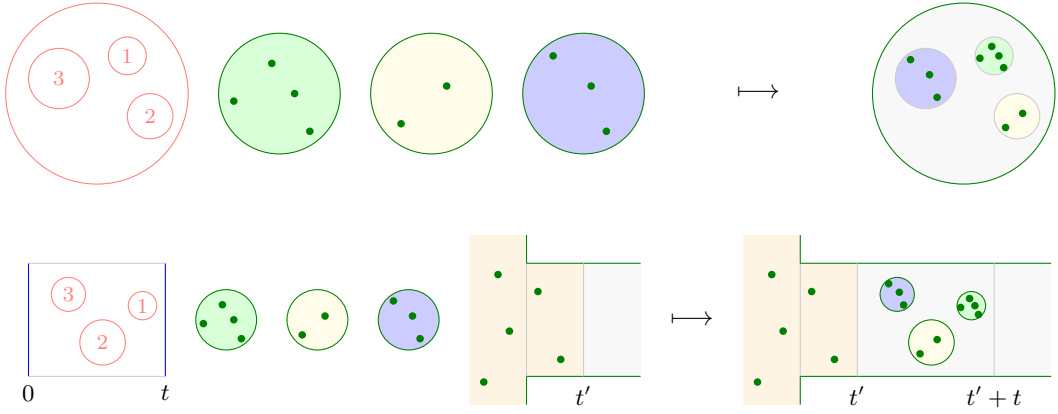

\begin{figure}[t]
\centering
\begin{tikzpicture}
[x=1mm,y=1mm]

\begin{scope}[xshift=6mm]
\draw[black!20] (-12,0) -- (12,0);
\node[red,fill,circle,inner sep=0.7pt] at (-12,0) {};
\node[red,fill,circle,inner sep=0.7pt] at (12,0) {};
\node[red,fill,circle,inner sep=0.7pt] at (-10,0) {};
\node[red,fill,circle,inner sep=0.7pt] at (-2,0) {};
\draw[red!50,decorate,decoration={brace,amplitude=2pt}] (-10,1) -- (-2,1);
\node[red!50,anchor=south,font=\footnotesize] at (-6,1) {$3$};
\node[red,fill,circle,inner sep=0.7pt] at (1,0) {};
\node[red,fill,circle,inner sep=0.7pt] at (6,0) {};
\draw[red!50,decorate,decoration={brace,amplitude=2pt}] (1,1) -- (6,1);
\node[red!50,anchor=south,font=\footnotesize] at (3.5,1) {$1$};
\node[red,fill,circle,inner sep=0.7pt] at (7,0) {};
\node[red,fill,circle,inner sep=0.7pt] at (11,0) {};
\draw[red!50,decorate,decoration={brace,amplitude=2pt}] (7,1) -- (11,1);
\node[red!50,anchor=south,font=\footnotesize] at (9,1) {$2$};
\end{scope}

\begin{scope}[xshift=30mm]
\fill[green!15] (0,0) circle (8);
\draw[green!50!black] (0,0) circle (8);
\node[green!50!black,fill,circle,inner sep=1pt] at (2,0) {};
\node[green!50!black,fill,circle,inner sep=1pt] at (4,-5) {};
\node[green!50!black,fill,circle,inner sep=1pt] at (-6,-1) {};
\node[green!50!black,fill,circle,inner sep=1pt] at (-1,4) {};
\end{scope}

\begin{scope}[xshift=50mm]
\fill[yellow!10] (0,0) circle (8);
\draw[green!50!black] (0,0) circle (8);
\node[green!50!black,fill,circle,inner sep=1pt] at (2,1) {};
\node[green!50!black,fill,circle,inner sep=1pt] at (-4,-4) {};
\end{scope}

\begin{scope}[xshift=70mm]
\fill[blue!20] (0,0) circle (8);
\draw[green!50!black] (0,0) circle (8);
\node[green!50!black,fill,circle,inner sep=1pt] at (-4,5) {};
\node[green!50!black,fill,circle,inner sep=1pt] at (1,1) {};
\node[green!50!black,fill,circle,inner sep=1pt] at (3,-5) {};
\end{scope}

\node at (93,0) {$\longmapsto$};

\begin{scope}[xshift=120mm]
\fill[black!3] (0,0) circle (12);
\fill[blue!5] (-10,0) arc (180:0:4) -- (-2,12) -- (-10,12) -- cycle;
\fill[green!5] (1,0) arc (180:0:2.5) -- (6,12) -- (1,12) -- cycle;
\fill[yellow!5] (7,0) arc (180:0:2) -- (11,12) -- (7,12) -- cycle;
\fill[pattern=crosshatch dots,pattern color=blue!50] (-10,0) arc (180:0:4) -- (-2,12) -- (-10,12) -- cycle;
\fill[pattern=crosshatch dots,pattern color=green!60] (1,0) arc (180:0:2.5) -- (6,12) -- (1,12) -- cycle;
\fill[pattern=crosshatch dots,pattern color=yellow!50] (7,0) arc (180:0:2) -- (11,12) -- (7,12) -- cycle;
\fill[white] (-12,0) arc (180:0:12) -- (12,12) -- (-12,12) -- cycle;
\draw[green!50!black] (0,0) circle (12);
  \begin{scope}[xshift=3.5mm,scale=0.3125]
  \fill[green!15] (0,0) circle (8);
  \draw[black!20] (0,0) circle (8);
  \node[green!50!black,fill,circle,inner sep=1pt] at (2,0) {};
  \node[green!50!black,fill,circle,inner sep=1pt] at (4,-5) {};
  \node[green!50!black,fill,circle,inner sep=1pt] at (-6,-1) {};
  \node[green!50!black,fill,circle,inner sep=1pt] at (-1,4) {};
  \end{scope}
  \begin{scope}[xshift=9mm,scale=0.25]
  \fill[yellow!10] (0,0) circle (8);
  \draw[black!20] (0,0) circle (8);
  \node[green!50!black,fill,circle,inner sep=1pt] at (2,1) {};
  \node[green!50!black,fill,circle,inner sep=1pt] at (-4,-4) {};
  \end{scope}
  \begin{scope}[xshift=-6mm,scale=0.5]
  \fill[blue!20] (0,0) circle (8);
  \draw[black!20] (0,0) circle (8);
  \node[green!50!black,fill,circle,inner sep=1pt] at (-4,5) {};
  \node[green!50!black,fill,circle,inner sep=1pt] at (1,1) {};
  \node[green!50!black,fill,circle,inner sep=1pt] at (3,-5) {};
  \end{scope}
\end{scope}

\begin{scope}[yshift=-30mm,xshift=6mm,scale=0.75]
\draw[black!20] (-12,0) -- (12,0);
\node[blue,fill,circle,inner sep=0.7pt] at (-12,0) {};
\node[blue,fill,circle,inner sep=0.7pt] at (12,0) {};
\node[red,fill,circle,inner sep=0.7pt] at (-11,0) {};
\node[red,fill,circle,inner sep=0.7pt] at (-6,0) {};
\draw[red!50,decorate,decoration={brace,amplitude=2pt}] (-11,1) -- (-6,1);
\node[red!50,anchor=south,font=\footnotesize] at (-8.5,1) {$3$};
\node[red,fill,circle,inner sep=0.7pt] at (-4.5,0) {};
\node[red,fill,circle,inner sep=0.7pt] at (3.5,0) {};
\draw[red!50,decorate,decoration={brace,amplitude=2pt}] (-4.5,1) -- (3.5,1);
\node[red!50,anchor=south,font=\footnotesize] at (-0.5,1) {$2$};
\node[red,fill,circle,inner sep=0.7pt] at (5,0) {};
\node[red,fill,circle,inner sep=0.7pt] at (11,0) {};
\draw[red!50,decorate,decoration={brace,amplitude=2pt}] (5,1) -- (11,1);
\node[red!50,anchor=south,font=\footnotesize] at (8,1) {$1$};
\node at (-12,0) [anchor=north,font=\small] {$0$};
\node at (12,0) [anchor=north,font=\small] {$t$};
\end{scope}

\begin{scope}[yshift=-30mm,xshift=23mm,scale=0.5]
\fill[green!15] (0,0) circle (8);
\draw[green!50!black] (0,0) circle (8);
\node[green!50!black,fill,circle,inner sep=1pt] at (2,0) {};
\node[green!50!black,fill,circle,inner sep=1pt] at (4,-5) {};
\node[green!50!black,fill,circle,inner sep=1pt] at (-6,-1) {};
\node[green!50!black,fill,circle,inner sep=1pt] at (-1,4) {};
\end{scope}

\begin{scope}[yshift=-30mm,xshift=35mm,scale=0.5]
\fill[yellow!10] (0,0) circle (8);
\draw[green!50!black] (0,0) circle (8);
\node[green!50!black,fill,circle,inner sep=1pt] at (2,1) {};
\node[green!50!black,fill,circle,inner sep=1pt] at (-4,-4) {};
\end{scope}

\begin{scope}[yshift=-30mm,xshift=47mm,scale=0.5]
\fill[blue!20] (0,0) circle (8);
\draw[green!50!black] (0,0) circle (8);
\node[green!50!black,fill,circle,inner sep=1pt] at (-4,5) {};
\node[green!50!black,fill,circle,inner sep=1pt] at (1,1) {};
\node[green!50!black,fill,circle,inner sep=1pt] at (3,-5) {};
\end{scope}

\begin{scope}[yshift=-30mm,xshift=55mm,scale=0.75]
\fill[YellowOrange!10] (0,-15) rectangle (10,15);
\fill[YellowOrange!10] (10,-10) rectangle (20,10);
\fill[black!3] (20,-10) rectangle (30,10);
\draw[green!50!black] (10,-15) -- (10,-10);
\draw[black!20] (10,-10) -- (10,10);
\draw[green!50!black] (10,10) -- (10,15);
\draw[green!50!black] (10,10) -- (30,10);
\draw[green!50!black] (10,-10) -- (30,-10);
\draw[black!20] (20,-10) -- (20,10);
\node[green!50!black,fill,circle,inner sep=1pt] at (16,-7) {};
\node[green!50!black,fill,circle,inner sep=1pt] at (12,5) {};
\node[green!50!black,fill,circle,inner sep=1pt] at (2.5,-11) {};
\node[green!50!black,fill,circle,inner sep=1pt] at (7,-2) {};
\node[green!50!black,fill,circle,inner sep=1pt] at (5,8) {};
\node at (20,-10) [anchor=north,font=\small] {$t'$};
\end{scope}

\node at (84.25,-30) {$\longmapsto$};

\begin{scope}[yshift=-30mm,xshift=91mm,scale=0.75]
\fill[YellowOrange!10] (0,-15) rectangle (10,15);
\fill[YellowOrange!10] (10,-10) rectangle (20,10);
\fill[black!3] (20,-10) rectangle (54,10);
\fill[blue!5] (21,0) arc (180:0:2.5) -- (26,10) -- (21,10) -- cycle;
\fill[yellow!5] (27.5,0) arc (180:0:4) -- (35.5,10) -- (27.5,10) -- cycle;
\fill[green!5] (37,0) arc (180:0:3) -- (43,10) -- (37,10) -- cycle;
\fill[pattern=crosshatch dots,pattern color=blue!50] (21,0) arc (180:0:2.5) -- (26,10) -- (21,10) -- cycle;
\fill[pattern=crosshatch dots,pattern color=yellow!50] (27.5,0) arc (180:0:4) -- (35.5,10) -- (27.5,10) -- cycle;
\fill[pattern=crosshatch dots,pattern color=green!60] (37,0) arc (180:0:3) -- (43,10) -- (37,10) -- cycle;
\draw[green!50!black] (10,-15) -- (10,-10);
\draw[black!20] (10,-10) -- (10,10);
\draw[green!50!black] (10,10) -- (10,15);
\draw[green!50!black] (10,10) -- (54,10);
\draw[green!50!black] (10,-10) -- (54,-10);
\draw[black!20] (20,-10) -- (20,10);
\draw[black!20] (44,-10) -- (44,10);
\node[green!50!black,fill,circle,inner sep=1pt] at (16,-7) {};
\node[green!50!black,fill,circle,inner sep=1pt] at (12,5) {};
\node[green!50!black,fill,circle,inner sep=1pt] at (2.5,-11) {};
\node[green!50!black,fill,circle,inner sep=1pt] at (7,-2) {};
\node[green!50!black,fill,circle,inner sep=1pt] at (5,8) {};
\node at (20,-10) [anchor=north,font=\small] {$t'$};
\node at (44,-10) [anchor=north,font=\small] {$t'+t$};
  \begin{scope}[xshift=40mm,scale=0.375]
  \fill[green!15] (0,0) circle (8);
  \draw[green!50!black] (0,0) circle (8);
  \node[green!50!black,fill,circle,inner sep=1pt] at (2,0) {};
  \node[green!50!black,fill,circle,inner sep=1pt] at (4,-5) {};
  \node[green!50!black,fill,circle,inner sep=1pt] at (-6,-1) {};
  \node[green!50!black,fill,circle,inner sep=1pt] at (-1,4) {};
  \end{scope}
  \begin{scope}[xshift=31.5mm,scale=0.5]
  \fill[yellow!10] (0,0) circle (8);
  \draw[green!50!black] (0,0) circle (8);
  \node[green!50!black,fill,circle,inner sep=1pt] at (2,1) {};
  \node[green!50!black,fill,circle,inner sep=1pt] at (-4,-4) {};
  \end{scope}
  \begin{scope}[xshift=23.5mm,scale=0.3125]
  \fill[blue!20] (0,0) circle (8);
  \draw[green!50!black] (0,0) circle (8);
  \node[green!50!black,fill,circle,inner sep=1pt] at (-4,5) {};
  \node[green!50!black,fill,circle,inner sep=1pt] at (1,1) {};
  \node[green!50!black,fill,circle,inner sep=1pt] at (3,-5) {};
  \end{scope}
\end{scope}

\end{tikzpicture}
\caption{The maps \eqref{eq:lscd-1-action} defining the $\lswc_{d-1}$ action on $(\cgammad{}{c,D}{M}{\xi},\cmap{}{c,\frac12\partial}{D^d}{X})$, also in dimension $d=2$ and for $k=3$. The light green, yellow, blue and orange colours represent sections defined on the complement of each configuration. Light grey indicates regions where the section is constant at the basepoint of $X$ (note that $\xi$ is trivial with fibre $X$ over the grey regions, so this makes sense). Dotted regions indicate that the map to $X$ is extended into this region by defining it to be independent of the vertical direction in this region. (\Cf Figure \ref{fig:stabilisation-maps} for a $3$-dimensional picture.)}
\label{fig:lscd-1-action}
\end{figure}

\section{Monodromy actions}\label{s:monodromy}

\begin{defn}
\label{d-monodromy-action}
Let $f \colon E \to B$ be a Serre fibration and $F = f^{-1}(b)$ for a point $b \in B$. Assume either that $F$ is a CW-complex or that $f$ is a Hurewicz fibration. Then the \emph{monodromy action} of $f$ is the action-up-to-homotopy
\begin{equation}
\label{eq-monodromy-action}
\mathrm{mon}_f \colon \pi_1(B,b) \longrightarrow \pi_0(\mathrm{hAut}(F))
\end{equation}
of $\pi_1(B,b)$ on the fibre $F$ defined as follows. Given an element $[\gamma] \in \pi_1(B,b)$ and a representative loop $\gamma \colon [0,1] \to B$, let $g \colon F \times [0,1] \to E$ be a choice of lift in the diagram:
\begin{equation}
\label{eq-monodromy-lifting}
\centering
\begin{split}
\begin{tikzpicture}
[x=1mm,y=1mm]
\node (tl) at (0,15) {$F$};
\node (tr) at (40,15) {$E$};
\node (bl) at (0,0) {$F \times [0,1]$};
\node (bm) at (25,0) {$[0,1]$};
\node (br) at (40,0) {$B$};
\draw[->] (tl) to node[above,font=\small]{$\mathrm{incl}$} (tr);
\draw[->>] (bl) to (bm);
\draw[->] (bm) to node[below,font=\small]{$\gamma$} (br);
\draw[->] (tl) to node[left,font=\footnotesize]{$(-,0)$} (bl);
\draw[->] (tr) to node[right,font=\small]{$f$} (br);
\draw[->,densely dashed] (bl) to (tr);
\end{tikzpicture}
\end{split}
\end{equation}
and define
\[
\mathrm{mon}_f([\gamma]) = [g(-,1)].
\]
\end{defn}

\begin{rmk}
In fact, all one needs for Definition \ref{d-monodromy-action} is a continuous map $f \colon E \to B$ and a point $b \in B$ such that $f$ satisfies the homotopy lifting property with respect to $F$ and $F \times [0,1]$ (\cf the proof of Lemma \ref{l-monodromy-action} below). 
\end{rmk}

\begin{lem}
\label{l-monodromy-action}
The construction of Definition \ref{d-monodromy-action} using the lifting diagram \eqref{eq-monodromy-lifting} gives a well-defined group homomorphism \eqref{eq-monodromy-action}.
\end{lem}
\begin{proof}
Suppose that $\gamma'$ is another representative of $[\gamma]$ and that $g'$ is a lift of $\gamma'$ in the diagram \eqref{eq-monodromy-lifting} (with $\gamma$ replaced by $\gamma'$). Let $k \colon F \times [0,1]^2 \to E$ be a choice of lift in the diagram:
\begin{equation}
\label{eq-monodromy-lifting-well-defined}
\centering
\begin{split}
\begin{tikzpicture}
[x=1mm,y=1mm]
\node (tl) at (0,15) {$F \times \bigl(\phantom{--}\bigr)$};
\begin{scope}[xshift=1.6mm,yshift=13.2mm]
\fill[black!10] (0,0) rectangle (4,4);
\draw[very thick] (4,0) -- (0,0) -- (0,4) -- (4,4);
\end{scope}
\node (tr) at (50,15) {$E$};
\node (bl) at (0,0) {$F \times [0,1]^2$};
\node (bm) at (25,0) {$[0,1]^2$};
\node (br) at (50,0) {$B$};
\draw[->] (tl) to node[above,font=\small]{$g_1 \cup \mathrm{const}_\mathrm{incl} \cup g_2$} (tr);
\draw[->>] (bl) to (bm);
\draw[->] (bm) to node[below,font=\small]{$h$} (br);
\draw[->] (tl) to node[left,font=\small]{$\mathrm{incl}$} (bl);
\draw[->] (tr) to node[right,font=\small]{$f$} (br);
\draw[->,densely dashed] (bl) to (tr);
\end{tikzpicture}
\end{split}
\end{equation}
where $h$ is a homotopy $\gamma \simeq \gamma'$ relative to endpoints. Then $k|_{F \times \{1\} \times [0,1]}$ is a homotopy $g(-,1) \simeq g'(-,1)$ of self-maps of $F$. This implies that the construction of Definition \ref{d-monodromy-action} gives a well-defined function $\mathrm{mon}_f \colon \pi_1(B,b) \to \pi_0(\mathrm{Map}(F,F))$. It remains to prove that $\mathrm{mon}_f$ is a homomorphism of monoids, since it will then follow that it has image contained in the underlying group $\pi_0(\mathrm{hAut}(F))$ of $\pi_0(\mathrm{Map}(F,F))$. It is clear that $\mathrm{mon}_f$ takes the constant loop to the identity map of $F$, since in this case we may take the lift in \eqref{eq-monodromy-lifting} to be the projection $F \times [0,1] \twoheadrightarrow F$ followed by the inclusion $F \hookrightarrow E$. We therefore just have to prove that
\[
\mathrm{mon}_f([\gamma_2 . \gamma_1]) = \mathrm{mon}_f([\gamma_2]) \circ \mathrm{mon}_f([\gamma_1])
\]
for elements $[\gamma_1],[\gamma_2] \in \pi_1(B,b)$. Choose lifts $g_1$ and $g_2$ in the diagrams:
\begin{equation}
\label{eq-monodromy-lifting-composition1}
\centering
\begin{split}
\begin{tikzpicture}
[x=1mm,y=1mm]
\node (tl) at (0,15) {$F$};
\node (tr) at (40,15) {$E$};
\node (bl) at (0,0) {$F \times [0,1]$};
\node (bm) at (25,0) {$[0,1]$};
\node (br) at (40,0) {$B$};
\draw[->] (tl) to node[above,font=\small]{$\mathrm{incl}$} (tr);
\draw[->>] (bl) to (bm);
\draw[->] (bm) to node[below,font=\small]{$\gamma_1$} (br);
\draw[->] (tl) to node[left,font=\footnotesize]{$(-,0)$} (bl);
\draw[->] (tr) to node[right,font=\small]{$f$} (br);
\draw[->,densely dashed] (bl) to node[above left=-1mm,font=\small]{$g_1$} (tr);
\begin{scope}[xshift=70mm]
\node (tl) at (0,15) {$F$};
\node (tr) at (40,15) {$E$};
\node (bl) at (0,0) {$F \times [1,2]$};
\node (bm) at (25,0) {$[1,2]$};
\node (br) at (40,0) {$B$};
\draw[->] (tl) to node[above,font=\small]{$\mathrm{incl}$} (tr);
\draw[->>] (bl) to (bm);
\draw[->] (bm) to node[below,font=\small]{$\gamma_2$} (br);
\draw[->] (tl) to node[left,font=\footnotesize]{$(-,1)$} (bl);
\draw[->] (tr) to node[right,font=\small]{$f$} (br);
\draw[->,densely dashed] (bl) to node[above left=-1mm,font=\small]{$g_2$} (tr);
\end{scope}
\end{tikzpicture}
\end{split}
\end{equation}
so we have $\mathrm{mon}_f([\gamma_2]) \circ \mathrm{mon}_f([\gamma_1]) = [g_2(-,2) \circ g_1(-,1)]$. We may now define a lift $F \times [0,2] \to E$ of the diagram:
\begin{equation}
\label{eq-monodromy-lifting-composition2}
\centering
\begin{split}
\begin{tikzpicture}
[x=1mm,y=1mm]
\node (tl) at (0,15) {$F$};
\node (tr) at (40,15) {$E$};
\node (bl) at (0,0) {$F \times [0,2]$};
\node (bm) at (22,0) {$[0,2]$};
\node (br) at (40,0) {$B$};
\draw[->] (tl) to node[above,font=\small]{$\mathrm{incl}$} (tr);
\draw[->>] (bl) to (bm);
\draw[->] (bm) to node[below,font=\small]{$\gamma_2 . \gamma_1$} (br);
\draw[->] (tl) to node[left,font=\footnotesize]{$(-,0)$} (bl);
\draw[->] (tr) to node[right,font=\small]{$f$} (br);
\draw[->,densely dashed] (bl) to (tr);
\end{tikzpicture}
\end{split}
\end{equation}
by:
\[
(x,t) \longmapsto \begin{cases}
g_1(x,t) & t \in [0,1] \\
g_2(g_1(x,1),t) & t \in [1,2].
\end{cases}
\]
By definition, it follows that $\mathrm{mon}_f([\gamma_2 . \gamma_1]) = [g_2(g_1(-,1),2)] = [g_2(-,2) \circ g_1(-,1)]$.
\end{proof}

\begin{notation}
\label{notation-inputs}
From now on, we fix, once and for all, choices of the objects of Definitions \ref{d:Mhat} and \ref{d:cgammadot}, namely:
\begin{itemizeb}
\item a manifold $M$ equipped with an embedded codimension-zero disc $D \subseteq \partial M$ with centre $*$,
\item a collar neighbourhood of $\partial M$, namely an embedding $b \colon (-\infty,0] \times \partial M \hookrightarrow M$ so that $b(0,-)$ is the inclusion $\partial M \subset M$,
\item[$\triangleright$] This determines the manifold $\hat{M}$ and its submanifolds $\hat{M}_r$ ($r \in [0,\infty)$) as in Definition \ref{d:Mhat}.
\item a fibre bundle $\xi \colon E \to M$, with basepoint $x_0 \in X \coloneqq \xi^{-1}(*)$,
\item a subset $c \subseteq \pi_0(\Sigma(\xi))$ (\cf Definition \ref{s:singularity}),
\item a trivialisation $\theta \colon \xi|_D \cong D \times X$. 
\item[$\triangleright$] We write $s_D$ for the section of $\xi|_D$ corresponding to the constant section of $D \times X$ at $x_0$.
\item[$\triangleright$] Using the trivialisation $\theta$, we extend $\xi$ by a trivial $X$-bundle to obtain a bundle $\hat{\xi}$ over $\hat{M}$.
\end{itemizeb}
\end{notation}

Recall the homotopy-equivalent models $\dot{C}_k(M) \simeq C_k(\mathring{M})$ and $\cgammad{k}{c,D}{M}{\xi} \simeq \cgamma{k}{c,D}{M}{\xi}$ for configuration spaces and configuration-section spaces from \S\ref{s:em-modules-over-en-algebras} (Definitions \ref{d:cdot} and \ref{d:cgammadot}).

\begin{lem}
\label{l:forgetful-Hurewicz}
There is a Hurewicz fibration
\begin{equation}
\label{eq:cdot-fibration}
\cgammad{k}{c,D}{M}{\xi} \too \dot{C}_k(M)
\end{equation}
given by forgetting the section data of a configuration-section.
\end{lem}
\begin{proof}
The forgetful map \eqref{eq:cdot-fibration} is defined by $(z,s,t) \mapsto (z,t)$, where $t>0$ is a real number, $z$ is a configuration in the interior of $\hat{M}_t$ and $s$ is a section of $\hat{\xi}$ over $\hat{M}_t \smallsetminus z$. There are homeomorphisms
\begin{equation}
\label{eq:identifications}
\dot{C}_k(M) \cong C_k(\mathring{M}) \times (0,\infty) \qquad\text{and}\qquad \cgammad{k}{c,D}{M}{\xi} \cong \cgamma{k}{c,D}{M}{\xi} \times (0,\infty)
\end{equation}
under which \eqref{eq:cdot-fibration} corresponds to the map $p \times \mathrm{id}_{(0,\infty)}$, where $p$ is the Hurewicz fibration \eqref{eq:fibre-bundle8}. Hence \eqref{eq:cdot-fibration} is also a Hurewicz fibration. The homeomorphisms above may be defined as follows. Let $D'$ be an open codimension-zero disc in $\partial M$ containing $D$ in its interior. Choose an identification of $b((-\infty,0] \times D') \subseteq M$ with $(-\infty,0] \times D$ so that $b(\{0\} \times D)$ corresponds to $\{0\} \times D$ (Figure \ref{fig:embedded-DR}). This induces an embedding $D \times \bR \hookrightarrow \hat{M}$, and we obtain a homeomorphism $\psi_r \colon \hat{M} \to \hat{M}$ for each $r \in \bR$ by defining $\psi_r(x,t) = (x,t+r)$ for $(x,t) \in D \times \bR$ and $\psi_r(y)=y$ for $y \in \hat{M} \smallsetminus (D \times \bR)$. The left-hand homeomorphism of \eqref{eq:identifications} may then be defined by
\[
(z,t) \longmapsto (\psi_{-t}(z),t).
\]
Choosing a trivialisation of $\hat{\xi}|_{D \times \bR}$ (extending the identity trivialisation of $\hat{\xi}|_{D \times [0,\infty)}$, which is trivial by construction), we may lift $\psi_r \colon \hat{M} \to \hat{M}$ to a bundle-homeomorphism $\widetilde{\psi}_r \colon \hat{\xi} \to \hat{\xi}$. The right-hand homeomorphism of \eqref{eq:identifications} may then be defined by
\[
(z,s,t) \longmapsto (\psi_{-t}(z),\widetilde{\psi}_t \circ s \circ \psi_{-t},t). \qedhere
\]
\end{proof}

\begin{figure}[t]
\centering
\begin{tikzpicture}
[x=1mm,y=1mm]
\fill[green!20] (20,-20) rectangle (40,20);
\fill[green!30] (30,-10) rectangle (40,10);
\fill[blue!40,opacity=0.5] (30,-10) rectangle (40,10);
\fill[blue!30] (40,-5) rectangle (80,5);
\draw[black!20] (30,-20) -- (30,20);
\draw[black!20] (30,-10) -- (40,-10);
\draw[black!20] (30,10) -- (40,10);
\draw (40,-20) -- (40,20);
\draw (40,5) -- (80,5);
\draw (40,-5) -- (80,-5);
\node at (40,-20) [fill,inner sep=1pt] {};
\node at (30,-20) [fill,inner sep=1pt] {};
\node at (80,-5) [fill,inner sep=1pt] {};
\foreach \x/\y/\z in {38/6/4,36/7/3,34/8/2,32/9/1}
  \draw (40,\y) .. controls (\x,\y) and (\x,\y) .. (\x,\z) -- (\x,-\z) .. controls (\x,-\y) and (\x,-\y) .. (40,-\y);
\node at (25,0) {$M$};
\node at (35,-15) {$\mathrm{im}(b)$};
\draw[decorate,decoration={brace,amplitude=3pt,mirror}] (41,-4.5) -- (41,4.5);
\node at (41.5,0) [anchor=west] {$D$};
\node at (40,-20) [font=\small,anchor=north] {$0$};
\node at (30,-20) [font=\small,anchor=north] {$-\infty$};
\node at (80,-5) [font=\small,anchor=north] {$\infty$};
\end{tikzpicture}
\caption{The turquoise region is $b((-\infty,0] \times D')$ and is identified with $(-\infty,0] \times D$ so that $b(\{0\} \times D)$ corresponds to $\{0\} \times D$. For illustration, four other slices $\{t\} \times D$ under this identification are drawn.}
\label{fig:embedded-DR}
\end{figure}
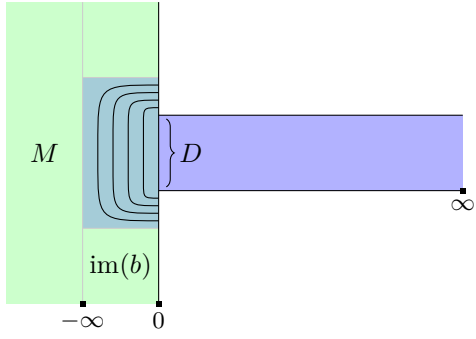

\begin{rmk}
The homeomorphisms \eqref{eq:identifications} constructed in the proof of Lemma \ref{l:forgetful-Hurewicz} give an alternative proof of Lemma \ref{l:equivalent-models}, although the construction of \eqref{eq:identifications} is a little more ad hoc.
\end{rmk}

\begin{defn}
For a real number $r\geq 1$, let $p_r = (*,r-\tfrac12) \in D \times [0,\infty) \subseteq \hat{M}$, where $*$ denotes the centre of the disc $D$. For an integer $k\geq 1$, the $k$-th ``standard configuration'' in $\hat{M}$ is defined to be $z_k = \{p_1,p_2,\ldots,p_k\}$, and the basepoint of $\dot{C}_k(M)$ is defined to be $(z_k,k)$.
\end{defn}

\begin{rmk}
\label{rmk:stabilisation}
Stabilisation maps for configuration spaces and configuration-section spaces will be defined in \S\ref{s:extension}, using the $E_0$-module structure of \S\ref{s:em-modules-over-en-algebras}. However, at the level of configuration spaces, it is already clear that the stabilisation map $\dot{C}_k(M) \to \dot{C}_{k+1}(M)$ should be defined by
\[
(z,t) \longmapsto (z \sqcup \{p_{t+1}\} , t+1).
\]
\end{rmk}

\begin{lem}
\label{l:formula-for-fibre}
The fibre of \eqref{eq:cdot-fibration} over $(z_k,k) \in \dot{C}_k(M)$ is the space
\begin{equation}
\label{eq:formula-for-fibre}
\Gamma_k^{c,D}(M;\xi) \coloneqq \Gamma^{c_D,D\times\{k\}} \bigl( \hat{M}_k \smallsetminus z_k ; \hat{\xi} \bigr)
\end{equation}
of sections $s$ of $\hat{\xi}$ defined over $\hat{M}_k \smallsetminus z_k \subseteq \hat{M}$ such that
\begin{itemizeb}
\item the restriction of $s$ to a small punctured neighbourhood of $p_i$ lies in $c_D \subseteq [S^{d-1},X]$ for each $i \in \{1,\ldots,k\}$, where $c_D$ is as in Definition \ref{d:cgammadot},
\item the restriction of $s$ to $D \times \{k\} \subseteq \partial \hat{M}_k$ is constant at the basepoint $x_0$ of $X$.
\end{itemizeb}
\end{lem}

\begin{coro}
\label{c:Hurewicz-fibration-sequence}
We have a Hurewicz fibration sequence of the form
\begin{equation}
\label{eq:Hurewicz-fibration-sequence}
\Gamma_k^{c,D}(M;\xi) \too \cgammad{k}{c,D}{M}{\xi} \too \dot{C}_k(M).
\end{equation}
\end{coro}
\begin{proof}[Proof of Corollary \ref{c:Hurewicz-fibration-sequence}.]
This follows immediately from Lemmas \ref{l:forgetful-Hurewicz} and \ref{l:formula-for-fibre}.
\end{proof}

\begin{proof}[Proof of Lemma \ref{l:formula-for-fibre}.]
Directly from the definitions, the fibre of \eqref{eq:cdot-fibration} over $(z_k,k)$ may be described as written, except that the first condition says that $s$ must satisfy the singularity condition $\hat{c} \subseteq \pi_0(\Sigma(\hat{\xi}))$, where $\hat{c}$ is determined by $c \subseteq \pi_0(\Sigma(\xi))$ as explained in Definition \ref{d:cgammadot}. But all of the points of $z_k$ lie in $D \times [0,\infty) \subseteq \hat{M}$, over which the bundle $\hat{\xi}$ is trivial with fibre $X$, so the singularity conditions around these points are equivalent to the conditions written in the lemma.
\end{proof}

\begin{defn}
\label{d:braid-groupoid}
Let
\[
\mathrm{Br}(M) = \bigl[\pi_1(\dot{C}_k(M))\bigr]_{k \in \bN}
\]
denote the groupoid whose objects are $\bN$, whose automorphism group of $k \in \bN$ is $\pi_1(\dot{C}_k(M))$ and which has no morphisms between distinct objects.
\end{defn}

\begin{defn}
\label{d:associated-monodromy-functor}
Fix $(M,D,*,b,\xi,x_0,c,\theta)$ as in Notation \ref{notation-inputs}.
The \emph{associated monodromy functor}
\begin{equation}
\label{eq:associated-monodromy-functor}
\mathrm{Mon}^{c,D}(M,\xi) \colon \mathrm{Br}(M) \longrightarrow \hotop
\end{equation}
takes the object $k \in \bN$ to the fibre $\Gamma_k^{c,D}(M;\xi)$ of \eqref{eq:cdot-fibration}. On automorphisms of $k$, it is defined by the monodromy action \eqref{eq-monodromy-action}, with $f = \eqref{eq:cdot-fibration}$.
\end{defn}

\section{Braid categories}\label{s:braid-categories}

\begin{defn}[{\cite[p.\ 219]{Grayson1976}}]
\label{d:quillen}
The \emph{Quillen bracket construction} $\langle \cD , \cC \rangle$ of a category $\cC$ equipped with a left-action of a monoidal category $\cD$ is the category with the same objects as $\cC$, and with morphisms given by $\langle \cD , \cC \rangle(c,c') = \mathrm{colim}_{\cD}(\cC(- \oplus c,c'))$, where $\oplus$ is the action of $\cD$ on $\cC$. In other words, a morphism $c \to c'$ in $\langle \cD , \cC \rangle$ is an equivalence class of morphisms $\varphi \colon d \oplus c \to c'$ in $\cC$ with $d$ in $\cD$, where the equivalence relation is the symmetric, transitive closure of the relation given by $(d_1,\varphi_1) \sim (d_2,\varphi_2)$ if there exists $\theta \colon d_1 \to d_2$ in $\cD$ such that $\varphi_1 = \varphi_2 \circ (\theta \oplus \mathrm{id}_c)$. This comes equipped with a canonical functor
\[
\cC \longrightarrow \langle \cD , \cC \rangle
\]
given by the identity on objects and by $\varphi \mapsto 0 \oplus \varphi$ on morphisms, where $0$ is the unit object of the monoidal structure on $\cD$.
\end{defn}

\begin{lem}
\label{l:boundary-connected-sum-action}
The groupoid $\mathrm{Br}(D^d)$ has a monoidal structure given by taking the boundary connected sum of two discs. If $d \geq 2$ it is braided and if $d \geq 3$ it is symmetric. Now let $M$ be a connected $d$-manifold with non-empty boundary, and let $D \subseteq \partial M$ be an embedded $(d-1)$-dimensional disc. There is a well-defined action of $\mathrm{Br}(D^d)$ on $\mathrm{Br}(M)$ given by boundary connected sum along $D$.
\end{lem}
\begin{proof}
Let us write $\mathrm{Br}(M) = \pi_1(\dot{C}(M),\{z_k\}_{k \in \bN})$, the fundamental groupoid of $\dot{C}(M)$ with respect to the set of basepoints $\{z_k \in \dot{C}_k(M) \mid k \in \bN\}$. This is a skeleton for, hence equivalent to, the full fundamental groupoid $\pi_1(\dot{C}(M))$, with objects \emph{all} points of $\dot{C}(M)$. When $M=D^d$ we have $\pi_1(\dot{C}(D^d)) \simeq \pi_1(C(\mathring{D}^d))$, since $\dot{C}(D^d)$ and $C(\mathring{D}^d)$ are homotopy equivalent. By Proposition \ref{p:e0-modules}(i), $C(\mathring{D}^d)$ is an $E_d$-algebra and $\dot{C}(M)$ is an $E_0$-module over it. Passing to fundamental groupoids and pulling back along the equivalences, this gives rise to the structure on $\mathrm{Br}(D^d)$ and $\mathrm{Br}(M)$ claimed in the lemma.
\end{proof}

\begin{rmknotation}
Note that $\mathrm{Br}(D^1)$ is the free monoidal category on one object, $\mathrm{Br}(D^2)$ is the free \emph{braided} monoidal category on one object and $\mathrm{Br}(D^3)$ is the free \emph{symmetric} monoidal category on one object. We will therefore abbreviate these groupoids by $\cM$, $\cB$ and $\cS$ respectively. The standard inclusions $D^1 \hookrightarrow D^2 \hookrightarrow D^3$ induce monoidal functors $\cM \hookrightarrow \cB \hookrightarrow \cS$ (the second one is also braided monoidal), and for $d \geq 3$ the standard inclusion $D^3 \hookrightarrow D^d$ induces an isomorphism $\cS \cong \mathrm{Br}(D^d)$.
\end{rmknotation}

\begin{defn}[{\cite[\S 5.2]{Krannich2019Homologicalstabilitytopological}}]
\label{d:coeff-system-category}
Let $M$ be a connected $d$-manifold with non-empty boundary, with $d \geq 2$, and let $D \subseteq \partial M$ be an embedded $(d-1)$-dimensional disc. By Lemma \ref{l:boundary-connected-sum-action}, there is a well-defined action of the braided monoidal category $\mathrm{Br}(D^d)$ on $\mathrm{Br}(M)$. The standard inclusion $D^2 \hookrightarrow D^d$ induces a braided monoidal functor $\cB \to \mathrm{Br}(D^d)$, and hence an action of $\cB$ on $\mathrm{Br}(M)$. Using Definition \ref{d:quillen}, we may therefore define the category $\cC(M) = \langle \cB , \mathrm{Br}(M) \rangle$, which is equipped with a canonical functor
\begin{equation}
\label{eq:underlying-groupoid}
\mathrm{Br}(M) \longrightarrow \cC(M).
\end{equation}
\end{defn}

\begin{lem}
The functor \eqref{eq:underlying-groupoid} is the inclusion of the underlying groupoid of $\cC(M)$.
\end{lem}
\begin{proof}
This is an immediate adaptation of the proof of \cite[Proposition 1.7]{Randal-WilliamsWahl2017}. They consider the setting of a monoidal groupoid $\cD$ acting on itself, and prove that the canonical functor $\cD \to \langle \cD , \cD \rangle$ is the inclusion of the underlying groupoid of $\langle \cD , \cD \rangle$, as long as $\cD$ has no zero-divisors and its monoidal unit has no non-trivial automorphisms. The same proof goes through more generally for a groupoid $\cC$ equipped with an action of a monoidal groupoid $\cD$ satisfying the same two conditions. In our setting we have $\cC = \mathrm{Br}(M)$ and $\cD = \cB$, which clearly satisfies these conditions.
\end{proof}

\begin{rmk}
Our notation is the reverse of that of \cite{Krannich2019Homologicalstabilitytopological}, since we are using the opposite convention of considering \emph{left}-actions of monoidal categories.
\end{rmk}

\begin{defn}[{\cite[\S 2.3 and \S 3.1]{Palmer2018Twistedhomologicalstability}}]
The categories $\cB(M)$ and $\cB_\sharp(M)$ both have $\bN$ as their set of objects. A morphism $k \to \ell$ in $\cB_\sharp(M)$ is a path $\gamma$, up to endpoint-preserving homotopy, in the space $C_r(\mathring{M})$ for some $r$, satisfying $\gamma(0) \subseteq \{ p_1,\ldots,p_k \}$ and $\gamma(1) \subseteq \{ p_1,\ldots,p_\ell \}$. Composition is defined by analogy with composition of partially-defined functions: given morphisms $\gamma \colon k \to \ell$ and $\delta \colon \ell \to m$, let $r = \lvert \gamma(1) \cap \delta(0) \rvert$ and define $\delta \circ \gamma$ to be the path in $C_r(\mathring{M})$ obtained by concatenating the corresponding restrictions of $\gamma$ and $\delta$. A morphism $k \to \ell$ in the subcategory $\cB(M) \subseteq \cB_\sharp(M)$ is a path $\gamma$ as above, with $r=k$.
\end{defn}

Heuristically, morphisms in $\cB(M)(k,\ell)$ may be thought of as ``injective braids'' with $k$ strands going from $k$ points to a subset of $\ell$ points, whereas morphisms in $\cC(M)(k,\ell)$ may be thought of as braids on $\ell$ strands \emph{modulo} braids on $\ell - k$ strands. Forgetting the last $\ell - k$ strands therefore defines a functor $\cC(M) \to \cB(M)$, as we explain next.

\begin{lem}
Under the conditions of Definition \ref{d:coeff-system-category}, there is a canonical functor
\[
\cC(M) \longrightarrow \cB(M),
\]
which is the identity on objects and is also:
\begin{itemizeb}
\item full, for any $M$,
\item faithful (and hence an isomorphism) if and only if $\mathrm{dim}(M) \geq 3$ and $M$ is simply-connected.
\end{itemizeb}
\end{lem}
\begin{proof}
The set of morphisms $m \to n$ of $\cC(M)$ is empty if $m>n$, and if $m \leq n$ it is naturally identified with the orbit set $B_n(M) / B_m$, where $B_n(M) = \pi_1(C_n(M),\{p_1,\ldots,p_n\})$ and $B_m$ acts on $B_n(M)$ via the homomorphism $v_{n-m}^m \colon B_m \to B_n(M)$ of Definition \ref{def:sigmak-and-vkl} below followed by right-multiplication of $B_n(M)$ on itself. On the other hand, the set of morphisms $m \to n$ of $\cB(M)$ is also empty if $m>n$, and if $m \leq n$ it is a homotopy class of paths in $C_m(M)$ from the basepoint configuration $\{p_1,\ldots,p_m\}$ to a subconfiguration of $\{p_1,\ldots,p_n\}$.

The functor $\cC(M) \to \cB(M)$ is defined on morphisms as follows. Given a morphism $m \to n$ of $\cC(M)$, represented by a loop of configurations $\gamma$ in $C_n(M)$ based at $\{p_1,\ldots,p_n\}$, forget all strands of $\gamma$ that start at $p_i$ for $m+1 \leq i \leq n$. The result is a path of configurations in $C_m(M)$ representing a morphism $m \to n$ of $\cB(M)$.

Given any path of configurations $\delta$ in $C_m(M)$ from $\{p_1,\ldots,p_m\}$ to a subconfiguration of $\{p_1,\ldots,p_n\}$, it is always possible to ``adjoin strands'' to $\delta$ to extend it to a loop of configurations $\gamma$ in $C_n(M)$. This implies that the functor $\cC(M) \to \cB(M)$ is full.

The fact that the functor $\cC(M) \to \cB(M)$ is faithful if $\mathrm{dim}(M) \geq 3$ and $M$ is simply-connected is stated in Remark 5.10 of \cite{Krannich2019Homologicalstabilitytopological}. It also follows from Theorem 9 of \cite{FadellNeuwirth1962Configurationspaces}, which implies that the functor is given by $\Sigma_n / \Sigma_{n-m} \to \mathrm{Inj}(m,n)$ on morphism-sets (from $m$ to $n$); we already know that this is surjective, so injectivity follows from a simple counting argument.

If $M$ is $2$-dimensional or $\pi_1(M)$ is non-trivial, it is easy to construct pairs of distinct morphisms $1 \to 2$ in $\cC(M)$ that become equal in $\cB(M)$, i.e., after forgetting the second strand of a $2$-strand braid on $M$. Thus the conditions that $\mathrm{dim}(M) \geq 3$ and $M$ is simply-connected are necessary for $\cC(M) \to \cB(M)$ to be a faithful functor.
\end{proof}

\begin{summary}
\label{summary:braid-categories}
Let $M$ be a connected $d$-manifold with non-empty boundary, with $d \geq 2$, and let $D \subseteq \partial M$ be an embedded $(d-1)$-dimensional disc. There are then canonical functors
\begin{equation}
\label{eq:braid-category-functors}
\mathrm{Br}(M) \longhookrightarrow \cC(M) \longtwoheadrightarrow \cB(M) \longhookrightarrow \cB_\sharp(M)
\end{equation}
that all act by the identity on their common set of objects (which is $\bN$). The first and third functors are faithful, and the second functor is full. The second functor is also faithful (and therefore an isomorphism) if and only if $M$ is simply-connected and has dimension at least $3$.
\end{summary}

\section{Stabilisation maps and extensions to \texorpdfstring{$\cC(M)$}{C(M)} and \texorpdfstring{$\cB_\sharp(M)$}{B-sharp(M)}}\label{s:extension}

As in \S\ref{s:monodromy}, we fix the data $(M,D,*,b,\xi,x_0,c,\theta)$ of Notation \ref{notation-inputs}, namely a bundle $\xi \colon E \to M$ over a $d$-manifold, a disc $D \subseteq \partial M$, etc.

\paragraph{Stabilisation maps.}
Let us choose, once and for all, an element of $\lswc_{d-1}(\sa,\sm;\sm)$ consisting of one $(d-1)$-disc in a $(d-1)$-rectangle of width $1$, as well as an element of $\cmap{}{c,\frac12\partial}{D^d}{X}$ where the configuration has exactly one point. See Figure \ref{fig:stabilisation-maps}(a).

By Proposition \ref{p:e0-modules}, the pair $(\cgammad{}{c,D}{M}{\xi} , \cmap{}{c,\frac12\partial}{D^d}{X})$ is an algebra over the linear Swiss cheese operad $\lswc_{d-1}$ (\emph{i.e.~$\cgammad{}{c,D}{M}{\xi}$ is an $E_0$-module over the $E_{d-1}$-algebra $\cmap{}{c,\frac12\partial}{D^d}{X}$}). Moreover, the pair $(\dot{C}(M) , C(\mathring{D}^d))$ is an algebra over $\lswc_d$, and hence also over $\lswc_{d-1}$ by restriction, and the maps
\[
(\cgammad{}{c,D}{M}{\xi} , \cmap{}{c,\frac12\partial}{D^d}{X}) \too (\dot{C}(M) , C(\mathring{D}^d))
\]
that send a configuration-section to its underlying configuration (forgetting the section) are maps of $\lswc_{d-1}$-algebras. This structure induces (horizontal) \emph{stabilisation maps}
\begin{equation}
\label{eq:cgamma-stabilisation}
\centering
\begin{split}
\begin{tikzpicture}
[x=1mm,y=1mm]
\node (tl) at (0,15) {$\cgammad{k}{c,D}{M}{\xi}$};
\node (tr) at (30,15) {$\cgammad{k+1}{c,D}{M}{\xi}$};
\node (bl) at (0,0) {$\dot{C}_k(M)$};
\node (br) at (30,0) {$\dot{C}_{k+1}(M)$};
\draw[->] (tl) to (tr);
\draw[->] (bl) to (br);
\draw[->] (tl) to (bl);
\draw[->] (tr) to (br);
\end{tikzpicture}
\end{split}
\end{equation}
commuting with the (vertical) forgetful maps, for all $k \in \bN$. Concretely, the top horizontal map is given by the second line of \eqref{eq:lscd-1-action}, plugging in our choices of elements above. The bottom horizontal map is defined similarly, ignoring sections and considering just configurations. See Figure \ref{fig:stabilisation-maps}(b). Note that the bottom horizontal map of \eqref{eq:cgamma-stabilisation} is exactly as already described in Remark \ref{rmk:stabilisation}.

\begin{figure}[t]
\centering
\begin{tikzpicture}
[x=1mm,y=1mm]

\begin{scope}
\draw[blue] (-12,-7) -- (-12,7);
\draw[blue] (12,-7) -- (12,7);
\draw[black!20] (-12,-7) -- (12,-7);
\draw[black!20] (-12,7) -- (12,7);
\draw[red!50] (0,0) node[font=\footnotesize]{$1$} circle (5);
\node at (-12,-7) [anchor=north,font=\small] {$0$};
\node at (12,-7) [anchor=north,font=\small] {$1$};
\end{scope}

\begin{scope}[xshift=30mm]
\fill[green!15] (0,0) circle (8);
\draw[red!50] (0,0) circle (8);
\draw[red!50] (-8,0) arc (180:360:8 and 3.2);
\draw[red!50,densely dotted] (8,0) arc (0:180:8 and 3.2);
\node at (0,0) [green!50!black,fill,circle,inner sep=1pt] {};
\end{scope}

\node at (15,-10) [anchor=north,font=\small,align=left,text width=112mm] {\textbf{(a)} Choices of elements of $\lswc_{d-1}(\textcolor{red}{\sa},\textcolor{blue}{\sm};\textcolor{blue}{\sm})$ and of $\cmap{}{c,\frac12\partial}{D^d}{X}$. \\ The light green region represents a map $D^d \smallsetminus \{0\} \to X$ in one of the homotopy classes $c_D \subseteq [S^{d-1},X]$, where $c_D$ is determined by $c$ as in Definition \ref{d:cgammadot}, sending the southern hemisphere of $\partial D^d$ to $\{x_0\} \subseteq X.$};

\begin{scope}[yshift=-60mm,xshift=-30mm]
\fill[blue!10] (-10,0)--(-10,20)--(0,25)--(50,25)--(40,20)--(40,0)--cycle;
\fill[black!3] (40,0)--(40,20)--(50,25)--(90,25)--(90,5)--(80,0)--cycle;
\fill[green!5] (59,12.5) arc (180:0:6) -- (71,22.5) arc (0:180:6 and 1.5) -- cycle;
\fill[pattern=crosshatch dots,pattern color=green!80] (59,12.5) arc (180:0:6) -- (71,22.5) arc (0:180:6 and 1.5) -- cycle;
\draw[black!20] (59,12.5)--(59,22.5);
\draw[black!20] (71,12.5)--(71,22.5);

\draw[green!50!black] (0,0) rectangle (80,20);
\draw[green!50!black] (10,5) rectangle (90,25);
\draw[green!50!black] (0,0)--(10,5);
\draw[green!50!black] (0,20)--(10,25);
\draw[green!50!black] (80,0)--(90,5);
\draw[green!50!black] (80,20)--(90,25);
\draw[green!50!black] (40,0)--(40,20)--(50,25)--(50,5)--cycle;
\draw[green!50!black,densely dashed] (0,0)--(-10,0);
\draw[green!50!black,densely dashed] (0,20)--(-10,20);
\draw[green!50!black,densely dashed] (10,5)--(0,5);
\draw[green!50!black,densely dashed] (10,25)--(0,25);

\draw[decorate,decoration={brace,amplitude=3pt,mirror}] (-10,-1) -- (40,-1);
\node at (15,-2) [anchor=north,font=\small] {$\hat{M}_t$};
\node at (40,0) [fill,circle,inner sep=1pt] {};
\node at (40,-0.5) [anchor=north,font=\small] {$t$};
\node at (80,0) [fill,circle,inner sep=1pt] {};
\node at (80,-0.5) [anchor=north,font=\small] {$t+1$};

\fill[green!15] (65,12.5) circle (6);
\draw[red!50] (65,12.5) circle (6);
\draw[red!50] (59,12.5) arc (180:360:6 and 2.4);
\draw[red!50,densely dotted] (71,12.5) arc (0:180:6 and 2.4);
\node at (65,12.5) [green!50!black,fill,circle,inner sep=1pt] {};
\draw[black!20] (65,22.5) circle (6 and 1.5);

\node at (35,12.5) [green!50!black,fill,circle,inner sep=1pt] {};
\node at (21,6) [green!50!black,fill,circle,inner sep=1pt] {};
\node at (12,18) [green!50!black,fill,circle,inner sep=1pt] {};
\end{scope}

\node at (15,-68) [anchor=north,font=\small,align=left,text width=112mm] {\textbf{(b)} The stabilisation maps of \eqref{eq:cgamma-stabilisation}. Given a (blue) configuration-section $(z,s)$ lying in $\hat{M}_t$ with $s|_{D \times \{t\}} = \mathrm{const}_{x_0}$, add a new point at $p_{t+1} = (*,t+\tfrac12)$ and extend the section as illustrated: grey indicates the constant map to $x_0$, green indicates the configuration-section in $D^d$ chosen above and dotted-green indicates that the green region is extended, as in Figure \ref{fig:lscd-1-action}, by defining it to be independent of the vertical direction in this region.};

\end{tikzpicture}
\caption{Stabilisation maps for configuration-section spaces.}
\label{fig:stabilisation-maps}
\end{figure}

\begin{defn}
\label{def:sigmak-and-vkl}
The map
\[
\pi_1(\dot{C}_k(M)) \too \pi_1(\dot{C}_{k+1}(M))
\]
of fundamental groups induced by the stabilisation map \eqref{eq:cgamma-stabilisation} will be denoted by $\sigma_k$. Identifying the interior of the $(d-1)$-disc $D \subseteq \partial M$ with $(-1,1)^{d-1}$, we also have inclusions
\[
(-1,1) \times \{0\}^{d-2} \times (k,k+\ell) \longhookrightarrow \mathrm{int}(\hat{M}_{k+\ell})
\]
for integers $k\geq 0$ and $\ell \geq 1$, which induce maps
\begin{equation}
\label{eq:vkl-space}
C_\ell((-1,1) \times \{0\}^{d-2} \times (k,k+\ell)) \too \dot{C}_{k+\ell}(M)
\end{equation}
given by $S \mapsto (S \sqcup \{p_1,\ldots,p_k\} , k+\ell)$ as illustrated in Figure \ref{fig:vkl}. The induced map
\begin{equation}
\label{eq:vkl}
B_\ell \cong \pi_1(C_\ell((-1,1) \times \{0\}^{d-2} \times (k,k+\ell))) \too \pi_1(\dot{C}_{k+\ell}(M))
\end{equation}
of fundamental groups is denoted by $v_k^\ell$.
\end{defn}

\begin{figure}[t]
\centering
\begin{tikzpicture}
[x=1mm,y=1mm,scale=0.8]

\begin{scope}[xshift=-74mm,yshift=12.5mm]
\fill[red!10] (0,-5) rectangle (40,5);
\draw[red] (0,-5) rectangle (40,5);
\node at (10,-2) [green!50!black,fill,circle,inner sep=1pt] {};
\node at (4,3) [green!50!black,fill,circle,inner sep=1pt] {};
\node at (16,1) [green!50!black,fill,circle,inner sep=1pt] {};
\node at (32,-3) [green!50!black,fill,circle,inner sep=1pt] {};
\end{scope}

\node at (-22,12.5) {$\longmapsto$};

\draw[green!50!black] (60,25)--(60,5);
\fill[red!10] (50,10)--(60,15)--(110,15)--(100,10)--cycle;
\draw[red] (50,10)--(60,15)--(110,15)--(100,10)--cycle;

\draw[green!50!black] (0,0) rectangle (100,20);
\draw[green!50!black] (10,5) rectangle (110,25);
\draw[green!50!black] (0,0)--(10,5);
\draw[green!50!black] (0,20)--(10,25);
\draw[green!50!black] (100,0)--(110,5);
\draw[green!50!black] (100,20)--(110,25);
\draw[green!50!black] (25,0)--(25,20)--(35,25)--(35,5)--cycle;
\draw[green!50!black] (60,5)--(50,0)--(50,20)--(60,25);

\draw[green!50!black,densely dashed] (0,0)--(-10,0);
\draw[green!50!black,densely dashed] (0,20)--(-10,20);
\draw[green!50!black,densely dashed] (10,5)--(0,5);
\draw[green!50!black,densely dashed] (10,25)--(0,25);

\node at (0,0) [fill,inner sep=1pt] {};
\node at (0,-0.5) [anchor=north,font=\small] {$k-2$};
\node at (25,0) [fill,inner sep=1pt] {};
\node at (25,-0.5) [anchor=north,font=\small] {$k-1$};
\node at (50,0) [fill,inner sep=1pt] {};
\node at (50,-0.5) [anchor=north,font=\small] {$k$};
\node at (100,0) [fill,inner sep=1pt] {};
\node at (100,-0.5) [anchor=north,font=\small] {$k+\ell$};

\node at (-7.5,12.5) [green!50!black,fill,circle,inner sep=1pt] {};
\node at (-7.5,12.5) [anchor=north,font=\small] {$p_{k-2}$};
\node at (17.5,12.5) [green!50!black,fill,circle,inner sep=1pt] {};
\node at (17.5,12.5) [anchor=north,font=\small] {$p_{k-1}$};
\node at (42.5,12.5) [green!50!black,fill,circle,inner sep=1pt] {};
\node at (42.5,12.5) [anchor=north,font=\small] {$p_k$};

\node at (65.5,11.5) [green!50!black,fill,circle,inner sep=1pt] {};
\node at (63,14) [green!50!black,fill,circle,inner sep=1pt] {};
\node at (76,13) [green!50!black,fill,circle,inner sep=1pt] {};
\node at (92,11) [green!50!black,fill,circle,inner sep=1pt] {};

\end{tikzpicture}
\caption{The map \eqref{eq:vkl-space} inducing $v_k^\ell \colon B_\ell \to \pi_1(\dot{C}_{k+\ell}(M))$.}
\label{fig:vkl}
\end{figure}

The maps $\sigma_k \colon \pi_1(\dot{C}_k(M)) \to \pi_1(\dot{C}_{k+1}(M))$ and $v_k^\ell \colon B_\ell \to \pi_1(\dot{C}_{k+\ell}(M))$ of Definition \ref{def:sigmak-and-vkl} may be used to characterise extensions of functors along the inclusion $\mathrm{Br}(M) \subset \cC(M)$:

\begin{prop}[{\cite[\S 5.2]{Krannich2019Homologicalstabilitytopological}}]
\label{p:extension-criterion}
Let $M$ be a connected $d$-manifold, for $d \geq 2$, and let $D \subseteq \partial M$ be an embedded $(d-1)$-dimensional disc. Choose a collar neighbourhood $b$ as in Notation \ref{notation-inputs} and take the basepoint $* \in \partial M$ to be the centre of $D$. Let $F \colon \mathrm{Br}(M) \to \cD$ be a functor. An extension of $F$ to the larger category $\cC(M) \supseteq \mathrm{Br}(M)$ is equivalent to a choice of morphism $s_k \colon F(k) \to F(k+1)$ of $\cD$ for each $k \in \bN$ such that, for any $k \geq 0$, $\ell \geq 1$, $\alpha \in \pi_1(\dot{C}_k(M))$ and $\beta \in B_\ell$, the following two diagrams commute, where $s_k^\ell = s_{k+\ell -1} \circ \cdots \circ s_{k+1} \circ s_k$.
\begin{equation}
\label{eq:extension-criterion}
\centering
\begin{split}
\begin{tikzpicture}
[x=1mm,y=1mm]
\node (tl) at (0,12) {$F(k)$};
\node (tr) at (30,12) {$F(k+1)$};
\node (bl) at (0,0) {$F(k)$};
\node (br) at (30,0) {$F(k+1)$};
\draw[->] (tl) to node[above,font=\small]{$s_k$} (tr);
\draw[->] (bl) to node[above,font=\small]{$s_k$} (br);
\draw[->] (tl) to node[left,font=\small]{$F(\alpha)$} (bl);
\draw[->] (tr) to node[right,font=\small]{$F(\sigma_k(\alpha))$} (br);
\begin{scope}[xshift=60mm]
\node (l) at (0,6) {$F(k)$};
\node (tr) at (30,12) {$F(k+\ell)$};
\node (br) at (30,0) {$F(k+\ell)$};
\draw[->] (l) to node[above,font=\small]{$s_k^\ell$} (tr);
\draw[->] (l) to node[below,font=\small]{$s_k^\ell$} (br);
\draw[->] (tr) to node[right,font=\small]{$F(v_k^\ell(\beta))$} (br);
\end{scope}
\end{tikzpicture}
\end{split}
\end{equation}
\end{prop}

\begin{prop}
\label{p:extension-to-C}
The stabilisation maps \eqref{eq:cgamma-stabilisation} determine an extension of the monodromy functor $\eqref{eq:associated-monodromy-functor} = \mathrm{Mon}^{c,D}(M,\xi) \colon \mathrm{Br}(M) \to \hotop$ to a functor
\begin{equation}
\label{eq:extension-to-C}
\smash{\widetilde{\mathrm{Mon}}}^{c,D}(M,\xi) \colon \cC(M) \too \hotop .
\end{equation}
\end{prop}
\begin{proof}
We will apply Proposition \ref{p:extension-criterion} with $\cD = \hotop$ and $F = \eqref{eq:associated-monodromy-functor} = \mathrm{Mon}^{c,D}(M,\xi)$. Recall that $\mathrm{Mon}^{c,D}(M,\xi)$ sends $k$ to the fibre $\Gamma_k^{c,D}(M,\xi)$ of the Hurewicz fibration \eqref{eq:cdot-fibration} over the basepoint $(z_k,k) \in \dot{C}_k(M)$. The bottom horizontal map of the map of Hurewicz fibrations \eqref{eq:cgamma-stabilisation} preserves basepoints, so its top horizontal map restricts to a map of fibres
\[
F(k) = \Gamma_k^{c,D}(M,\xi) \too \Gamma_{k+1}^{c,D}(M,\xi) = F(k+1),
\]
which we define to be $s_k$. It remains to check the two conditions \eqref{eq:extension-criterion} of Proposition \ref{p:extension-criterion}.

The element $\alpha \in \pi_1(\dot{C}_k(M))$ acts on $F(k) = \Gamma_k^{c,D}(M,\xi)$ through a ``point-pushing'' diffeomorphism $\theta_\alpha$ of $\hat{M}_k \smallsetminus z_k$. (Only the isotopy class of $\theta_\alpha$ is important, since the diagrams \eqref{eq:extension-criterion} live in the homotopy category.) The element $\sigma_k(\alpha) \in \pi_1(\dot{C}_{k+1}(M))$ is the extension of $\alpha$ to a loop of $k+1$ points in $\hat{M}_{k+1}$ given by leaving the point $p_{k+1}$ fixed. Hence we may choose the point-pushing diffeomorphism $\theta_{\sigma_k(\alpha)}$ of $\hat{M}_{k+1} \smallsetminus z_{k+1}$ (through which $\sigma_k(\alpha)$ acts on $F(k+1) = \Gamma_{k+1}^{c,D}(M,\xi)$) to be the extension of $\theta_\alpha$ by the identity on $(D \times [k,k+1]) \smallsetminus \{p_{k+1}\}$. This implies that the left-hand square of \eqref{eq:extension-criterion} commutes up to homotopy.

Now consider any element $s \in F(k)$, a section of $\hat{\xi}$ defined on $\hat{M}_k \smallsetminus z_k$. Extend it in a standard way $\ell$ times, as shown in Figure \ref{fig:stabilisation-maps}(b), to obtain a section $\bar{s}$ of $\hat{\xi}$ defined on $\hat{M}_{k+\ell} \smallsetminus z_{k+\ell}$. The restriction of $\bar{s}$ to $P = (D \times [k,k+\ell]) \smallsetminus \{p_{k+1},\ldots,p_{k+\ell}\}$ may be thought of as a map $P \to X$, since $\hat{\xi}$ is trivial over $P$. Note that this map $P \to X$ does not in fact depend on $s$: it just consists of $\ell$ concatenated copies of the standard map to $X$ illustrated on the right-hand side of Figure \ref{fig:stabilisation-maps}(b). We denote it by $f_\ell \colon P \to X$. Any element $\beta \in B_\ell$ determines a self-diffeomorphism (up to isotopy) $\beta \colon P \to P$ by point-pushing.

\textbf{Claim:} $f_\ell \circ \beta$ is homotopic to $f_\ell$ through maps sending the two ends $D \times \{k\}$ and $D \times \{k+\ell\}$ to the basepoint $x_0 \in X$.

This claim implies that $\bar{s} \cdot \beta$ is homotopic to $\bar{s}$ through maps sending $D \times \{k+\ell\}$ to $x_0$, in other words, there is a path $\bar{s} \cdot \beta \sim \bar{s}$ in $\Gamma_{k+\ell}^{c,D}(M,\xi) = F(k+\ell)$, and moreover these paths may be chosen continuously in $s$. In other words, the right-hand triangle of \eqref{eq:extension-criterion} commutes up to homotopy.

It therefore remains to prove the claim above, and it will suffice to prove it when $\beta$ is the standard generator of $B_\ell$ that interchanges the two punctures $p_i$ and $p_{i+1}$ for $k+1 \leq i \leq \ell-1$. We will do this diagrammatically for dimensions $d\geq 3$ and by considering fundamental groups for dimension $d=2$.

First assume that $d\geq 3$. In this case, Figure \ref{fig:birds-eye-view} illustrates a ``bird's eye'' view of the map $f_\ell \colon P \to X$, where $P = (D \times [k,k+\ell]) \smallsetminus \{p_{k+1},\ldots,p_{k+\ell}\}$, by collapsing the ``vertical'' direction of Figures \ref{fig:stabilisation-maps}(b) and \ref{fig:vkl} (corresponding to the last copy of $[-1,1]$ in our identification of $D \subseteq \partial M$ with $[-1,1]^{d-1}$). Since $d\geq 3$ the resulting bird's-eye-view picture still has $d-1 \geq 2$ dimensions (only two are pictured in Figure \ref{fig:birds-eye-view}, of course). Now consider the point-pushing diffeomorphism $\beta \colon P \to P$ corresponding to the generator of the braid group that interchanges the punctures $p_i$ and $p_{i+1}$. From Figure \ref{fig:birds-eye-view}, and due to the fact that (the bird's eye view of) $f_\ell$ is \emph{the same} in a small neighbourhood of each puncture $p_{k+1},\ldots,p_{k+\ell}$, it is clear that the homotopy class of $f_\ell \circ \beta$ is the same as that of $f_\ell$. Moreover, a homotopy connecting them may be chosen to be supported in a small neighbourhood of an arc connecting $p_i$ and $p_{i+1}$, so it does not affect the ends $D \times \{k\}$ and $D \times \{k+\ell\}$ of $P$, which are therefore still mapped to the basepoint $x_0 \in X$ during the homotopy. This proves the claim when $d\geq 3$.

When $d=2$ we do not collapse the vertical direction, and in this case Figure \ref{fig:2dim-case}(a) illustrates the map $f_\ell \colon P \to X$ without any dimension reduction, and Figure \ref{fig:2dim-case}(b) illustrates the effect of the point-pushing diffeomorphism $\beta \colon P \to P$ that interchanges the punctures $p_i$ and $p_{i+1}$. We will argue that $\ref{fig:2dim-case}\text{(a)} \simeq \ref{fig:2dim-case}\text{(b)}$ relative to the left, bottom and right edges of the rectangle (and hence in particular relative to the left and right edges of the rectangle, which is the statement of the claim above).

A map $P \to X$ sending the left, bottom and right edges of the rectangle to the basepoint $x_0 \in X$ corresponds, up to relative homotopy, to an ordered $\ell$-tuple of elements of $\pi_1(X,x_0)$. Recall that, when defining the stabilisation maps, we chose an element of $\cmap{1}{c,\frac12\partial}{D^2}{X}$ (see Figure \ref{fig:stabilisation-maps}(a)), which is, up to homotopy, a choice of element $\kappa \in \pi_1(X,x_0)$ lying in one of the conjugacy classes $c_D \subseteq [S^1,X] = \mathrm{Conj}(\pi_1(X,x_0))$, where $c_D$ is determined by $c$ as explained in Definition \ref{d:cgammadot}. The map $\ref{fig:2dim-case}\text{(a)} \colon P \to X$ therefore corresponds to the ordered $\ell$-tuple $(\kappa,\kappa,\ldots,\kappa)$. Now the effect of the point-pushing diffeomorphism $\beta \colon P \to P$, under this correspondence, is to interchange the $i$-th and $(i+1)$-st elements of an ordered tuple while conjugating one by the other; in symbols: $(\ldots,\lambda_i,\lambda_{i+1},\ldots) \mapsto (\ldots,\lambda_i \lambda_{i+1} \lambda_i^{-1} , \lambda_i , \ldots)$. But the ordered $\ell$-tuple $(\kappa,\kappa,\ldots,\kappa)$ corresponding to \ref{fig:2dim-case}(a) is clearly fixed under this action, so $\ref{fig:2dim-case}\text{(b)} = \ref{fig:2dim-case}\text{(a)} \circ \beta$ also corresponds to $(\kappa,\kappa,\ldots,\kappa)$, and therefore $\ref{fig:2dim-case}\text{(a)} \simeq \ref{fig:2dim-case}\text{(b)}$ relative to the left, bottom and right edges of the rectangle.
\end{proof}

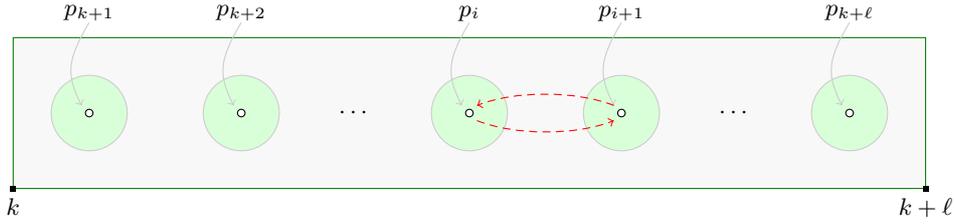
\begin{figure}[t]
\centering
\begin{tikzpicture}
[x=1mm,y=1mm,font=\small]

\fill[black!3] (0,0) rectangle (120,20);
\fill[green!15] (10,10) circle (5);
\fill[green!15] (30,10) circle (5);
\fill[green!15] (60,10) circle (5);
\fill[green!15] (80,10) circle (5);
\fill[green!15] (110,10) circle (5);

\draw[green!50!black] (0,0) rectangle (120,20);
\node at (0,0) [fill,inner sep=1pt] {};
\node at (0,0) [anchor=north,font=\small] {$k$};
\node at (120,0) [fill,inner sep=1pt] {};
\node at (120,0) [anchor=north] {$k+\ell$};
\node at (10,10) [fill,white,circle,inner sep=1pt] {};
\node at (30,10) [fill,white,circle,inner sep=1pt] {};
\node at (60,10) [fill,white,circle,inner sep=1pt] {};
\node at (80,10) [fill,white,circle,inner sep=1pt] {};
\node at (110,10) [fill,white,circle,inner sep=1pt] {};
\node at (10,10) [draw,circle,inner sep=1pt] {};
\node at (30,10) [draw,circle,inner sep=1pt] {};
\node at (60,10) [draw,circle,inner sep=1pt] {};
\node at (80,10) [draw,circle,inner sep=1pt] {};
\node at (110,10) [draw,circle,inner sep=1pt] {};
\node at (10,21) [anchor=south] {$p_{k+1}$};
\draw[black!20,->] (10,22) .. controls (7,17) and (7,15) .. (9,11);
\node at (30,21) [anchor=south] {$p_{k+2}$};
\draw[black!20,->] (30,22) .. controls (27,17) and (27,15) .. (29,11);
\node at (60,21) [anchor=south] {$p_i$};
\draw[black!20,->] (60,22) .. controls (57,17) and (57,15) .. (59,11);
\node at (80,21) [anchor=south] {$p_{i+1}$};
\draw[black!20,->] (80,22) .. controls (77,17) and (77,15) .. (79,11);
\node at (110,21) [anchor=south] {$p_{k+\ell}$};
\draw[black!20,->] (110,22) .. controls (107,17) and (107,15) .. (109,11);

\node at (45,10) {$\cdots$};
\node at (95,10) {$\cdots$};

\draw[black!20] (10,10) circle (5);
\draw[black!20] (30,10) circle (5);
\draw[black!20] (60,10) circle (5);
\draw[black!20] (80,10) circle (5);
\draw[black!20] (110,10) circle (5);

\draw[red,densely dashed,->] (79,11) .. controls (74,13) and (66,13) .. (61,11);
\draw[red,densely dashed,->] (61,9) .. controls (66,7) and (74,7) .. (79,9);

\end{tikzpicture}
\caption{A bird's eye view of the map $f_\ell \colon P \to X$, where $P = (D \times [k,k+\ell]) \smallsetminus \{p_{k+1},\ldots,p_{k+\ell}\}$. Grey indicates the constant map to the basepoint $x_0 \in X$ and the green regions are mapped to $X$ as in Figure \ref{fig:stabilisation-maps}(b), depending in particular on our choices in Figure \ref{fig:stabilisation-maps}(a). The red arrows illustrate the effect of the point-pushing diffeomorphism $\beta \colon P \to P$.}
\label{fig:birds-eye-view}
\end{figure}

\begin{figure}[t]
\centering
\begin{tikzpicture}
[x=1mm,y=1mm,font=\small]

\node at (-5,10) [font=\normalsize] {\textbf{(a)}};
\fill[black!3] (0,0) rectangle (120,20);

\fill[green!15] (5,10) arc (180:360:5) -- (15,20) -- (5,20) -- cycle;
\fill[green!15] (25,10) arc (180:360:5) -- (35,20) -- (25,20) -- cycle;
\fill[green!15] (55,10) arc (180:360:5) -- (65,20) -- (55,20) -- cycle;
\fill[green!15] (75,10) arc (180:360:5) -- (85,20) -- (75,20) -- cycle;
\fill[green!15] (105,10) arc (180:360:5) -- (115,20) -- (105,20) -- cycle;

\draw[green!50!black] (0,0) rectangle (120,20);
\node at (0,0) [fill,inner sep=1pt] {};
\node at (0,0) [anchor=north,font=\small] {$k$};
\node at (120,0) [fill,inner sep=1pt] {};
\node at (120,0) [anchor=north] {$k+\ell$};
\node at (10,10) [fill,white,circle,inner sep=1pt] {};
\node at (30,10) [fill,white,circle,inner sep=1pt] {};
\node at (60,10) [fill,white,circle,inner sep=1pt] {};
\node at (80,10) [fill,white,circle,inner sep=1pt] {};
\node at (110,10) [fill,white,circle,inner sep=1pt] {};
\node at (10,10) [draw,circle,inner sep=1pt] {};
\node at (30,10) [draw,circle,inner sep=1pt] {};
\node at (60,10) [draw,circle,inner sep=1pt] {};
\node at (80,10) [draw,circle,inner sep=1pt] {};
\node at (110,10) [draw,circle,inner sep=1pt] {};
\node at (10,21) [anchor=south] {$p_{k+1}$};
\draw[black!20,->] (10,22) .. controls (7,17) and (7,15) .. (9,11);
\node at (30,21) [anchor=south] {$p_{k+2}$};
\draw[black!20,->] (30,22) .. controls (27,17) and (27,15) .. (29,11);
\node at (60,21) [anchor=south] {$p_i$};
\draw[black!20,->] (60,22) .. controls (57,17) and (57,15) .. (59,11);
\node at (80,21) [anchor=south] {$p_{i+1}$};
\draw[black!20,->] (80,22) .. controls (77,17) and (77,15) .. (79,11);
\node at (110,21) [anchor=south] {$p_{k+\ell}$};
\draw[black!20,->] (110,22) .. controls (107,17) and (107,15) .. (109,11);

\node at (45,10) {$\cdots$};
\node at (95,10) {$\cdots$};

\draw[black!20] (5,20) -- (5,10) arc (180:360:5) -- (15,20);
\draw[black!20] (25,20) -- (25,10) arc (180:360:5) -- (35,20);
\draw[black!20] (55,20) -- (55,10) arc (180:360:5) -- (65,20);
\draw[black!20] (75,20) -- (75,10) arc (180:360:5) -- (85,20);
\draw[black!20] (105,20) -- (105,10) arc (180:360:5) -- (115,20);

\begin{scope}[yshift=-30mm]
\node at (-5,10) [font=\normalsize] {\textbf{(b)}};
\fill[black!3] (0,0) rectangle (120,20);

\fill[green!15] (5,10) arc (180:360:5) -- (15,20) -- (5,20) -- cycle;
\fill[green!15] (25,10) arc (180:360:5) -- (35,20) -- (25,20) -- cycle;
\fill[green!15] (75,20) .. controls (75,15) and (55,15) .. (55,10) arc (180:360:5) .. controls (65,15) and (85,15) .. (85,20) -- cycle;
\fill[green!15] (65,20) .. controls (65,15) and (54,15) .. (54,10) arc (180:270:6) .. controls (65,4) and (75,5) .. (75,10) arc (180:0:5) .. controls (85,5) and (70,1) .. (60,1) arc (270:180:9) .. controls (51,15) and (55,15) .. (55,20) -- cycle;
\fill[green!15] (105,10) arc (180:360:5) -- (115,20) -- (105,20) -- cycle;

\draw[green!50!black] (0,0) rectangle (120,20);
\node at (0,0) [fill,inner sep=1pt] {};
\node at (0,0) [anchor=north,font=\small] {$k$};
\node at (120,0) [fill,inner sep=1pt] {};
\node at (120,0) [anchor=north] {$k+\ell$};
\node at (10,10) [fill,white,circle,inner sep=1pt] {};
\node at (30,10) [fill,white,circle,inner sep=1pt] {};
\node at (60,10) [fill,white,circle,inner sep=1pt] {};
\node at (80,10) [fill,white,circle,inner sep=1pt] {};
\node at (110,10) [fill,white,circle,inner sep=1pt] {};
\node at (10,10) [draw,circle,inner sep=1pt] {};
\node at (30,10) [draw,circle,inner sep=1pt] {};
\node at (60,10) [draw,circle,inner sep=1pt] {};
\node at (80,10) [draw,circle,inner sep=1pt] {};
\node at (110,10) [draw,circle,inner sep=1pt] {};
\node at (10,21) [anchor=south] {$p_{k+1}$};
\draw[black!20,->] (10,22) .. controls (7,17) and (7,15) .. (9,11);
\node at (30,21) [anchor=south] {$p_{k+2}$};
\draw[black!20,->] (30,22) .. controls (27,17) and (27,15) .. (29,11);
\node at (60,21) [anchor=south] {$p_i$};
\draw[black!20,->] (60,22) .. controls (57,17) and (57,15) .. (59,11);
\node at (80,21) [anchor=south] {$p_{i+1}$};
\draw[black!20,->] (80,22) .. controls (77,17) and (77,15) .. (79,11);
\node at (110,21) [anchor=south] {$p_{k+\ell}$};
\draw[black!20,->] (110,22) .. controls (107,17) and (107,15) .. (109,11);

\node at (45,10) {$\cdots$};
\node at (95,10) {$\cdots$};

\draw[black!20] (5,20) -- (5,10) arc (180:360:5) -- (15,20);
\draw[black!20] (25,20) -- (25,10) arc (180:360:5) -- (35,20);
\draw[black!20] (75,20) .. controls (75,15) and (55,15) .. (55,10) arc (180:360:5) .. controls (65,15) and (85,15) .. (85,20);
\draw[black!20] (65,20) .. controls (65,15) and (54,15) .. (54,10) arc (180:270:6) .. controls (65,4) and (75,5) .. (75,10) arc (180:0:5) .. controls (85,5) and (70,1) .. (60,1) arc (270:180:9) .. controls (51,15) and (55,15) .. (55,20);
\draw[black!20] (105,20) -- (105,10) arc (180:360:5) -- (115,20);
\end{scope}

\end{tikzpicture}
\caption{\textbf{(a)} The map $f_\ell \colon P \to X$ when $d=2$, with colour-coding as in Figure \ref{fig:birds-eye-view}. \textbf{(b)} The map $f_\ell \circ \beta \colon P \to X$, where $\beta$ is the point-pushing diffeomorphism interchanging the punctures $p_i$ and $p_{i+1}$.}
\label{fig:2dim-case}
\end{figure}
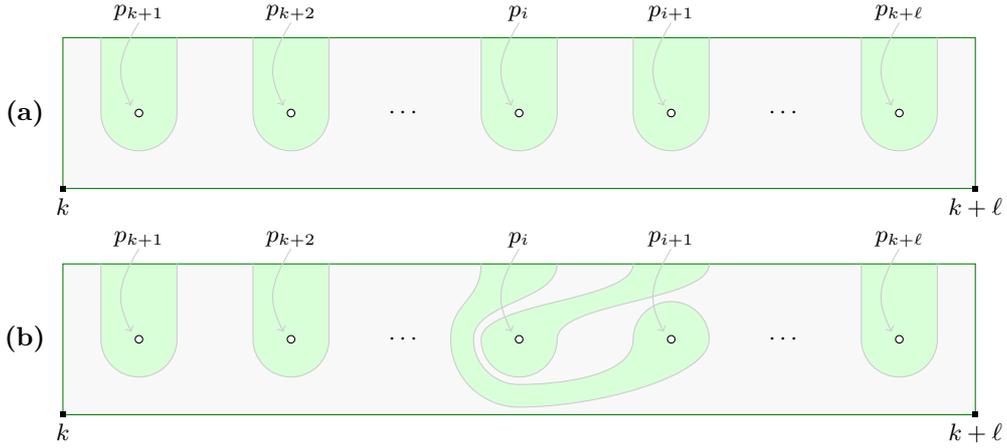

\begin{rmk}
The general setting of \cite{Krannich2019Homologicalstabilitytopological} is for $E_0$-modules over $E_2$-algebras; the theorem stated in the next section (Theorem \ref{t:ths-Krannich}) in terms of the category $\cC(M)$ is a rephrasing of this in a special case. If $d \geq 3$, the fact that the configuration-section spaces on $M$ form an $E_0$-module over an $E_{d-1}$-algebra, hence in particular over an $E_2$-algebra, \emph{automatically} implies that we have an extension to $\cC(M)$. On the other hand, when $d=2$, it is not tautological that we have an extension to $\cC(M)$, although it is still true, by Proposition \ref{p:extension-to-C}.
\end{rmk}

Under certain conditions, the monodromy functor \eqref{eq:extension-to-C} of Proposition \ref{p:extension-to-C} factors through the functor of braid categories $\cC(M) \to \cB_\sharp(M)$ (see Summary \ref{summary:braid-categories}).

\begin{prop}
\label{p:extension-to-Bsharp}
Let $\mathrm{dim}(M) \geq 3$ and assume that either $\pi_1(M)=0$ or that the handle-dimension of $M$ is at most $\mathrm{dim}(M) - 2$. Also assume that $\xi$ is the trivial bundle over $M$ with fibre $X$. Then the monodromy functor \eqref{eq:extension-to-C} factors through the functor of braid categories $\cC(M) \to \cB_\sharp(M)$.
\end{prop}
\begin{proof}
The fibre $\mathrm{Map}^{c,D}(\hat{M}_k \smallsetminus z_k , X)$ of the fibration $\cmapd{k}{c,D}{M}{X} \to \dot{C}_k(M)$ decomposes up to homotopy as $\mathrm{Map}_*(M,X) \times (\Omega_c^{d-1}X)^k$, where $\Omega_c^{d-1}X$ denotes the union of path-components of the loopspace $\Omega^{d-1}X$ corresponding to the subset $c \subseteq [S^{d-1},X]$. Since $M$ has dimension at least $3$, the fundamental group $\pi_1(\dot{C}_k(M))$ decomposes as $\pi_1(M)^k \rtimes \Sigma_k$ (see for example \cite[Theorem~9]{FadellNeuwirth1962Configurationspaces} or \cite[Theorem~1]{Birman1969Onbraidgroups}). Under these identifications, and under the given assumptions on $M$, \cite[Corollary 7.2]{PalmerTillmann2020Pointpushingactions} implies that the monodromy action is given, for $\alpha_i \in \pi_1(M)$, $\sigma \in \Sigma_k$, $f \in \mathrm{Map}_*(M,X)$ and $g_i \in \Omega_c^{d-1}X$, by the following formula:
\begin{equation}
\label{eq:action-formula}
(\alpha_1,\ldots,\alpha_k;\sigma) \cdot (f,g_1,\ldots,g_k) = (f , \bar{g}_1 ,\ldots, \bar{g}_k),
\end{equation}
where $\bar{g}_i = f_*(\alpha_i) . g_{\sigma(i)} . \mathrm{sgn}(\alpha_i)$. The element $f_*(\alpha_i) \in \pi_1(X)$ acts on the point $g_{\sigma(i)} \in \Omega_c^{d-1}X$ via the usual action-up-to-homotopy of $\pi_1(X)$ on the iterated loopspaces of $X$. The sign $\mathrm{sgn}(\alpha_i) \in \{\pm 1\}$ depends on whether or not the loop $\alpha_i$ lifts to a loop in the orientation double cover of $M$, and $-1$ acts on $g_{\sigma(i)} \in \Omega_c^{d-1}X$ via a reflection of $S^{d-1}$.

The formula \eqref{eq:action-formula} may be visualised as follows (\cf Figure 7.1 of \cite{PalmerTillmann2020Pointpushingactions}). The element $(\alpha_1,\ldots,\alpha_k;\sigma)$ of $\pi_1(M)^k \rtimes \Sigma_k$ is a $k$-strand braid, where the strands may pass through each other, and where each strand is labelled by an element of $\pi_1(M)$. The element $(f,g_1,\ldots,g_k)$ is acted upon by pushing the $i$-th element $g_i$ backwards along the $i$-th strand of the braid while acting on it by $f_*(\phantom{-})$ of the label of that strand as well as the involution $\mathrm{sgn}(\phantom{-})$ of the label of the strand.

This description of the monodromy action immediately suggests how to extend the monodromy functor \eqref{eq:extension-to-C} to $\cB_\sharp(M)$, since the morphisms of $\cB_\sharp(M)$ have a similar combinatorial description when $M$ has dimension at least $3$. Namely, a morphism $m \to n$ in $\cB_\sharp(M)$ may be viewed as a braid (whose strands may cross each other) from a subset of $\{1,\ldots,m\}$ to a subset of $\{1,\ldots,n\}$, where each strand is labelled by an element of $\pi_1(M)$ (\cf \cite[Remark 5.10]{Krannich2019Homologicalstabilitytopological}). In fact, we will define an extension of the monodromy functor to $\cB_\sharp(M)^{\mathrm{op}}$, but this will finish the proof since $\cB_\sharp(M)$ is canonically isomorphic to its opposite category.

To define the extension of the monodromy functor to $\cB_\sharp(M)^{\mathrm{op}}$, we describe how a morphism $m \to n$ of $\cB_\sharp(M)$ acts on $(f,g_1,\ldots,g_n)$ for $f \in \mathrm{Map}_*(M,X)$ and $g_i \in \Omega_c^{d-1}X$. First, fix a basepoint $* \in \Omega_c^{d-1}X$. If there is a strand ending at position $i$, we push the $i$-th element $g_i$ backwards along this strand, acting on it by $f_*(\phantom{-})$ and $\mathrm{sgn}(\phantom{-})$ of the strand's label. We then fill in any blanks in the resulting partial $m$-tuple of elements of $\Omega_c^{d-1}X$ with the basepoint $*$.

In formulas, this is written as follows. A morphism $m \to n$ of $\cB_\sharp(M)$ is given by a partially-defined injective function $\sigma$ from a subset of $\{1,\ldots,m\}$ to a subset of $\{1,\ldots,n\}$ and an element $\alpha_i \in \pi_1(M)$ for each $i \in \mathrm{dom}(\sigma)$. This morphism acts by
\[
(f,g_1,\ldots,g_n) \longmapsto (f,\bar{g}_1,\ldots,\bar{g}_m),
\]
where $\bar{g}_i = f_*(\alpha_i).g_{\sigma(i)}.\mathrm{sgn}(\alpha_i)$ if $i \in \mathrm{dom}(\sigma)$ and $\bar{g}_i = *$ otherwise.
\end{proof}

\section{Polynomiality and stability}\label{s:polynomiality}

In this section we complete the proof of our main homological stability result for configuration-section spaces. First, we show that the composition of \eqref{eq:extension-to-C} with the functor $H_i(-;\bK) \colon \hotop \to \mathrm{Vect}_{\bK} \to \mathrm{Ab}$ is ``of degree $\leq i$'' for all $i\geq 0$ and all fields $\bK$. Via a result of \cite{Krannich2019Homologicalstabilitytopological} (recalled below as Theorem \ref{t:ths-Krannich}), this implies twisted homological stability for configuration spaces with coefficients in $H_i(\eqref{eq:extension-to-C};\bK)$, which then implies the desired result by a spectral sequence comparison argument.

\paragraph{Polynomial functors.}

The category $\cC(M)$ has a canonical endofunctor
\[
s \colon \cC(M) \too \cC(M)
\]
defined as follows. The maps $\sigma_k \colon \pi_1(\dot{C}_k(M)) \to \pi_1(\dot{C}_{k+1}(M))$ induced by the stabilisation maps (\cf Definition \ref{def:sigmak-and-vkl}) induce an endofunctor of $\mathrm{Br}(M)$ given by $k \mapsto k+1$ on objects. This endofunctor is compatible with the left-action of $\cB$ on $\mathrm{Br}(M)$, so it induces an endofunctor of $\cC(M) = \langle \cB , \mathrm{Br}(M) \rangle$, which we denote by $s$. There is moreover a natural transformation
\[
\iota \colon \mathrm{id}_{\cC(M)} \too s
\]
given by the morphisms $\iota_k = (1,\mathrm{id}_{k+1}) \colon k \to k+1$ of $\cC(M)$ for $k \in \bN$. (Recall that a morphism $a \to b$ in $\cC(M)$ is determined by an object $c$ of $\cB$ and a morphism $a+c \to b$ of $\mathrm{Br}(M)$.) We note that $s(\iota_k) = (1,v_k^2(\tau_1)) = v_k^2(\tau_1) \circ \iota_{k+1}$, where $\tau_1 \in B_2$ is the standard generator and $v_k^2 \colon B_2 \to \pi_1(\dot{C}_{k+2}(M))$ is the homomorphism defined in Definition \ref{def:sigmak-and-vkl}.

\begin{defn}
\label{d:degree}
For a category $\cC$ equipped with an endofunctor $s\colon \cC \to \cC$ and natural transformation $\iota \colon \mathrm{id}_\cC \to s$, and an abelian category $\cA$, the \emph{degree} of a functor $T \colon \cC \to \cA$ takes values in $\{-1\} \cup \bN \cup \{\infty\}$ and is defined recursively as follows. The only functor $T$ of degree $-1$ is $T=0$. If functors of degree $\leq d$ have been defined, we say that $T$ has degree $d+1$ if and only if the natural transformation $T\iota \colon T \to Ts$ is split injective in the functor category $\mathrm{Fun}(\cC,\cA)$ and the functor
\[
\Delta T = \mathrm{coker}(T\iota \colon T \to Ts) \colon \cC \to \cA
\]
has degree $d$. Once all functors of finite degree have been defined, all remaining functors are said to have degree $\infty$.
\end{defn}

\begin{rmk}
For $\cC = \cC(M)$, equipped with the endofunctor and natural transformation described above, this corresponds to the notion of ``\emph{split degree at $0$}'' of \cite[Definition 4.6]{Krannich2019Homologicalstabilitytopological}. There are an analogous endofunctor $s$ and natural transformation $\iota$ on the category $\cB_\sharp(M)$ (which commute with the functor $\cC(M) \to \cB_\sharp(M)$ from Summary \ref{summary:braid-categories}), and when $\cC = \cB_\sharp(M)$ equipped with these, the definition above corresponds to the \emph{degree} of \cite[Definition 3.1]{Palmer2018Twistedhomologicalstability}. In this setting, the degree of $T$ has an alternative characterisation in terms of \emph{cross-effects} of $T$ (Definition 3.15 and Lemma 3.16 of \cite{Palmer2018Twistedhomologicalstability}). See also \cite{Palmer2017comparisontwistedcoefficient} for a more general overview of notions of \emph{degree} of a functor via recursion (as above) or via cross-effects.
\end{rmk}

\begin{lem}
\label{lem:direct-sum-tensor-product}
Let $\cC$ be a category as in Definition \ref{d:degree}, $\cA$ an abelian category, $T_1$ and $T_2 \colon \cC \to \cA$ functors and $A \in \mathrm{ob}(\cA)$. Then we have
\begin{itemizeb}
\item $\mathrm{deg}(T_1 \oplus T_2) = \mathrm{max}\{ \mathrm{deg}(T_1) , \mathrm{deg}(T_2) \}$,
\item $\mathrm{deg}(T_1 \otimes A) \leq \mathrm{deg}(T_1)$, and, more generally,
\item $\mathrm{deg}(T_1 \otimes T_2) \leq \mathrm{deg}(T_1) + \mathrm{deg}(T_2)$ whenever $\mathrm{deg}(T_1)$ and $\mathrm{deg}(T_2)$ are non-negative,
\end{itemizeb}
where we assume that $(\cA,\otimes)$ is an abelian monoidal category for the second and third points.
\end{lem}
\begin{proof}
This is a direct generalisation of \cite[Lemma 3.18]{Palmer2018Twistedhomologicalstability}, whose proof also generalises directly.
\end{proof}

\begin{defn}
Let $\cC$ be as in Definition \ref{d:degree}. We say that a functor
\[
T \colon \cC \longrightarrow \mathrm{Ho}(\mathrm{Top})
\]
has \emph{slope $\leq\lambda$}, for $\lambda \in (0,\infty)$, if for every field $\bK$ and for each integer $i\geq 0$, the composite functor
\[
H_i(-;\bK) \circ T \colon \cC \longrightarrow \mathrm{Vect}_\bK
\]
has degree $\leq \lambda i$ in the sense of Definition \ref{d:degree}.
\end{defn}

Our main homological stability result is the following. Suppose that we have chosen a bundle over a manifold $\xi \colon E \to M$, a disc $D \subseteq \partial M$, a singularity condition $c \subseteq \pi_0(\Sigma(\xi))$, etc, as described in Notation \ref{notation-inputs}. There are quotient maps
\[
\pi_{d-1}(X) \longtwoheadrightarrow [S^{d-1},X] \cong \pi_0(\Sigma(\xi|_D)) \longtwoheadrightarrow \pi_0(\Sigma(\xi)).
\]
The left-hand quotient is given by the identification of $[S^{d-1},X]$ with the orbits $\pi_{d-1}(X)/\pi_1(X)$ of the natural $\pi_1(X)$-action on $\pi_{d-1}(X)$, and the right-hand quotient was explained in Definition \ref{d:cgammadot}. Taking pre-images, the singularity condition $c$ therefore determines subsets $c_D \subseteq [S^{d-1},X]$ and $\widetilde{c}_D \subseteq \pi_{d-1}(X)$.

\begin{thm}
\label{thm-hom-stab}
Suppose that the subset $\widetilde{c}_D \subseteq \pi_{d-1}(X)$ has size $1$. Then the stabilisation maps
\[
\cgamma{k}{c,D}{M}{\xi} \too \cgamma{k+1}{c,D}{M}{\xi}
\]
induce isomorphisms on integral homology up to degree $\tfrac{k}{2} - 2$ and surjections up to degree $\tfrac{k}{2} - 1$. With coefficients in a field, both of these ranges may be improved by one.
\end{thm}

This is Theorem \ref{tmain} of the introduction. The additional claim about improving the slope when $2$ is invertible is contained in the following:

\begin{thm}
\label{thm-hom-stab-improved}
Suppose that the subset $\widetilde{c}_D \subseteq \pi_{d-1}(X)$ has size $1$. Suppose also that $\mathrm{dim}(M) \geq 3$ and that either $\pi_1(M) = 0$ or the handle-dimension of $M$ is at most $\mathrm{dim}(M)-2$. Also suppose that $\xi$ is a trivial bundle. Then the stabilisation maps
\[
\cgamma{k}{c,D}{M}{\xi} \too \cgamma{k+1}{c,D}{M}{\xi}
\]
induce isomorphisms on homology with $\bZ[\tfrac12]$ coefficients up to degree $k-2$ and surjections up to degree $k-1$. With coefficients in a field of characteristic $\neq 2$, both of these ranges may be improved by one.
\end{thm}

The main technical input for these theorems is the following.

\begin{prop}
\label{prop:polynomiality}
If $\lvert \widetilde{c}_D \rvert = 1$, the functor $\eqref{eq:extension-to-C} = \smash{\widetilde{\mathrm{Mon}}}^{c,D}(M,\xi)$ has slope $\leq 1$.
\end{prop}

We will also use part of Theorem D of \cite{Krannich2019Homologicalstabilitytopological}, which we recall in the following:

\begin{thm}[{\cite[part of Theorem D]{Krannich2019Homologicalstabilitytopological}}]
\label{t:ths-Krannich}
Let $(M,D,b,*)$ be as in Proposition \ref{p:extension-criterion} and let
\[
G \colon \cC(M) \too \mathrm{Ab}
\]
be a functor to the category of abelian groups. If $\mathrm{deg}(G) \leq r$, then the maps
\[
H_i(\dot{C}_k(M);G(k)) \too H_i(\dot{C}_{k+1}(M);G(k+1)),
\]
induced by the stabilisation maps \eqref{eq:cgamma-stabilisation} together with the functor $G$, are isomorphisms in the range of degrees $2i \leq k-r-2$ and surjections in the range of degrees $2i \leq k-r$.
\end{thm}

\begin{proof}[Proof of Theorem \ref{thm-hom-stab} (homological stability for configuration-section spaces).]
From the fibration sequences \eqref{eq:Hurewicz-fibration-sequence} and the stabilisation maps \eqref{eq:cgamma-stabilisation}, we have a map of fibration sequences of the form:
\begin{equation}
\label{eq:map-of-key-fibrations}
\centering
\begin{split}
\begin{tikzpicture}
[x=1mm,y=1mm]
\node (tl) at (0,30) {$\Gamma_k^{c,D}(M,\xi)$};
\node (tr) at (40,30) {$\Gamma_{k+1}^{c,D}(M,\xi)$};
\node (ml) at (0,15) {$\cgammad{k}{c,D}{M}{\xi}$};
\node (mr) at (40,15) {$\cgammad{k+1}{c,D}{M}{\xi}$};
\node (bl) at (0,0) {$\dot{C}_k(M)$};
\node (br) at (40,0) {$\dot{C}_{k+1}(M),$};
\draw[->] (tl) to (tr);
\draw[->] (ml) to (mr);
\draw[->] (bl) to (br);
\draw[->] (tl) to (ml);
\draw[->] (ml) to (bl);
\draw[->] (tr) to (mr);
\draw[->] (mr) to (br);
\end{tikzpicture}
\end{split}
\end{equation}
which has an associated map of Serre spectral sequences
\begin{equation}
\label{eq:map-of-SSS}
\centering
\begin{split}
\begin{tikzpicture}
[x=1mm,y=1mm]
\node (tl) at (0,15) {$H_p(\dot{C}_k(M) ; H_q( \Gamma_k^{c,D}(M,\xi) ; \bK ))$};
\node (tr) at (70,15) {$H_p(\dot{C}_{k+1}(M) ; H_q( \Gamma_{k+1}^{c,D}(M,\xi) ; \bK ))$};
\node (bl) at (0,0) {$H_{p+q}(\cgammad{k}{c,D}{M}{\xi} ; \bK)$};
\node (br) at (70,0) {$H_{p+q}(\cgammad{k+1}{c,D}{M}{\xi} ; \bK).$};
\node at (0,7.5) {\rotatebox{270}{$\Rightarrow$}};
\node at (70,7.5) {\rotatebox{270}{$\Rightarrow$}};
\draw[->] (tl) to (tr);
\draw[->] (bl) to (br);
\end{tikzpicture}
\end{split}
\end{equation}
By Proposition \ref{prop:polynomiality}, for any field $\bK$, the functor $H_q(-;\bK) \circ \smash{\widetilde{\mathrm{Mon}}}^{c,D}(M,\xi)$ has degree $\leq q$ for each $q \geq 0$, and hence Theorem \ref{t:ths-Krannich} implies that \eqref{eq:map-of-SSS} is an isomorphism on $E^2$ pages in the range of bidegrees $2p \leq k-q-2$, and a surjection for $2p \leq k-q$. In particular, it is an isomorphism for total degree $p+q \leq \tfrac{k}{2} - 1$ and a surjection for $p+q \leq \tfrac{k}{2}$. By a spectral sequence comparison argument (see \cite[Theorem 1]{Zeeman1957proofofcomparison} or \cite[Remarque 2.10]{CollinetDjamentGriffin2013Stabilitehomologiquepour}), the same statements hold also in the limit. Composing with the homotopy equivalences of Lemma \ref{l:equivalent-models}, we conclude that the stabilisation maps
\begin{equation}
\label{eq:stab-maps-with-identifications}
\cgamma{k}{c,D}{M}{\xi} \simeq \cgammad{k}{c,D}{M}{\xi} \too \cgammad{k+1}{c,D}{M}{\xi} \simeq \cgamma{k+1}{c,D}{M}{\xi}
\end{equation}
induce isomorphisms on $\bK$-homology up to degree $\tfrac{k}{2}-1$ and surjections up to degree $\tfrac{k}{2}$. Applying this for $\bK = \bF_p$ and using the maps of long exact sequences induced by the short exact sequences of coefficients $0 \to \bZ/p^r \to \bZ/p^{r+1} \to \bZ/p \to 0$, we deduce the same statements for homology with coefficients in $\bZ(p^\infty)$, where $\bZ(p^\infty)$ is the direct limit of $\bZ/p \to \bZ/p^2 \to \bZ/p^3 \to \cdots$. Then using the short exact sequence of coefficients
\[
0 \to \bZ \too \bQ \too \bQ/\bZ = \bigoplus_p \bZ(p^\infty) \to 0,
\]
we conclude that the stabilisation maps \eqref{eq:stab-maps-with-identifications} induce isomorphisms on integral homology up to degree $\tfrac{k}{2}-2$ and surjections up to degree $\tfrac{k}{2}-1$.
\end{proof}

\begin{proof}[Proof of Theorem \ref{thm-hom-stab-improved} (the improvement when $2$ is invertible).]
Under the hypotheses of the theorem, Proposition \ref{p:extension-to-Bsharp} implies that the monodromy functor \eqref{eq:extension-to-C} factors through the functor of braid categories $\cC(M) \to \cB_\sharp(M)$, and therefore defines a (polynomial) twisted coefficient system on $\cB_\sharp(M)$. This allows us to apply \cite[Theorem A]{Palmer2018Twistedhomologicalstability} instead of \cite[Theorem D]{Krannich2019Homologicalstabilitytopological} (Theorem \ref{t:ths-Krannich}) to deduce twisted homological stability for the configuration spaces $C_k(M)$ with coefficients in this functor. The advantage of \cite{Palmer2018Twistedhomologicalstability} is that its proof shows that if, for a particular manifold $M$ and ring $R$, one knows untwisted homological stability for $C_k(M')$ with coefficients in $R$ with a certain slope $\ell$, where $M'$ is $M$ minus any finite number of points, then it proves that twisted homological stability holds for $C_k(M)$ for any polynomial twisted coefficient system $\cB_\sharp(M) \to \mathrm{Mod}_R$, \emph{with the same slope $\ell$}. Thus, we may apply \cite[Theorem 1.4]{KupersMiller2015} (and the universal coefficient theorem), which implies that, if $\bK$ is a field of characteristic $\neq 2$, then the configuration spaces $C_k(M')$ are homologically stable with slope $1$ for coefficients in $\bK$. (This uses the assumption that $\mathrm{dim}(M) \geq 3$ for a second time.) By \cite{Palmer2018Twistedhomologicalstability}, the configuration spaces $C_k(M)$ are homologically stable with slope $1$ for any polynomial twisted coefficient system $\cB_\sharp(M) \to \mathrm{Mod}_{\bK}$ (\cf \cite[Remark 6.5]{Palmer2018Twistedhomologicalstability}), and the spectral sequence comparison argument of the proof of Theorem \ref{thm-hom-stab} then implies that the stabilisation maps \eqref{eq:stab-maps-with-identifications} induce isomorphisms on $\bK$-homology up to degree $k-1$ and surjections up to degree $k$. Finally, to deduce that the stabilisation maps \eqref{eq:stab-maps-with-identifications} induce isomorphisms on $\bZ[\tfrac12]$ homology up to degree $k-2$ and surjections up to degree $k-1$, we use the fact that $\bQ / \bZ[\tfrac12] \cong \bigoplus_{p \neq 2} \bZ(p^{\infty})$ and the short exact sequences of coefficient groups from the end of the proof of Theorem \ref{thm-hom-stab}.
\end{proof}

\begin{proof}[Proof of Proposition \ref{prop:polynomiality}.]
Let $\bK$ be a field and $i\geq 0$ an integer. Write
\[
F_i = H_i(-;\bK) \circ \smash{\widetilde{\mathrm{Mon}}}^{c,D}(M,\xi) \colon \cC(M) \too \mathrm{Vect}_\bK .
\]
In this notation, we need to show that $\mathrm{deg}(F_i) \leq i$, where $\mathrm{deg}(-)$ is as in Definition \ref{d:degree}.

Recall that we have assumed that the subset $\widetilde{c}_D \subseteq \pi_{d-1}(X)$ has cardinality $1$, and denote by $Z$ the corresponding path-component of $\Omega^{d-1}X$. Also write $Y = \Gamma^D(M,\xi)$ for the space of sections of $\xi \colon E \to M$ that restrict to the fixed section $s_D$ (\cf Notation \ref{notation-inputs}) on $D \subseteq \partial M$. There are natural maps
\begin{equation}
\label{eq:decomposition-of-Fk}
e_k \colon Y \times Z^k \too \Gamma_k^{c,D}(M,\xi) = \smash{\widetilde{\mathrm{Mon}}}^{c,D}(M,\xi)(k),
\end{equation}
defined in Figure \ref{fig:ik}, such that the square
\begin{center}
\begin{tikzpicture}
[x=1mm,y=1mm]
\node (tl) at (0,15) {$Y \times Z^k$};
\node (tr) at (30,15) {$Y \times Z^{k+1}$};
\node (bl) at (0,0) {$\Gamma_k^{c,D}(M,\xi)$};
\node (br) at (30,0) {$\Gamma_{k+1}^{c,D}(M,\xi)$};
\draw[->] (tl) to (tr);
\draw[->] (bl) to (br);
\draw[->] (tl) to node[left,font=\small]{$e_k$} (bl);
\draw[->] (tr) to node[right,font=\small]{$e_{k+1}$} (br);
\end{tikzpicture}
\end{center}
commutes up to homotopy, where the bottom horizontal map is the stabilisation map (namely the top horizontal map of \eqref{eq:map-of-key-fibrations}, which is also $\smash{\widetilde{\mathrm{Mon}}}^{c,D}(M,\xi)(\iota_k)$) and the top horizontal map is the obvious inclusion $(s,f_1,\ldots,f_k) \mapsto (s,f_1,\ldots,f_k,*)$, where $* \in Z$ is any basepoint (exactly which basepoint does not matter since $Z$ is path-connected). Moreover, the map \eqref{eq:decomposition-of-Fk} is a topological embedding, and it is not hard to define a deformation retraction of $\Gamma_k^{c,D}(M,\xi)$ onto its image -- hence \eqref{eq:decomposition-of-Fk} is a homotopy equivalence.

Now consider an automorphism $\alpha \in \mathrm{Aut}_{\cC(M)}(k) = \pi_1(\dot{C}_k(M))$, which, via the endofunctor $s$ of $\cC(M)$, induces an automorphism $s(\alpha) = \sigma_k(\alpha) \in \pi_1(\dot{C}_{k+1}(M))$. (Recall that the notation $\sigma_k(-)$ was introduced in Definition \ref{def:sigmak-and-vkl}.) One may check that, under the identifications \eqref{eq:decomposition-of-Fk}, we have an equality
\begin{equation}
\label{eq:decomposition-alpha}
\smash{\widetilde{\mathrm{Mon}}}^{c,D}(M,\xi)(s(\alpha)) = \smash{\widetilde{\mathrm{Mon}}}^{c,D}(M,\xi)(\alpha) \times \mathrm{id}_Z
\end{equation}
in the group of homotopy automorphisms up to homotopy $\pi_0(\mathrm{hAut}(Y \times Z^{k+1}))$.

Next, consider the morphism $\iota_k \colon k \to k+1$ of $\cC(M)$. As observed earlier in this section, we have the identity $s(\iota_k) = v_k^2(\tau_1) \circ \iota_{k+1}$, where $\tau_1 \in B_2$ is the standard generator and the homomorphism $v_k^2 \colon B_2 \to \pi_1(\dot{C}_{k+2}(M))$ is as in Definition \ref{def:sigmak-and-vkl}. We also noted above that, under the identifications \eqref{eq:decomposition-of-Fk}, the map $\smash{\widetilde{\mathrm{Mon}}}^{c,D}(M,\xi)(\iota_k)$ corresponds to the obvious inclusion of $Y \times Z^k$ into $Y \times Z^{k+1}$ using the fixed basepoint $*$ of $Z$. Using this, one may check that, under the identifications \eqref{eq:decomposition-of-Fk}, we have an equality
\begin{equation}
\label{eq:decomposition-iota}
\smash{\widetilde{\mathrm{Mon}}}^{c,D}(M,\xi)(s(\iota_k)) = \smash{\widetilde{\mathrm{Mon}}}^{c,D}(M,\xi)(\iota_k) \times c_Z
\end{equation}
in the homotopy set $\pi_0(\mathrm{Map}(Y \times Z^{k+1} , Y \times Z^{k+2})) = [Y \times Z^{k+1} , Y \times Z^{k+2}]$, where $c_Z \colon Z \to Z$ is the identity map if $d \geq 3$, and if $d=2$ it is the homotopy automorphism of $Z \subseteq \Omega X$ given by conjugating a given loop in the path-component $Z$ of $\Omega X$ by the fixed loop $* \in Z$.

The morphisms of $\cC(M)$ are generated by its automorphisms together with the morphisms $\iota_k$ for $k \in \bN$, so the identifications \eqref{eq:decomposition-of-Fk}, \eqref{eq:decomposition-alpha} and \eqref{eq:decomposition-iota} imply that we have a natural isomorphism
\begin{equation}
\label{eq:decomposition-of-functors}
\smash{\widetilde{\mathrm{Mon}}}^{c,D}(M,\xi) \circ s \,\cong\, \smash{\widetilde{\mathrm{Mon}}}^{c,D}(M,\xi) \times \mathrm{C}(c_Z)
\end{equation}
of functors $\cC(M) \to \hotop$, where, for an endomorphism $f \colon A \to A$ in $\hotop$, the functor $\mathrm{C}(f) \colon \cC(M) \to \hotop$ sends each object to $A$, each automorphism to $\mathrm{id}_A$ and each morphism $\iota_k$ to $f$. Applying $H_i(-;\bK)$ and using the K{\"u}nneth theorem (including the \emph{naturality} of the K{\"u}nneth isomorphism), the decomposition \eqref{eq:decomposition-of-functors} induces an isomorphism of functors
\begin{equation}
\label{eq:decomposition-of-functors2}
F_i \circ s \,\cong\, \bigoplus_{j=0}^i F_{i-j} \otimes H_j(\mathrm{C}(c_Z);\bK),
\end{equation}
such that the natural transformation $F_i \,\iota \colon F_i \to F_i \circ s$ corresponds, under \eqref{eq:decomposition-of-functors2}, to the inclusion of the $j=0$ summand, where we are using the fact that $Z$ is path-connected to identify $H_0(\mathrm{C}(c_Z);\bK)$ with the constant functor at $\bK$.

Using \eqref{eq:decomposition-of-functors2} and the fact that $F_i \,\iota$ corresponds to the inclusion of the $j=0$ summand under this identification, we deduce (i) that $F_i \,\iota$ is split-injective in the functor category $\mathrm{Fun}(\cC(M) , \mathrm{Vect}_\bK)$ and (ii) that we have an isomorphism of functors
\begin{equation}
\label{eq:decomposition-of-functors3}
\Delta F_i \,\cong\, \bigoplus_{j=1}^i F_{i-j} \otimes H_j(\mathrm{C}(c_Z);\bK).
\end{equation}
The fact that $c_Z$ is a homotopy automorphism, i.e., invertible in $\hotop$, implies that the functor $H_j(\mathrm{C}(c_Z);\bK)$ sends each $\iota_k$ to an isomorphism, which implies, by definition, that
\begin{equation}
\label{eq:degree-of-cZ}
\mathrm{deg}(H_j(\mathrm{C}(c_Z);\bK)) \leq 0.
\end{equation}

We now prove that $\mathrm{deg}(F_i) \leq i$ by induction on $i\geq 0$. For $i=0$, the identification \eqref{eq:decomposition-of-functors3} says that $\Delta F_0 = 0$. Together with the fact that $F_0 \,\iota$ is split-injective, this implies that $\mathrm{deg}(F_0) \leq 0$. For $i\geq 1$, using the identification \eqref{eq:decomposition-of-functors3}, Lemma \ref{lem:direct-sum-tensor-product}, the fact \eqref{eq:degree-of-cZ} and the inductive hypothesis, we see that
\[
\mathrm{deg}(\Delta F_i) = \mathrm{deg} \Bigl( \bigoplus_{j=1}^i F_{i-j} \otimes H_j(\mathrm{C}(c_Z);\bK) \Bigr) \leq \mathrm{max}_{j=1}^i \{ \mathrm{deg}(F_{i-j}) \} \leq i-1.
\]
Together with the fact that $F_i \,\iota$ is split-injective, this implies by definition that $\mathrm{deg}(F_i) \leq i$.
\end{proof}

\begin{figure}[t]
\centering
\begin{tikzpicture}
[x=1mm,y=1mm,font=\small]

\fill[yellow!10] (-20,-10) rectangle (0,30);
\fill[black!3] (0,0) rectangle (90,20);
\fill[red!10] (4,0) -- (7,10) arc (180:0:3) -- (16,0) -- cycle;
\fill[green!15] (4,0) -- (7,10) arc (180:0:3) -- (16,0) -- (16,20) -- (4,20) -- cycle;
\fill[red!10] (24,0) -- (27,10) arc (180:0:3) -- (36,0) -- cycle;
\fill[green!15] (24,0) -- (27,10) arc (180:0:3) -- (36,0) -- (36,20) -- (24,20) -- cycle;
\fill[red!10] (74,0) -- (77,10) arc (180:0:3) -- (86,0) -- cycle;
\fill[green!15] (74,0) -- (77,10) arc (180:0:3) -- (86,0) -- (86,20) -- (74,20) -- cycle;

\draw[black!10] (4,0)--(4,20);
\draw[black!10] (16,0)--(16,20);
\draw[black!10] (24,0)--(24,20);
\draw[black!10] (36,0)--(36,20);
\draw[black!10] (74,0)--(74,20);
\draw[black!10] (86,0)--(86,20);

\draw[green!50!black] (0,0) rectangle (90,20);
\draw[green!50!black] (0,0) -- (0,-10);
\draw[green!50!black,densely dashed] (0,-10) -- (-20,-10) -- (-20,30) -- (0,30);
\draw[green!50!black] (0,30) -- (0,20);
\node at (0,0) [fill,inner sep=1pt] {};
\node at (0,0) [anchor=north west,font=\small] {$0$};
\node at (90,0) [fill,inner sep=1pt] {};
\node at (90,0) [anchor=north] {$k$};
\node at (-10,10) [font=\normalsize] {$M$};
\node at (10,10) [fill,white,circle,inner sep=1pt] {};
\node at (30,10) [fill,white,circle,inner sep=1pt] {};
\node at (80,10) [fill,white,circle,inner sep=1pt] {};
\node at (10,10) [draw,circle,inner sep=1pt] {};
\node at (30,10) [draw,circle,inner sep=1pt] {};
\node at (80,10) [draw,circle,inner sep=1pt] {};
\node at (55,10) {$\cdots$};
\draw[decorate,decoration={brace,amplitude=3pt,mirror}] (91,0) -- (91,20);
\node at (92,10) [anchor=west,font=\normalsize] {$D$};

\draw[red] (4,0) -- (7,10) arc (180:0:3) -- (16,0);
\draw[red] (24,0) -- (27,10) arc (180:0:3) -- (36,0);
\draw[red] (74,0) -- (77,10) arc (180:0:3) -- (86,0);

\end{tikzpicture}
\caption{The map $e_k \colon Y \times Z^k \too \Gamma_k^{c,D}(M,\xi)$ from the proof of Proposition \ref{prop:polynomiality}, where $Z$ is a given path-component of $\Omega^{d-1}X = \mathrm{Map}((D^{d-1},\partial D^{d-1}),(X,\{x_0\}))$ and $Y = \Gamma^D(M,\xi)$. \\
Given inputs $(s,f_1,\ldots,f_k)$, the section $e_k(s,f_1,\ldots,f_k)$ of $\hat{\xi}$ over $\hat{M}_k \smallsetminus z_k$ is given by $s$ in the yellow region (namely $M$) and by the maps $f_1,\ldots,f_k$ on the red arcs (representing embedded $(d-1)$-discs with their boundary on the bottom face of $D \times [0,k]$). Recall that, over $(D \times [0,k]) \smallsetminus z_k$, the bundle is trivial with fibre $X$, so we may think of sections as maps $(D \times [0,k]) \smallsetminus z_k \to X$. The bottom face of $D \times [0,k]$ is sent to the basepoint $x_0$ of $X$. We then extend the map in the red regions by defining it to be constant along radii centred at the punctures $z_k$, and we extend it in the green regions by defining it to be constant in the vertical direction.}
\label{fig:ik}
\end{figure}

\begin{rmk}[\emph{The hypothesis of Proposition \ref{prop:polynomiality}.}]
The assumption $\lvert \widetilde{c}_D \rvert = 1$ (where $\widetilde{c}_D \subseteq \pi_{d-1}(X)$ is the subset induced by the ``singularity condition'' $c \subseteq \pi_0(\Sigma(\xi))$) of Proposition \ref{prop:polynomiality} means that the corresponding subspace $Z \subseteq \Omega^{d-1}X$ is a single path-component, rather than a union of several path-components. The path-connectedness of $Z$ is used in a key way for the identification \eqref{eq:decomposition-of-functors3} of $\Delta F_i$ in terms of $F_{i-1}, F_{i-2}, \ldots$. If $Z$ were disconnected, we would have $H_0(\mathrm{C}(c_Z);\bK) = \bK^{\pi_0(Z)}$ and, denoting by $Z_0$ the path-component of $Z$ containing its basepoint, the identification \eqref{eq:decomposition-of-functors3} would become
\[
\Delta F_i \,\cong\, F_i \otimes \bK^{\pi_0(Z) \smallsetminus \{Z_0\}} \oplus \bigoplus_{j=1}^i F_{i-j} \otimes H_j(\mathrm{C}(c_Z);\bK),
\]
which would break the inductive argument, since this decomposition involves $F_i$ itself.
\end{rmk}

\begin{rmk}[\emph{Naturality of the K{\"u}nneth theorem.}]
\label{rmk:naturality}
In order to transform the identification \eqref{eq:decomposition-of-functors} of functors $\cC(M) \to \hotop$ into the identification \eqref{eq:decomposition-of-functors2} of functors $\cC(M) \to \mathrm{Vect}_\bK$, we used the K{\"u}nneth theorem for field coefficients, which does not involve any Tor terms. If we had worked instead with $\bZ$ coefficients, we would have obtained a decomposition similar to \eqref{eq:decomposition-of-functors2} -- \emph{at the level of objects} -- including also some Tor terms. The appearance of the Tor terms themselves is no problem, since one may prove an analogue of the second and third points of Lemma \ref{lem:direct-sum-tensor-product} for $\mathrm{Tor}(-)$ instead of $\otimes$, so they behave as desired with respect to degree. The problem instead is that the K{\"u}nneth short exact sequences are split, but not \emph{naturally} split (unless -- of course -- the Tor terms vanish). Thus, we would not have been able to obtain a decomposition \emph{of functors} analogous to \eqref{eq:decomposition-of-functors2}, since the naturality of the K{\"u}nneth theorem, in the case where the Tor terms vanish, was key to upgrading \eqref{eq:decomposition-of-functors2} from an isomorphism at the level of objects to an isomorphism of functors. See also Remark 4.4 of \cite{Palmer2018Twistedhomologicalstability} for a similar comment about the (non-)naturality of the splitting of the K{\"u}nneth short exact sequence.
\end{rmk}

\begin{rmk}[\emph{Improving the range of stability.}]
Our main homological stability result states that, on homology with field coefficients, the stabilisation maps \eqref{eq:stab-maps-with-identifications} induce isomorphisms up to degree $\tfrac{k}{2} - 1$ and surjections up to degree $\tfrac{k}{2}$. This then implies the analogous statements for homology with integral coefficients -- but only in a range of degrees that is smaller by one, i.e., isomorphisms up to degree $\tfrac{k}{2} - 2$ and surjections up to degree $\tfrac{k}{2} - 1$.

However, under certain hypotheses, this loss of one in the range of degrees for integral homology may be avoided: namely, if the integral homology groups $H_*(\Omega^{d-1}X;\bZ)$ and $H_*(\Gamma^D(M,\xi);\bZ)$ are torsion-free in all degrees. Under this assumption, one may run the proof of Proposition \ref{prop:polynomiality} using $\bZ$ in place of $\bK$, since the torsion-freeness assumption implies that one may apply the K{\"u}nneth theorem for $\bZ$ coefficients without the appearance of any Tor terms (see Remark \ref{rmk:naturality} above for the importance of the vanishing of the Tor terms, and see Remark 4.4 of \cite{Palmer2018Twistedhomologicalstability} for the analogous remark in a similar setting). We could then also run the proof of Theorem \ref{tmain} directly with $\bZ$ coefficients, without any need to pass from field coefficients to $\bZ$ coefficients at the end, which is where we lose $1$ from the range of stability.
\end{rmk}

Finally, we prove a split-injectivity result for the homology of configuration-mapping spaces, under certain conditions on the underlying manifold $M$. Let us fix a connected manifold $M$, an embedded disc $D \subseteq \partial M$, a based space $X$ and a subset $c \subseteq [S^{d-1},X]$. Recall that the stabilisation maps $\cmap{k}{c,D}{M}{X} \to \cmap{k+1}{c,D}{M}{X}$ fit into a map of fibre sequences of the form \eqref{eq:map-of-key-fibrations}, inducing a map of their associated Serre spectral sequences:
\begin{equation}
\label{eq:map-of-SSS-repeated}
\centering
\begin{split}
\begin{tikzpicture}
[x=1mm,y=1mm,font=\small]
\node (tl) at (0,15) {$H_p(\dot{C}_k(M) ; H_q( \mathrm{Map}^{c,D}(\hat{M}_k \smallsetminus z_k , X) ; \bZ ))$};
\node (tr) at (72,15) {$H_p(\dot{C}_{k+1}(M) ; H_q( \mathrm{Map}^{c,D}(\hat{M}_{k+1} \smallsetminus z_{k+1} , X) ; \bZ ))$};
\node (bl) at (0,0) {$H_{p+q}(\cmap{k}{c,D}{M}{X} ; \bZ)$};
\node (br) at (72,0) {$H_{p+q}(\cmap{k+1}{c,D}{M}{X} ; \bZ)$};
\node at (0,7.5) {\rotatebox{270}{$\Rightarrow$}};
\node at (72,7.5) {\rotatebox{270}{$\Rightarrow$}};
\draw[->] (tl) to (tr);
\draw[->] (bl) to (br);
\end{tikzpicture}
\end{split}
\end{equation}
(Note that above we take coefficients in $\bZ$ rather than in a field, as was the case in \eqref{eq:map-of-SSS}.)

Our splitting result is then the following.

\begin{thm}
\label{thm:split-injectivity}
Let $M$ have dimension at least $3$ and assume that either $\pi_1(M)=0$ or the handle-dimension of $M$ is at most $\mathrm{dim}(M) - 2$. Then the map of Serre spectral sequences \eqref{eq:map-of-SSS-repeated} above is split-injective on each entry of the $E^2$ pages.
\end{thm}

\begin{proof}
By \cite[Theorem A and Remark 1.3]{Palmer2018Twistedhomologicalstability}, for any abelian group-valued functor $T \colon \cB_\sharp(M) \to \mathrm{Ab}$, the induced stabilisation maps on $T$-twisted homology
\[
H_*(C_k(M);T(k)) \too H_*(C_{k+1}(M);T(k+1))
\]
are split-injective in all degrees. This implies the statement of Theorem \ref{thm:split-injectivity} by applying it to the extended monodromy functor $\cB_\sharp(M) \to \hotop$ of Proposition \ref{p:extension-to-Bsharp} composed with the homology functor $H_q(-;\bZ) \colon \hotop \to \mathrm{Ab}$.
\end{proof}


\phantomsection
\addcontentsline{toc}{section}{References}
\renewcommand{\bibfont}{\normalfont\small}
\setlength{\bibitemsep}{0pt}
\printbibliography

\vspace{0pt plus 1filll}

\noindent {\itshape\small Institutul de Matematică Simion Stoilow al Academiei Române, 21 Calea Griviței, Bucharest, Romania}

\noindent {\itshape\small Mathematical Institute, University of Oxford, Andrew Wiles Building, Oxford, OX2 6GG, UK}

\noindent {\tt mpanghel@imar.ro}

\noindent {\tt tillmann@maths.ox.ac.uk}

\end{document}